\documentclass[11pt, a4paper]{amsart}
\usepackage[utf8]{inputenc}
\usepackage[usenames,dvipsnames]{xcolor}
\usepackage{amsmath,amsthm,amsfonts}
\usepackage{verbatim}
\usepackage{amssymb}
\usepackage{array}  
\usepackage{graphicx} 
\usepackage{color}
\usepackage{mathrsfs}
\usepackage{graphicx}
\usepackage{dsfont}
\definecolor{orangebis}{rgb}{0.99,0.25,0.00}
\definecolor{greenbis}{rgb}{0.10,0.85,0.10}
\definecolor{bluebis}{rgb}{0.10,0.30,0.99}
\usepackage[final]{hyperref} 
\usepackage{subcaption}
\usepackage[foot]{amsaddr}

\usepackage{tikz-cd}

\usepackage{multicol}

\usepackage{lipsum}
\usepackage{geometry}

\theoremstyle{plain}
\newtheorem{thm}{Theorem}[section]

\newtheorem{claim}[thm]{Claim}

\newtheorem{theorem}[thm]{Theorem}

\newtheorem{corollary}[thm]{Corollary}
\newtheorem{proposition}[thm]{Proposition}

\newtheorem{lemma}[thm]{Lemma}

\theoremstyle{definition}
\newtheorem{definition}[thm]{Definition}
\newtheorem{notation}[thm]{Notation}
\newtheorem{condition}[thm]{Condition}
\newtheorem{remark}[thm]{Remark}
\newtheorem{example}[thm]{Example}

\theoremstyle{remark}

\newcommand{\N}{\mathbb{N}}
\newcommand{\R}{\mathbb{R}}

\newcommand{\prob}{\mathbb{P}}
\newcommand{\E}{\mathbb{E}}

\newcommand{\un}{\mathds{1}}
\newcommand{\one}{\mathds{1}}

\renewcommand{\textbf}[1]{\begingroup\bfseries\mathversion{bold}#1\endgroup}

\def\calC{\mathcal{C}}
\def\calD{\mathcal{D}}

\def\calF{\mathcal{F}}

\def\calH{\mathcal{H}}
\def\calI{\mathcal{I}}

\def\calU{\mathcal{U}}

\def\calW{\mathcal{W}}

\def\cov{\mathrm{Cov}}
\def\piv{\textup{Piv}}

\def\E{\mathbb{E}} %{{\bf E}}

\def\dim{\textup{dim}}

\newcommand{\Vx}[1][{}]{V_{x_{#1}}}
\newcommand{\Vxd}[1][{}]{V_{x_{#1}}'}
\newcommand{\tVxd}[1][{}]{\widetilde{V}_{x_{#1}}'}
\newcommand{\VFd}[1][{}]{V_{F_{#1}}'}
\newcommand{\tVFd}[1][{}]{\widetilde{V}_{F_{#1}}'}

\def\piv{\textup{Piv}}
\newcommand{\pivx}[1][{}]{\piv_{x_{#1}}(\hat{A}_{#1})}
\newcommand{\pivxs}[1][{}]{\piv_{x_{#1}}^\sigma(\hat{A}_{#1})}
\newcommand{\pivxp}[1][{}]{\piv_{x_{#1}}^+(\hat{A}_{#1})}
\newcommand{\pivxm}[1][{}]{\piv_{x_{#1}}^-(\hat{A}_{#1})}
\newcommand{\tpiv}{\widetilde{\piv}}
\newcommand{\tpivx}[1][{}]{\widetilde{\piv}_{x_{#1}}(\hat{A}_{#1})}
\newcommand{\tpivxs}[1][{}]{\widetilde{\piv}^{\sigma_{#1}}_{x_{#1}}(\hat{A}_{#1})}
\newcommand{\tpivxp}[1][{}]{\widetilde{\piv}^+_{x_{#1}}(\hat{A}_{#1})}
\newcommand{\tpivxm}[1][{}]{\widetilde{\piv}^-_{x_{#1}}(\hat{A}_{#1})}
\newcommand{\tpivF}[1][{}]{\widetilde{\piv}_{F_{#1}}(\hat{A}_{#1})}

\newcommand{\dv}{\mathrm{dv}}
\renewcommand{\d}{\mathrm{d}}

\def \eps {\varepsilon}

\def\st{\ :\ }

\def\br#1{\left(#1\right)}
\def\brb#1{\left[#1\right]}
\def\brc#1{\left\{#1\right\}}

\newcommand{\ar}[1]{{\color{brown} #1}}

\numberwithin{equation}{section}
\begin{document}

\title[A covariance formula for topological events of smooth Gaussian fields]{A  covariance formula for topological events \\ of smooth Gaussian fields}
\author{Dmitry Beliaev\textsuperscript{1}}
\email{dmitry.belyaev@maths.ox.ac.uk}
\address{\textsuperscript{1}Mathematical Institute, University of Oxford}
\author{Stephen Muirhead\textsuperscript{2}}
\email{s.muirhead@qmul.ac.uk}
\address{\textsuperscript{2}School of Mathematical Sciences, Queen Mary University of London}
\author{Alejandro Rivera\textsuperscript{3}}
\email{alejandro.rivera@univ-grenoble-alpes.fr}
\address{\textsuperscript{3}Institute of Mathematics, EPFL}
\thanks{The first author was partially supported by the Engineering and Physical Sciences Research Council (EPSRC) Fellowship EP/M002896/1 ``Random Fractals''. The second author was partially supported by the Engineering and Physical Sciences Research Council (EPSRC) Grant EP/N0094361/1 ``The many faces of random characteristic polynomials''. The third author was partially supported by the ERC Starting Grant LIKO, Pr.\ ID6:76999. The authors would like to thank Hugo Vanneuville for helpful discussions at an early stage of the project, as well as for pointing out that the covariance formula gives an alternate justification for the Harris criterion (see Section~\ref{s:hc}), and finally, for interesting discussions concerning reformulations of Piterbarg's formula. The authors would also like to thank an anonymous referee for pointing out a mistake in the statement of Corollary 1.6 in a previous version of this article.}
\date{\today}
\begin{abstract} 
We derive a covariance formula for the class of `topological events' of smooth Gaussian fields on manifolds; these are events that depend only on the topology of the level sets of the field, for example (i) crossing events for level or excursion sets, (ii) events measurable with respect to the number of connected components of level or excursion sets of a given diffeomorphism class, and (iii) persistence events. As an application of the covariance formula, we derive strong mixing bounds for topological events, as well as lower concentration inequalities for additive topological functionals (e.g.\ the number of connected components) of the level sets that satisfy a law of large numbers. The covariance formula also gives an alternate justification of the Harris criterion, which conjecturally describes the boundary of the percolation university class for level sets of stationary Gaussian fields. Our work is inspired by \cite{rv_17a}, in which a correlation inequality was derived for certain topological events on the plane, as well as by \cite{piterbarg}, in which a similar covariance formula was established for finite-dimensional Gaussian vectors.
\end{abstract}
\maketitle
%\pagebreak
%\setcounter{tocdepth}{2}
%\tableofcontents

This is a correction to the published version of the article (\textit{Annals of Probability} 48(6), 2845--2893 (2020)).  In the published version it is claimed that the contribution from $0$-dimensional strata (e.g.\ the corners of a polytope) may be omitted from definition of the pivotal measures $d\pi^{\pm}$ that appear in the main covariance formula Theorem \ref{t:main}, and subsequent formulae, because the Hessian on $0$-dimensional strata is always equal to zero. This was incorrect. Instead the contribution from each $0$-dimensional stratum $F_i = \{x_i\}$ must be included, by interpreting the relevant terms in the definition of $d\pi^{\pm}$ as follows:
\begin{itemize}
\item  The Riemannian volume measure $\dv_{F_i}$ is a delta mass at $x_i$;
\item The derivative $d_{x_i}f_t^i|_{F_i}$ is identically zero, and the density of the Gaussian vector in \eqref{e:vec} is defined relative to the induced subspace $d_{x_i} f_t^i|_{F_i} = 0$ (equiv.\ $d_{x_i} f_t^i|_{F_i}$ is omitted from the vector);
\item The Hessian $H_{x_i}^{F_i} f_t^i$ is equal to $1$.
\end{itemize}
We have made minor adjustments to the paper to accomodate this change. A few other minor typos have also been corrected.

\section{Introduction}
\label{s:intro}
In recent years there has been much progress in the study of the topology of level sets of smooth Gaussian fields. {Techniques} have been developed to estimate their homology (see \cite{ns_11, ns_16}, and also \cite{bw_17, cs_14, gw_14, kw_17, sw_16}), and also their large scale connectivity properties (see \cite{alex_96, bg_16}, and also \cite{bmw_17, ms_83b, mv_18, rv_17b}) using ideas from Bernoulli percolation. When studying the topology of level sets, one often has to estimate quantities such as 
\[ \cov(A_1,A_2):=\prob[A_1\cap A_2]-\prob[A_1]\prob[A_2],\]
where $A_1$ and $A_2$ are events of topological nature. Since the events $A_1$ and $A_2$ in general do not admit explicit integral representations, the quantity $\cov(A_1,A_2)$ is often estimated indirectly, leading to inequalities of varying precision. In the present work we prove an \textit{exact} formula for $\cov(A_1,A_2)$, where $A_1$ and $A_2$ belong to a large class of `topological events'.

\smallskip
Let us illustrate our formula with a simple example. Let $f$ be an a.s.\ $C^2$ centred Gaussian field on $\R^2$, with covariance $K(x,y):=\cov(f(x),f(y))$, such that, for each distinct $x,y\in\R^2$, $(f(x),\nabla f(x),f(y), \nabla f(y))$ is a non-degenerate Gaussian vector. Let $B_1$ and $B_2$ be two boxes on the plane $\R^2$, not necessarily disjoint, each with two opposite sides distinguished (we call these `left' and `right', with the remaining sides being `top' and `bottom'). For each $i\in\{1,2\}$, consider the event $A_i$ that there exists a continuous path in $B_i\cap\{f\geq 0\}$ joining the `left' and `right' sides. This is known as a `crossing event' for the excursion set $\{f\geq 0\}$, and is of fundamental importance in the study of the connectivity of the level sets \cite{bg_16}. As a corollary of our general covariance formula, we establish the following exact formula for $\cov(A_1,A_2)$:

\begin{corollary}\label{c:crossings}
The quantity $\cov(A_1,A_2)$ is equal to
\[
\sum_{j_1,j_2 \in \{1,\ldots,9 \} }\int_{F_{j_1}^{1}\times F_{j_2}^{2}} \! \! K(x_1,x_2)\int_0^1 \!  \gamma_{t;x_1,x_2}(0) \,  \E_{t;x_1,x_2} \bigg[\prod_{i=1,2}|\det(H^{j_i}_{x_i}f_t^i)|\one_{\piv_{x_i}^{t,i}(A_i)} \bigg]\, \d t  \, \dv_{F_{j_1}^{1}}\dv_{F_{j_2}^{2}},
\]
where:
\begin{itemize}
\item For each $i\in\{1,2\}$, $F_1^{i}:=\mathring{B_i}$ denotes the interior of $B_i$, equipped with its two-dimensional Lebesgue measure $\dv_{F_1^{i}}$, $(F_j^{i})_{j=2,3,4,5}$ denote the sides of $B_i$, equipped with their natural length measure $\dv_{F_j^{i}}$, and $(F_j^{i})_{j=6,7,8,9}$ denote the corners of $B_i$, equipped with a delta measure $\dv_{F_j^{i}}$; the $F_j^i$ are therefore disjoint.
\item For each $t \in [0,1]$, $f_t=(f_t^1,f_t^2)=(f^1,tf^1+\sqrt{1-t^2}f^2)$ denotes a Gaussian field on $\mathbb{R}^2 \times \mathbb{R}^2$ that interpolates between $(f^1, f^1)$ and $(f^1,f^2)$, where $f^1$ and $f^2$ are independent copies of~$f$. For each distinct $x_1\in B_1$ and $x_2\in B_2$, $\gamma_{t;x_1,x_2}(0)$ denotes the density at $0$ of the Gaussian vector
 \begin{equation}
 \label{e:crossings2}
 (f_t^1(x_1), \nabla f_t^1|_{F_{j_1}^1}(x_1),f_t^2(x_2),\nabla f_t^2|_{F_{j_2}^2}(x_2)) \in \mathbb{R} \times T_{x_1}F_{j_1}^{1} \times \mathbb{R} \times T_{x_2}F_{j_2}^{2}, \end{equation}
where $F_{j_i}^{i}$ denotes the unique interior/side/corner that contains $x_i$; moreover, $\E_{t;x_1,x_2}\brb{\cdot}$ denotes expectation conditional on the vector \eqref{e:crossings2} vanishing, and $H_{x_i}^{j_i} f_t^i  = \nabla^2 f_t^i|_{F_{j_i}}(x_i)$ denotes the Hessian at the point $x_i$ of $f_t^i$ restricted to the face $F^i_{j_i}$. For the corners, we remove the term $\nabla f_t^i|_{F_{j_i}^i}(x_i)$ from the vector in \eqref{e:crossings2}, and set $H_{x_i}^{j_i} f_t^i$ equal to $1$.
\item For each $i\in\{1,2\}$, $t \in [0,1]$ and $x\in B_i$, $\piv_x^{t,i}(A_i)$ denotes the event that there exists a continuous path in $B_i\cap\{f_t^i\geq 0\}$ joining the `left' and `right' sides, and a continuous path in $B_i\cap\{f_t^i\leq 0\}$ joining the `top' and `bottom' sides, both of which pass through $x$ (see Figure \ref{fig:pivotal}; central panels). This is a natural analogue of a `pivotal event' in Bernoulli percolation (see~\cite{bollobas_riordan, Grimmett}).
\end{itemize}
\end{corollary} 
\begin{figure}
\includegraphics[width=0.9\textwidth]{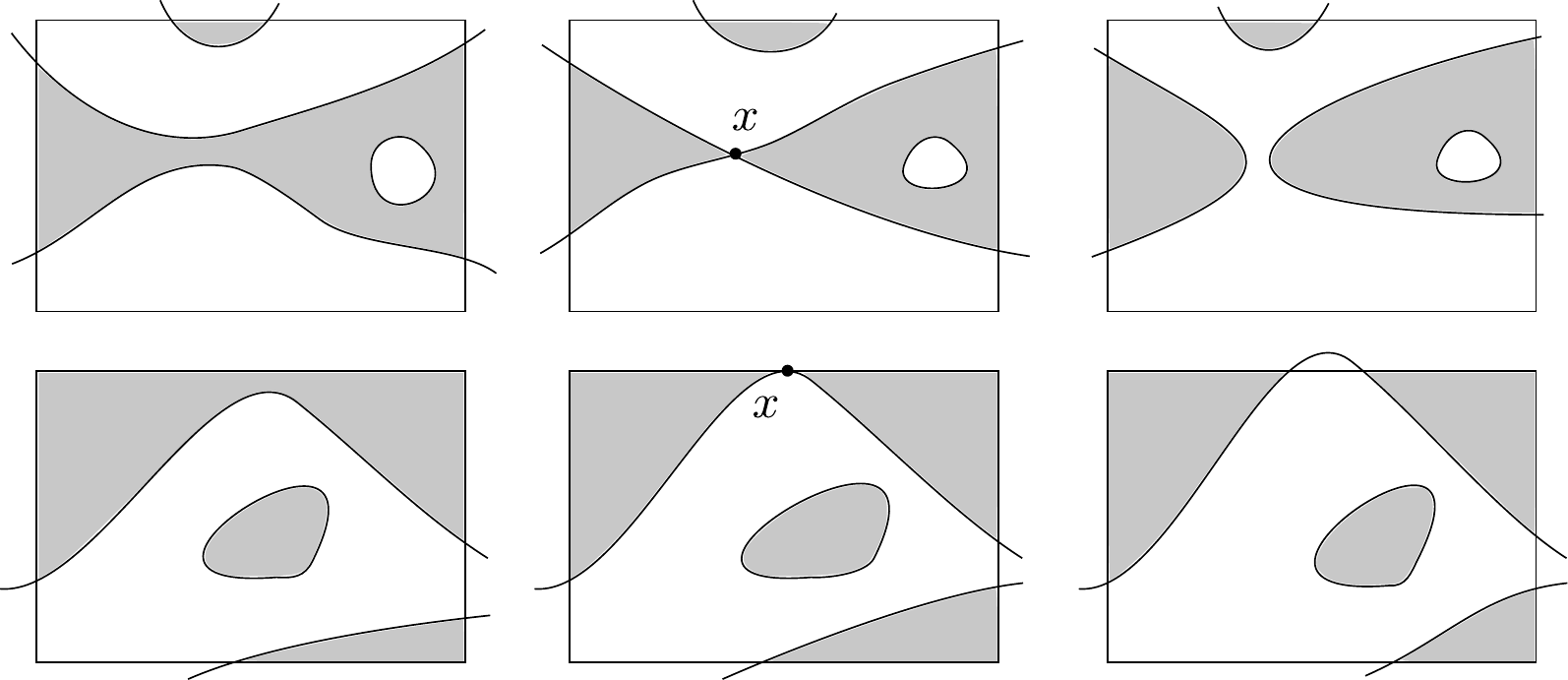}
\caption{An illustration of the crossing events $A_i$ and the pivotal events $\piv_x(A_i)$ that appear in the covariance formula in Corollary~\ref{c:crossings}. Left panels: Two realisations of a field $f$ which exhibit the left-right crossing event for $\{f > 0\}$ in the rectangle $B$ (shown in grey). Right panels: After a small perturbation of $f$ (compared to the left panel), the left-right crossing event no longer occurs. Central panels: The `pivotal event' at which the crossing event first fails in this perturbation; this event can be of three possible types, either involving a level-$0$ critical point $x$ of $f$ in the interior of $B$ (top figure), or involving a level-$0$ critical point $x$ of $f$ restricted to the top side of $B$ (bottom figure), or involving a level-$0$ point $x$ on the corner of $B$ (not shown).}
\label{fig:pivotal}
\end{figure}

Let us make three observations concerning the formula in Corollary \ref{c:crossings}:
\begin{itemize}
\item If $K$ is non-negative then so is the integrand in the formula, and we deduce that
\[ \prob[A_1\cap A_2]\geq \prob[A_1]\prob[A_2]. \]
This is the analogue of the Fortuyn-Kasteleyn-Ginibre (FKG) inequality (see \cite{bollobas_riordan, Grimmett}), originally proven in the Gaussian setting by Pitt~\cite{pitt_82}.
\item Assume that $f$ is stationary, let $\kappa(x)=K(0,x)$ and denote $\overline{\kappa}(r)=\sup_{|x|\geq r}|\kappa(x)|$. Since $f$ is Gaussian, the Hessians $H^{j_i}_{x_i}f_t^i$ have finite moments and so, by stationarity, the conditional expectation in the formula is bounded. Thus, if $B_1$ and $B_2$ have sides of length $O(R)$ and are at distance of order at least $R$, we deduce a `strong mixing' bound for crossing events, namely that
\begin{equation}
\label{e:crossing}
| \prob[A_1\cap A_2] - \prob[A_1]\prob[A_2] | =O(R^4\overline{\kappa}(R)) .
\end{equation}
In particular, as long as $\kappa(R)=o(R^{-4})$, the crossing events $A_1$ and $A_2$ are asymptotically independent, recovering the recent result of Rivera and Vanneuville \cite{rv_17a}.
\item Setting $B_1 = B_2$ (and so $A_1 = A_2$), Corollary \ref{c:crossings} also yields a formula for the \textit{variance} of (the indicator function of) the crossing event $A_i$.
\end{itemize} 

\smallskip
The main result of this paper (see Theorem \ref{t:main}) consists of a vast generalisation of Corollary~\ref{c:crossings} to the class of topological events of smooth Gaussian fields on manifolds of any dimension. In particular, this permits a generalisation of the mixing bound \eqref{e:crossing} to arbitrary topological events on manifolds (see Corollary~\ref{c:mix_euc} for the Euclidean case and Theorem \ref{t:mix_gen} for the general case). Since the statement of Theorem \ref{t:main} requires several preliminary definitions, in this introduction we instead focus on applications of this formula, including (i) the aforementioned strong mixing bounds, and (ii) lower concentration inequalities for additive topological functionals of the level sets, such as such as the number of connected components contained in a given domain.

\smallskip
Our work was largely inspired by \cite{rv_17a} in which the mixing bound \eqref{e:crossing} was first established, improving similar bounds that had previously appeared in \cite{bg_16, bm_18}. Here we extend the techniques and results in~\cite{rv_17a} to arbitrary topological events and to higher dimensions; the key difference in our approach is that we work \textit{directly in the continuum}, rather than with discretisations of the field as in~\cite{bg_16, bm_18, rv_17a}.

\subsection{Topological events}
We begin by describing the class of topological events to which our results apply. Broadly speaking, we study events that depend only on the topology of the level sets $\{f=\ell\}$ (or excursion sets $\{f>\ell\}$) of a Gaussian field $f$ restricted to reasonable bounded domains $B\subset \R^d$. One might think that it would therefore be enough to study homeomorphism classes of pairs $(\{f>\ell\}\cap B,B)$, however, this would in fact \textit{not} identify crossing events, which distinguish marked sides of the reference domain~$B$. Moreover, as in the case of a product of homeomorphic sets, one might wish to distinguish between factors. For these reasons, we work instead with equivalence classes induced by isotopies that preserve certain subsets of~$B$, using the formalism of \textit{stratifications}.

\smallskip
An \textit{affine stratified set} in $\R^d$ is a compact subset $B\subset\R^d$ equipped with a finite partition $B=\sqcup_{F\in\calF}F$ into open connected subsets of affine subspaces of $\R^d$, such that for each $F,F'\in\calF$, $F\cap\overline{F'}\neq\emptyset\Rightarrow F\subset\overline{F'}$. The partition $\calF$ is called a \textit{stratification} of $B$. When there is no risk of ambiguity, we will often refer to $B$ itself as an affine stratified set. For example, a closed cube in $\R^d$, equipped with the collection of the interiors of its faces of all dimensions and its corners, is an affine stratified set. 

\smallskip
Given an affine stratified set $(B,\calF)$ of $\R^d$ and a continuous map $H:B\times[0,1]\rightarrow B$, we say that $H$ is a \textit{stratified isotopy if} for each $t\in [0,1]$, $H(\cdot,t)$ is a homeomorphism such that for each $F\in\calF$, $H(F\times\{t\})=F$. The \textit{stratified isotopy class} of a subset $E\subset B$, denoted by $[E]_B$, is the set of $H(E\times\{1\})$ where $H:B\times[0,1]\rightarrow B$ ranges over the set of stratified isotopies of $B$ with $H(\cdot,0)=id_B$. We consider the stratified isotopy class $[\{f>0\}]_B$ of the excursion set $\{f>0\}$, which captures what we mean by the `topology' of the level set $\{f=0\}$ restricted to~$B$. As we verify in Corollary~\ref{c:measurability}, under mild conditions on $f$ the stratified istotopy class $[\{f>0\}]_B$ is measurable with respect to~$f$. 

\smallskip
A \textit{topological event} in $B$ is an event measurable with respect to $[\{f>0\}]_B$. Important examples include:
\begin{itemize}
\item As in Corollary \ref{c:crossings}, crossing events for level or excursion sets inside a box $B$, e.g.\ the event that a connected component of $\{f = 0\} \cap B$ or $\{ f > 0\} \cap B$ intersects opposite $(d-1)$-dimensional faces of~$B$ (Corollary \ref{c:crossings} concerned the case $d=2$).
\item Events that depend on the number of the connected components of a level or excursion set inside a polytope $B$, or more generally the number of such components of a given diffeomorphism class (see, e.g., \cite{cs_14, gw_14, ns_11, ns_16, sw_16}).
\item The `persistence' event that $\{f|_B > 0\}$ (see, e.g., \cite{as_15, dm_15, ffn_17, rivera_2017}).
\end{itemize}
We write $\sigma_{\textup{top}}(B)$ to denote the $\sigma$-algebra of topological events on $B$.

\subsection{Strong mixing in the Euclidean setting}
The \textit{strong mixing} of a random field is defined via the decay, for domains $B_1$ and $B_2$ that are well-separated in space, of the $\alpha$-mixing coefficient 
\begin{equation}
\label{e:alpha}
 \alpha(B_1, B_2) = \sup_{A_1 \in \sigma(B_1), \, A_2 \in \sigma(B_2)} | \mathbb{P}[A_1 \cap A_2] - \mathbb{P}[A_1] \mathbb{P}[A_2] | , 
 \end{equation}
where $\sigma(B)$ denotes the sub-$\sigma$-algebra generated by the restriction of $f$ to the domain $B$. Strong mixing is a classical notion in probability theory with important connections to laws of large numbers, central limit theorems, and extreme value theory (see, e.g., \cite{doukhan_94, ll_96, loynes_65, rosenblatt_56}) among other topics. While for general continuous processes there is a rich literature on strong mixing (see \cite{bradley_05} for a review), in the study of \textit{smooth} random fields the concept of strong mixing is often far too restrictive. For example, if the spectral density of a stationary Gaussian process decays exponentially (which implies the real analyticity of the covariance kernel and the corresponding sample paths), then by \cite{kr_60} there is no strong mixing regardless of how rapidly correlations decay, unless one restricts the class of events that are controlled by the $\alpha$-mixing coefficient. As a first application of our covariance formula we derive conditions that guarantee the strong mixing of the class of topological events.

\smallskip
Let $f$ be an a.s.\ $C^2$ stationary Gaussian field on $\R^d$ with covariance $\kappa(x) = \textup{Cov}(f(0),f(x))$, and suppose that, for each distinct $x,y\in\R^d$, $(f(x),\nabla f(x) ,f(y), \nabla f(y))$ is a non-degenerate Gaussian vector. These conditions ensure that $\kappa$ is $C^4$, and that the level set $\{f=0\}$ is a $C^2$-smooth hypersurface. For each pair of affine stratified sets $B_1, B_2 \subset \mathbb{R}^d$, define the `topological' $\alpha$-mixing coefficient
\[ \alpha_{\textup{top}}(B_1, B_2) =  \sup_{A_1 \in \sigma_{\textup{top}}(B_1) ,\, A_2 \in \sigma_{\textup{top}}(B_2) } | \prob[A_1 \cap A_2] - \prob[A_1] \prob[A_2] | . \] 

\begin{corollary}[Strong mixing for topological events]
\label{c:mix_euc}
There exist $c_1, c_2> 0$ such that, for every pair of affine stratified sets $(B_1,\calF_1)$ and $(B_2,\calF_2)$ in $\mathbb{R}^d$ satisfying
\[\max_{\alpha\in\N^d:|\alpha|\leq 2} \, \, \sup_{x_1 \in B_1, x_2 \in B_2} |\partial^\alpha \kappa(x_1-x_2)|< c_1,\]
it holds that
\[
\alpha_{\textup{top}}(B_1, B_2)\leq c_2 \, |\mathcal{F}_1| |\mathcal{F}_2| \max_{F_1 \in \mathcal{F}_1 , F_2 \in \mathcal{F}_2}   \int_{F_1 \times F_2}  \!\!\!    |\kappa(x_1 - x_2)| \,  \dv_{F_1}(x_1) \, \dv_{F_2}(x_2) 
\]
where $\dv_F$ is the natural Riemannian volume measure on $F$ (a delta mass if $F$ is a corner). In particular, recalling that $\bar{\kappa}(s) = \sup_{|x| \ge s} |\kappa(x)|$, if 
  \begin{equation}
\label{e:decay}
\lim_{|x|\rightarrow\infty}|\partial^\alpha\kappa(x)|=0 \ , \quad \text{for all } \alpha\in\N^d \text{ such that } |\alpha|\leq 2 ,
 \end{equation} 
then for every pair of disjoint affine stratified sets $B_1, B_2 \subset \mathbb{R}^d$ there exist $c_3, c_4 > 0$ such that
\begin{equation}
\label{e:mix_euc2}
\alpha_{\textup{top}}(s B_1, sB_2 ) \le c_3 s^{2d} \, \bar{\kappa}(c_4 s)  \quad \text{for all } s \ge 1 . 
 \end{equation}
\end{corollary}

Corollary \ref{c:mix_euc} demonstrates that topological events on well-separated boxes $B_1, B_2 \subset \mathbb{R}^d$ are independent up to an additive error that depends (up to a constant) solely on the double integral of the absolute value of the covariance kernel on the boxes; we expect this result to have many applications. Later we present a generalisation of Corollary \ref{c:mix_euc} to Gaussian fields on general manifolds (see Theorem \ref{t:mix_gen}). The proof of Corollary \ref{c:mix_euc} is given in Section \ref{s:application_proofs}.

\begin{remark}
\label{r:mix_euc}
The constant $c_1$ in Corollary \ref{c:mix_euc} can be chosen in a way that depends only on the dimension~$d$, on $\kappa(0)$, and on the Hessian of $\kappa$ at $0$, whereas the constant $c_2$ can be chosen in a way that depends, in addition to these, also on $ \max_j  (\partial^4 \kappa(0)/\partial x_j^4)$.
\end{remark}

\begin{remark}
We do not assume that the field $f$ is centred. Since adding a constant does not change the covariance kernel, Corollary \ref{c:mix_euc} also bounds the strong mixing of topological events that are defined in terms of non-zero levels.  Notably, neither $c_1$ nor $c_2$ depends on the mean value of the field.
\end{remark}

\begin{remark}
As explained above, the mixing bound in Corollary \ref{c:mix_euc} was already known in two dimensions, at least in the case of crossing events \cite{rv_17a} (see also \eqref{e:crossing}); our results extends this mixing bound to arbitrary dimensions and arbitrary topological events. Note also that an analogue of~\eqref{e:mix_euc2} was recently established \cite{mv_18} for a version of the $\alpha$-mixing coefficient that controls all events (not necessarily topological) that depend \textit{monotonically} on~$f$ (this includes, for instance, crossing events for $\{f > 0\}$); in this case the factor $s^{2d}$ can be improved to~$s^d$.
\end{remark}

\subsection{Application to lower concentration for topological counts}
We next present a simple application of Corollary \ref{c:mix_euc} to give a taste of the utility of mixing bounds. A \textit{topological count} is a set of integer-valued random variables $N = N(B)$, indexed by affine stratified sets $B \subset \mathbb{R}^d$, each of which is measurable with respect to the corresponding $\sigma$-algebra $\sigma_{\textup{top}}(B)$. We call a topological count \textit{super-additive} if, for every affine stratified set $B$ and every collection of disjoint affine stratified sets $(B_i)_{i \le k}$ contained in~$B$,
\begin{equation}
\label{e:super_additive}
N(B)  \geq \sum_{i\le k} N(B_i)\, .
\end{equation}
Examples of super-additive topological counts include the number of connected components of level or excursion sets that are fully contained in a set \cite{ns_16}, or more generally the number of connected components of these sets that have a certain diffeomorphism class \cite{cs_14, gw_14, sw_16}. In one dimension, topological counts reduce to the number of solutions to $\{f = 0\}$ in intervals, a quantity studied extensively since the works of Kac and Rice in the 1940s \cite{kac_43, rice_45}. We say that a topological count $N$ satisfies a \textit{law of large numbers} if there exists a $c_N > 0$ such that, for  every affine stratified set $B \subset \mathbb{R}^d$, as $s \to \infty$
\begin{equation}
\label{e:lln}   \frac{N(sB)}{s^d \, \textup{Vol}(B)} \to c_N  \quad \text{in probability} .
\end{equation}
Nazarov--Sodin have shown \cite{ns_11, ns_16} (see also \cite{bmm_18, kw_17}) that if $f$ is ergodic (and under certain mild extra conditions) the number of connected components of level or excursion sets satisfies a law of large numbers, and in fact, \eqref{e:lln} converges a.s.\ and in mean; the same result was later shown to be true also for the number of connected components of a given diffeomorphism type \cite{bw_17, cs_14, sw_16} (in the one dimensional case this follows immediately from the ergodic theorem). As was shown in \cite{rv_17a}, quantitative mixing bounds can be used to deduce the lower concentration of super-additive topological counts:

\begin{corollary}[Lower concentration for topological counts]
\label{c:con_euc}
Let $N$ denote a super-additive topological count that satisfies a law of large numbers \eqref{e:lln} with limiting constant $c_N > 0$. Assume that \eqref{e:decay} holds. Then for every affine stratified set $B \subset \mathbb{R}^d$ and constants $\eps, C > 0$, there exist $c_1, c_B > 0$ such that, for every $s \ge 1$,
\begin{equation}
\label{e:con_euc}
  \mathbb{P}\left[  \frac{N(sB)}{s^d \, \textup{Vol}(B)}  \le c_N   - \eps  \right] \le c_1  \inf_{r \in [1, s]} \left(   e^{- C ( s/r)^d } + e^{c_B (s/r)^d } (rs)^d \bar{\kappa}(r) \right)  ,
  \end{equation}
 where the constant $c_B > 0$ depends only on the stratified set $B$. In particular, if there exist $c_2, \alpha > 0$ such that $\kappa(x) \le c_2 |x|^{-\alpha}$ for every $|x| \ge 1$, then for every $\eps, \delta > 0$ we can set  $r = c_3 s/ (\log s)^{1/d}$  for a sufficiently large choice of $c_3 > 0$ (depending on $c_B, \alpha$ and $\delta$) and apply \eqref{e:con_euc} for $C > 0$ sufficiently large (depending on $c_3$ and $\delta$) to deduce the existence of a $c_4 > 0$ such that, for every $s \ge 1$,
\[  
\mathbb{P}\left[  \frac{N(sB)}{s^d \, \textup{Vol}(B)}  \le c_N - \eps  \right] \le  c_4 s^{2d - \alpha + \delta}  . 
\]
Similarly, if there exist $c_2,  \alpha, \beta > 0$ such that $\kappa(x) \le c_2 e^{- \beta |x|^\alpha}$ for every $|x| \ge 1$, then setting $r = c_3 s^{d/(d+\alpha)}$ for a sufficiently large choice of $c_3 > 0$ and then choosing $C$ sufficiently large we deduce that for every $\gamma>0$ there is  $c_4 > 0$ such that, for every $s \ge 1$,
\[  
\mathbb{P}\left[  \frac{N(sB)}{s^d \, \textup{Vol}(B)}  \le c_N - \eps  \right] \le c_4 
\exp\left(- \gamma  s^{d \alpha / (d + \alpha)}\right) . 
\]

\end{corollary}

\begin{remark}
As for Corollary \ref{c:mix_euc}, Corollary \ref{c:con_euc} was also already known in two dimensions (at least in the case of the number of connected components of level sets \cite{rv_17a}) but not in higher dimensions. A stronger version of Corollary~\ref{c:con_euc} was also recently established in the one dimensional case (i.e.\ for the number of zeros of a one-dimensional stationary Gaussian process~\cite{bdfz_17}), and also for the number of connected components of the zero level set of random spherical harmonics (RSHs) \cite{ns_11}; the results in \cite{bdfz_17, ns_11} are proven using very different techniques to ours, and in the latter case relies heavily on the specific structure of the RSHs. 
\end{remark}

%%%%%%%%%%%
\section{A covariance formula for topological events}
\label{s:main_formula}

In this section we present our covariance formula in the general setting of smooth Gaussian fields on smooth manifolds. We also discuss further applications of the formula beyond those we gave in Section \ref{s:intro}, and give a sketch of its proof.

\subsection{The covariance formula}
\label{ss:main_formula}

We begin by fixing definitions, starting with the `stratified sets' on which we work; our main reference is \cite{goresky_macpherson}. Let $(M, g)$ be a smooth Riemannian manifold of dimension $d$.

\begin{definition}[Stratified set]
\label{d:stratification}
Let $B\subset M$ be a compact subset. Assume there is a partition of $B$ into a finite collection $\calF$ of smooth locally closed submanifolds, called \textit{strata}, satisfying the following additional properties:
\begin{itemize}
\item The strata  cover $B$, i.e.\ $B=\coprod_{F\in\calF}F$.
\item Any two strata $F_1$ and $F_2$ satisfy $F_1\cap\overline{F_2}\neq\emptyset\Leftrightarrow F_1\subset\overline{F_2} $. This allows us to equip $\calF$ with the partial order $<$ defined such that, for any two strata $F_1$ and $F_2$,
\[ F_1\cap\overline{F_2}\neq\emptyset\Leftrightarrow F_1=F_2\text{ or }F_1<F_2 . \]
\item For each $F_1<F_2$ the following is true. Consider any embedding of $M$ in Euclidean space, and let $(x_k)_{k\in\N}$ and $(y_k)_{k\in\N}$ be sequences of points satisfying (i) for each $k\in\N$, $x_k\in F_2$ and $y_k\in F_1$, (ii) $x_k$ and $y_k$ converge to a common point $y\in F_1$, (iii) the tangent planes $T_{x_k}F_2$ converge to a limit $\tau$, and (iv) the lines $\lambda_k$ generated by the vectors $x_k-y_k$ converge to a limit $\lambda$. Then it holds that $\lambda\subset\tau$. Equivalently, it is enough that this condition be fulfilled for one fixed embedding of $M$ in Euclidean space. Limits $\tau$ of this kind are called \textit{generalised tangent spaces at $y$}.
\item For each $F_1,F_2\in\calF$ such that $F_1<F_2$, there exists a smooth sub-bundle $TF_2|_{F_1}$ of $TM|_{F_2}$, whose rank is the dimension of $F_2$, that contains $TF_1$ as a sub-bundle, and such that (i) the map $y\mapsto T_yF_2$, with values in the adequate Grassmannian bundle defined on $F_2$, extends by continuity to $F_1$ together with all of its derivatives, and (ii) for each sequence of points $x_k\in F_2$ converging to a limit $x\in F_1$, $\lim_{k\rightarrow+\infty}T_{x_1}F_2=T_xF_2|_{F_1}$. We call $TF_2|_{F_1}$ the \textit{generalised tangent bundle of $F_2$ over~$F_1$} (see Figure \ref{fig:tame}).
\end{itemize}
The collection $\calF$ is called a \textit{tame stratification} of $B$. A \textit{stratified set of $M$} is a pair $(B, \calF)$ consisting of a compact subset $B \subset M$ and a tame stratification $\calF$ of $B$. When there is no risk of ambiguity, we will often write that $B\subset M$ is a stratified set without explicit mention of its tame stratification~$\calF$.
\end{definition}
\begin{figure}[h]
\includegraphics[width=0.4\textwidth]{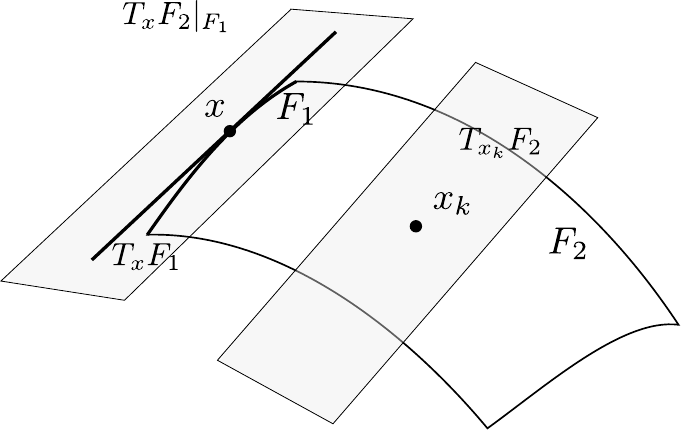}
\includegraphics[width=0.4\textwidth]{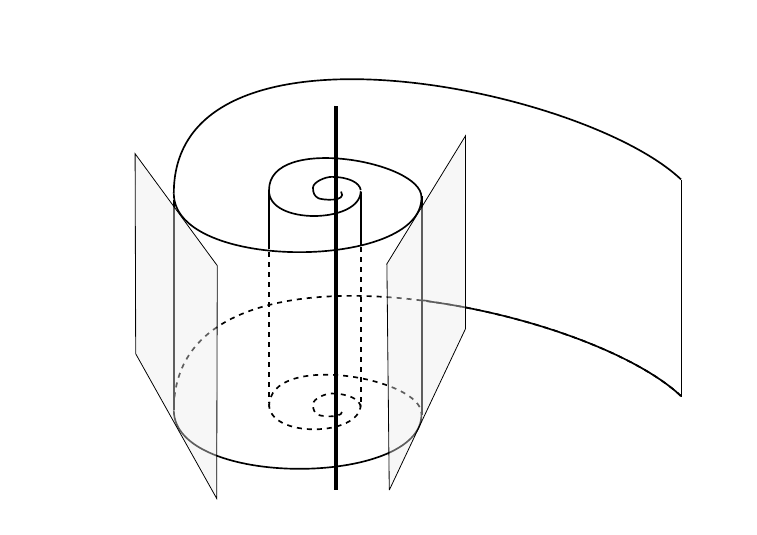}
\caption{Left: An example of a tame stratification $\calF = \{F_1, F_2\}$ of a compact set~$B$. Here the generalised tangent bundle $T_xF_2|_{F_1}$ is well-defined since, as the points $x_k$ converge to $x\in F_1$, the respective tangent planes also converge. Right: A rough depiction of the `rapid spiral sheet', which is an example of a set that cannot be tamely stratified (see Example \ref{ex:stratified_sets}); here tangent planes do not converge, and so the generalised tangent bundle is not well-defined.}
\label{fig:tame}
\end{figure}

\begin{remark}\label{rk:whitney_stratifications}
A partition $\calF$ of a compact subset $B$ satisfying the first three properties required in Definition \ref{d:stratification} is called a Whitney stratification (see for instance Part I, Section~1.2 of \cite{goresky_macpherson}); indeed, the third property is known as `Whitney's condition (b)'. While Whitney stratifications have many interesting properties, sometimes the structure of a stratification can force functions on it to have degenerate stratified critical points (see Example \ref{ex:stratified_sets}). To avoid such pathologies, we add the additional fourth condition which is satisfied in most natural examples. In fact, this additional `tameness' property is only used at a single place in the proof of the covariance formula, namely, to prove Claim \ref{cl:incidence_manifolds}.
\end{remark}

Let us present several important examples (and one non-example) of stratified sets, beginning with the trivial stratification:

\begin{example}[Trivial stratification]
Let $M$ be a compact manifold without boundary. Then $\calF = \{ M \}$ is a tame stratification of $M$. Moreover, let $\Omega\subset M$ be a compact subset with smooth boundary~$\partial\Omega$. Then $\calF=\{\mathring{\Omega},\partial\Omega\}$ is a tame stratification of $\Omega$.
\end{example}

In the case that $M=\R^d$, by gluing boxes and other polytopes together one obtains sets equipped with a natural stratification that will, in most case, be tame. Our definition of `affine stratified set', introduced in Section \ref{s:intro}, covers all such examples:

\begin{example}[Affine stratified sets]
\label{d:affine_stratified_set}
The affine stratified sets introduced in Section \ref{s:intro} are stratified sets of $M = \R^d$.
\end{example}

One can also consider individual `polytopes', such as the boxes in Corollary \ref{c:crossings}, to be stratified sets of $M = \R^d$:
\begin{example}[Polytopes]
A polytope in $\R^d$ is naturally equipped with a stratification whose strata are the faces of the polytope of all dimensions. Though to our knowledge there is no consensus on the definition of a polytope in $\R^d$, it is easy to check whether or not a specific example satisfies Definition \ref{d:stratification}.
\end{example}

We also present one non-example, in the form of the `rapid spiral':

\begin{example}[Rapid spiral]\label{ex:stratified_sets}
The rapid spiral $B=\{r=e^{-\theta^2}\}$ (see Figure~\ref{fig:tame}) admits a natural partition that satisfies all the conditions of a tame stratification except the last; in particular, this partition is a Whitney stratification. The rapid spiral $B$ exhibits certain pathologies that result from the lack of tameness, for instance, there are no stratified Morse functions on $B$ (see \cite[Part I, Example 2.2.2]{goresky_macpherson}).
\end{example}

We next extend the definition of topological events given in Section \ref{s:intro} to the general setting of stratified sets. Let $f$ be a continuous Gaussian field on $M$, defined on a probability space $\Omega$. Let $\mu:M\rightarrow \R$ and $K:M\times M\rightarrow \R$ denote respectively the mean and covariance kernel of $f$. Assume that $f$ satisfies the following condition (generalising the conditions in Section \ref{s:intro}):

\begin{condition}\label{cond:non_degeneracy}
The field $f$ is a.s.\ $C^2$. Moreover, for each distinct $x,y\in M$, the Gaussian vector 
\[ 
(f(x),d_xf,f(y),d_yf)\in \R\times T_x^*M\times\R\times T_y^*M \]
 is non-degenerate.
\end{condition}

This condition ensures that $\mu$ is $C^2$ and that $K$ is of class $C^{2,2}$. Let us now define the class of \textit{topological events} on a stratified set $B$.

\begin{definition}[Topological events]
\label{d:topological_events}
Let $(B, \calF)$ be a stratified set of $M$. A \textit{stratified homeomorphism of $B$} is a homeomorphism $h:B\rightarrow B$ such that for each $F\in\calF$, $h(F)=F$. A \textit{stratified isotopy} of $B$ is a continuous map $H:B\times[0,1]\rightarrow B$ such that for each $t\in[0,1]$, $H(\cdot,t):B\rightarrow B$ is a stratified homeomorphism of $B$. We say that two stratified homeomorphisms $h_0,h_1:B\rightarrow B$ are $\calF$-\textit{isotopic} if there exists a stratified isotopy $H$ such that $H(\cdot,0)=h_0$ and $H(\cdot,1)=h_1$. 

Let $\calD$ denote the \text{excursion set} $\{f>0\}$. The \textit{stratified isotopy class} of $\calD$ in $B$, denoted $[\calD]_B$, is the set of $h(\calD\cap B)$ where $h$ ranges over all stratified homeomorphisms of $B$ that are $\calF$-isotopic to the identity. As we establish in Corollary \ref{c:measurability}, under Condition \ref{cond:non_degeneracy} there are a countable number of stratified isotopy classes, and we equip the set of classes with its maximal $\sigma$-algebra. We will also verify in Corollary \ref{c:measurability} that the map $[\calD]_B$ from the probability space $\Omega$ into the set of stratified isotopy classes is measurable. A \textit{topological event on $B$} is an event $A\subset\Omega$ measurable with respect to the random variable $[\calD]_B$.
\end{definition}

Henceforth we fix two stratified sets $(B_1, \calF_1)$ and $(B_2, \calF_2)$ of $M$ (not necessarily disjoint). Our main formula expresses the covariance between topological events on $B_1$ and $B_2$ in terms of an integral over the `pivotal measure' of the events. This measure is defined in terms of (i) `pivotal points', and (ii) a certain interpolation between $f$ and an independent copy of itself; we introduce these concepts now. Our definition of `pivotal points' is related to the notion of `pivotal sites' in percolation theory (see \cite[Section 2.4]{Grimmett}), whereas the interpolation is based on the classical interpolation argument of Piterbarg \cite{piterbarg}.

\begin{definition}[Pivotal points]
\label{d:pivotal_sets}
Fix $\hat{A}\subset C^1(M)$. For every $u \in C^1(M)$, we say that $x \in M$ is \textit{pivotal for} $u$ (\textit{with respect to} $\hat{A}$) if, for any open neighbourhood $W$ of $x$ in $M$, there exists a function $h\in C^2_c(W)$ such that for every sufficiently small $\delta>0$, $u+\delta h\in\hat{A}$ and $u-\delta h\notin \hat{A}$. Such a function $u$ is described as having a \textit{pivotal point at} $x \in M$, and we denote by $\pivx\subset C^1(M)$ the set of all such $u$'s. If $h$ can be chosen so that $h\geq 0$, we say that $x$ is \textit{positively pivotal for} $u$, and we denote by $\pivxp\subset C^1(M)$ the set of such $u$'s. Similarly, $x$ is \textit{negatively pivotal for} $u$ if $h$ can be chosen so that $h \leq 0$, and we denote $\pivxm\subset C^1(M)$ the set of such $u$'s.
\end{definition}

\begin{definition}[Interpolation]
Let $\tilde{f}$ be an independent copy of $f$. For each $t\in[0,1]$, define the Gaussian field on $M \times M$
\begin{equation}
\label{e:inter}
 f_t(x)=(f^1_t(x),f^2_t(x)):=(f(x),t (f(x)-\mu(x))+\sqrt{1-t^2}(\tilde{f}(x)-\mu(x))+\mu(x)).
 \end{equation}
 Observe that $f^1_t$ and $f^2_t$ have the same law as $f$, and $\textup{Cov}(f^1_t(x_1),f^2_t(x_2))=tK(x_1,x_2)$; in particular, $f^1_0$ and $f^2_0$ are independent, while $f^1_1=f^2_1$. Also, observe that $f^1_t$ and $f^2_t$ both satisfy Condition~\ref{cond:non_degeneracy}. For each $x_1 \in F_1 \in \calF_1$ and $x_2 \in F_2 \in \calF_2$, denote by $\gamma_{t;x_1,x_2}(0)$ the density at zero of the Gaussian vector
\begin{equation}
\label{e:vec}
(f^1_t(x_1),d_{x_1}f^1_t|_{F_1},f^2_t(x_2),d_{x_2}f^2_t|_{F_2})
\end{equation}
in orthonormal coordinates of $\R\times T_{x_1}^*F_1\times\R\times T_{x_2}^*F_2 $, and denote by $\E_{t; x_1,x_2}\brb{\cdot}$ expectation conditional on the vector \eqref{e:vec} vanishing; this conditional expectation is well defined and described by the usual Gaussian regression formula (\cite[Proposition 1.2]{azais_wschebor}) since the vector \eqref{e:vec} is non-degenerate. For $0$-dimensional strata $F_i = \{x_i\}$, by convention $d_{x_i} f_t^i|_{F_i} = 0$, and the density $\gamma_{t;x_1,x_2}(0)$ is taken with respect to the induced subspace $d_{x_i} f_t^i|_{F_i} = 0$ (equiv.\ the term $d_{x_i} f_t^i|_{F_i} = 0$ is omitted from the vector \eqref{e:vec}). Note that, since $x_1$ and $x_2$ correspond to unique strata $F_1$ and $F_2$, to ease notation we have dropped the explicit dependence of $\gamma_{t;x_1,x_2}(0)$ and $\E_{t; x_1,x_2}\brb{\cdot}$ on $F_1$ and $F_2$.
\end{definition}

We are now ready to define the pivotal measure, or more precisely, two `signed' pivotal measures. Fix topological events $A_1$ and $A_2$ on $B_1$ and $B_2$ respectively. Denote by $\widetilde{A}_1,\widetilde{A}_2$ the measurable sets of stratified isotopy classes in $B_1$ and $B_2$ respectively that define these topological events, and let $\hat{A}_1$ (resp.\ $\hat{A}_2$) be the set of functions $u\in C^1(M)$ such that $[\{u > 0\}]_{B_1}\in\widetilde{A}_1$ (resp.\ $[\{u >  0\}]_{B_2}\in\widetilde{A}_2)$. 

\smallskip
Denote by $\dv_g$ the Riemannian volume measure on $M$. Similarly, for each stratum $F \in \calF_1 \cup \calF_2$, denote by $\dv_F$ the Riemannian volume measure induced by $g_F$, the restriction of $g$ to~$F$.  For $0$-dimensional strata $F_i = \{x_i\}$, by convention $\dv_{F_i}$ is a delta mass at $x_i$. If $u\in C^2(M)$ and $x$ is a critical point of $u$, we denote by $H_xu$ the Hessian of $u$ at $x$ (which is well defined since $x$ is a critical point of $u$; see for instance \cite[Chapter 1]{nicolaescu_morse_theory}). More generally, if $F\subset M$ is a smooth submanifold of $M$ and $d_xu|_F=0$, then let $H^F_xu$ be the Hessian of $u|_F$ at $x$.  For $0$-dimensional strata $F_i = \{x_i\}$, by convention $H^{F_i}_{x_i} u = 1$.
\begin{definition}[Pivotal measures]
\label{d:pivotal_measure}
For each $t \in [0, 1]$ and $\sigma\in\{-,+\}$, define the \textit{signed pivotal intensity function} $I_t^{\sigma}( x_1,x_2)$ on $B_1 \times B_2$ to be
\begin{equation} \label{e:intensity} 
\sum_{\substack{\sigma_1,\sigma_2\in\{-,+\},\\ \sigma_1\sigma_2=\sigma}} \! \! \! \! \! \!
\gamma_{t;x_1,x_2}(0) \,
{\E_{t; x_1,x_2}\brb{|\det(H_{x_1}^{F_1}f_t^1)\det(H_{x_2}^{F_2}f_t^2)|;\, f^1_t\in \piv_{x_1}^{\sigma_1}(\hat{A}_1), f^2_t\in\piv_{x_2}^{\sigma_2}(\hat{A}_2)}} 
\end{equation}
where $F_1$ and $F_2$ denote the (unique) strata in $\calF_1$ and $\calF_2$ that contain $x_1$ and $x_2$ respectively, and the determinants are taken with respect to orthonormal bases of $T_{x_i} F_i$. The \textit{signed pivotal measures} $\d\pi^\sigma(x_1,x_2)$ on $B_1 \times B_2$ are defined, for $\sigma \in \{-, +\}$, as 
\[
\d\pi^\sigma(x_1, x_2 ) = \Big( \int_0^1 I_t^{\sigma}(x_1 , x_2) \, \d t \Big) \, \dv_{F_1}(x_1)\dv_{F_2}(x_2) . 
\]
\end{definition}
 
We emphasise that, although the `pivotal measures' depend on both (i) the stratified sets $B_i$, and (ii) the topological events $A_i$, to ease notation we have left these dependencies implicit. Observe also that $\d \pi^\sigma$ is a sum of measures of different dimensions that are supported on pairs of strata $(F_1, F_2) \in \calF_1 \times \calF_2$. On each such pair, the measures $\d \pi^\pm$ are singular with respect to each other and mutually continuous with respect to the product of Riemannian volume measures.

\begin{remark}
\label{r:piv_mon}
If $A_1$ and $A_2$ are both increasing events (meaning that, for $i \in \{1, 2\}$, if $u \in \hat{A}_i$ and $h$ is a non-negative function, then $u + h \in \hat{A}_i$), then the negative pivotal measure $\d\pi^-$ is identically zero since $\pivxm[i]$ is empty by definition. The same is true if $A_1$ and $A_2$ are both decreasing events, since then $\pivxp[i]$ is empty. Similarly, if $A_1$ is increasing and $A_2$ is decreasing, then $\d\pi^+$ is identically zero.
\end{remark}

We are now ready to present our covariance formula in full generality:

\begin{theorem}[Covariance formula for topological events]
\label{t:main}
Let $(B_1, \calF_1)$ and $(B_2, \calF_2)$ be stratified sets of $M$. Let $f$ be a Gaussian field on $M$ satisfying Condition \ref{cond:non_degeneracy}. Then the covariance of topological events $A_1$ and $A_2$ on $B_1$ and $B_2$ respectively can be expressed as
\[ 
\prob\brb{A_1\cap A_2}-\prob\brb{A_1}\prob\brb{A_2}=\int_{B_1\times B_2}  K(x,y) \, \left(\d\pi^+(x,y)-\d\pi^-(x,y)\right) \! ,
\]
where $d\pi^+$ and $d\pi^-$ denote the pivotal measures introduced in Definition \ref{d:pivotal_measure}.
\end{theorem}

Let us offer some intuition behind the covariance formula in Theorem \ref{t:main}. The starting point of our analysis is the observation that
\[    
\prob\brb{A_1\cap A_2} =\prob[ f_1 \in \hat{A}_1 \times \hat{A}_2   ]     \quad \text{and} \quad   \prob\brb{A_1}\prob\brb{A_2}  =\prob[ f_0 \in \hat{A}_1 \times \hat{A}_2  ]  , 
\]
and hence
\[ 
\prob\brb{A_1\cap A_2}-\prob\brb{A_1}\prob\brb{A_2} =  \int_0^1  \frac{d}{dt} \prob[ f_t \in \hat{A}_1 \times \hat{A}_2   ]  \, \d t .  
\]
As we explain in Section \ref{ss:proof_sketch}, the structure of the Gaussian measure allows us to express 
\[
 \frac{d}{dt} \prob[ f_t \in \hat{A}_1 \times \hat{A}_2  ]  
 \]
 as an integral, over pairs of strata $(F_1, F_2) \in \calF_1 \times \calF_2$, of the (signed) two-point intensity functions~$I^{\sigma}_t$ of critical points that are `pivotal' for the events $A_1$ and $A_2$ respectively, weighted by a term that is the inner product of the outward normal vectors at the boundary of the events $A_1$ and $A_2$; by the properties of the Gaussian measure (in particular, the reproducing property of the covariance kernel), this inner product is just (a normalisation of) the covariance kernel~$K$.
 
 \smallskip
To understand the form of the intensity functions $I^\sigma_t$, notice that pivotal points are necessarily critical points at the zero level. Hence we can understand $I^{\sigma}_t$ as a restriction to pivotal points of the standard two-point intensity function for critical points of $f_t$ on $(F_1, F_2)$ at the zero level, which by the well-known Kac-Rice formula (see \cite[Chapter 6]{azais_wschebor}) is given by
\[  
\gamma_{t;x_1,x_2}(0) \, \E_{t; x_1,x_2}\brb{|\det(H^{F_1}_{x_1} f_t^1)\det(H^{F_2}_{x_2} f_t^2)|} . 
 \]
Note that our intensity functions are signed; this is because we must distinguish pairs of pivotal points that are pivotal `in the same direction', in the sense that a local increase in $f$ causes the events $A_1$ and $A_2$ to \textit{both} occur or to \textit{both not} occur, from those that are pivotal `in opposite directions'. 

\smallskip
It is possible that some variant of Theorem \ref{t:main} remains true for a wider class of smooth random fields. The Kac-Rice formula applies far beyond the Gaussian setting, and in principle one can also express the intensity of pivotal points for non-Gaussian fields. As for the initial interpolation step, by formulating it using the Ornstein-Uhlenbeck semigroup (as in, say, \cite{chatterjee_2008} or as suggested in \cite{vanneuville_notes}) the setting could perhaps be extended to measures related to other Markov semigroups. We leave this for future investigation.

\subsection{Applications}
\label{ss:applications}

We next present applications of the covariance formula in Theorem \ref{t:main}; some of these have already been discussed (see Corollaries \ref{c:mix_euc} and  \ref{c:con_euc}), but here we give extensions to more general settings. The proofs will be deferred to Section \ref{s:application_proofs}. Throughout this section we assume that $f$ satisfies Condition~\ref{cond:non_degeneracy}.

\subsubsection{Strong mixing for topological events}
\label{ss:sm}

Our first application generalises the strong mixing statement in Corollary \ref{c:mix_euc} to the set-up in Section \ref{ss:main_formula}. For a stratified set $B \subset M$, let $\sigma_{\text{top}}(B)$ denote the $\sigma$-algebra consisting of topological events in $B$, and for a pair of stratified sets $B_1, B_2 \subset M$, define the corresponding `topological' $\alpha$-mixing coefficient
\begin{equation}
\label{e:alpha_gen}
\alpha_{\text{top}}(B_1, B_2) = \sup_{A_1 \in \sigma_{\text{top}}(B_1), \, A_2 \in \sigma_{\text{top}}(B_2) } | \mathbb{P}[A_1 \cap A_2] - \mathbb{P}[A_1] \mathbb{P}[A_2] |  . 
\end{equation}

\begin{theorem}[Strong mixing for topological events]
\label{t:mix_gen}
There exists a constant $c_d > 0$, depending only on the dimension of the manifold $M$, such that for every pair of stratified sets $(B_1, \calF_1)$ and $(B_2, \calF_2)$ of $M$, 
\[ 
\alpha_{\text{top}}(B_1, B_2)   
\le  
c_d \, \sum\limits_{F_1 \in \calF_1 , F_2 \in \calF_2} \! \! c_{F_1, F_2}  \,  
\int_{F_1 \times F_2}  \!  |K(x_1 , x_2)| \,  \dv_{F_1}(x_1) \, \dv_{F_2}(x_2)    ,
\]
where $c_{F_1, F_2}$ is equal to the maximum, over $i, j, k \in \{1, 2\}$, of
\[ 
\sup_{x_1\in F_1, x_2 \in F_2}\frac{ \br{\E \brb{ \| H_{x_i}^{F_i} f \|_\text{op}^2 \, | \, d_{x_i} f|_{F_i} = 0 }}^{d_i} }{\sqrt{ \textup{det}(\Delta(x_1, x_2) ) } } 
\max 
\brc{ 
1 , \Big(    \frac{ K(x_j, x_j) \,   \textup{det}( d_{x_k} \otimes d_{x_k} K|_{F_i \times F_i} )  }{ \sqrt{\textup{det}(\Delta(x_1, x_2))}} \Big)^{2d_i}    
}  , 
\] 
and where $\|\cdot\|_\text{op}$ denotes the ($L^2$-)operator norm,  $d_i = \textup{dim}(F_i)$, and $\Delta(x_1, x_2)$ is the covariance matrix, in orthonormal coordinates, of the (non-degenerate) Gaussian vector
\[ (f(x_1),d_{x_1}f|_{F_1},f (x_2),d_{x_2}f |_{F_2})  \]
(omitting the term $d_{x_i}f|_{F_i}$ if $F_i$ is $0$-dimensional).
\end{theorem}

\begin{remark}
\label{r:riemann}
All the terms in the definition of $c_{F_1,F_2}$ can be written as a quotient of powers of polynomials of partial derivatives of $K$ of order at most $(2,2)$. This means that (i) $c_{F_1,F_2}$ depends continuously on the $C^{2,2}$ norm of $K$, and (ii) $c_{F_1,F_2}$ is homogeneous in $K$ (the degree of homogeneity is easily seen to be $-1$, which compensates the presence of $K(x_1,x_2)$ in the integral).
\end{remark}

\subsubsection{Sequences of fields: The Kostlan ensemble}

In Corollary \ref{c:mix_euc} we stated a quantitative mixing bound for rescaled (affine) stratified sets $sB_1$ and $sB_2$ as $s \to \infty$. In the setting of compact manifolds $M$, it is often more appropriate to work with a \textit{sequence} of Gaussian fields on $M$ that converge to a local limit, and consider the topological mixing between \textit{fixed} disjoint stratified sets $B _1, B_2\subset M$ (in fact, this includes the setting in Corollary~\ref{c:mix_euc} as a special case, by rescaling the field rather than the sets). 

\smallskip
Rather than work in full generality, here we work only with the \textit{Kostlan ensemble}, which is the sequence $(f_n)_{n \in \mathbb{N}}$ of smooth centred isotropic Gaussian fields on $\mathbb{S}^d$ with covariance kernels
\[
  K(x, y)  =  \cos^n( d_{\mathbb{S}^d}(x,y)) = \langle x , y \rangle^n   ,   
\]
where $d_{\mathbb{S}^d}(\cdot, \cdot)$ denotes the spherical distance; it is easy to check that each $f_n$ satisfies Condition~\ref{cond:non_degeneracy}. The sequence $f_n$ converges to a local limit on the scale $s_n = 1/\sqrt{n}$, in the sense that for any $x_0 \in \mathbb{S}^{d}$ the rescaled field 
\begin{equation}
\label{e:kostlan_scaling}   f ( \exp_{x_0}(x/ \sqrt{n} )  ) \ , \quad x \in \mathbb{R}^d 
\end{equation}
 converges on compact sets to the smooth stationary Gaussian field on $\mathbb{R}^d$ with covariance $\kappa(x) = e^{-\|x-y\|^{2}/2}$; here $\exp_{x_0}:\R^{d} \to \mathbb{S}^{d}$ denotes the exponential map based at~$x_0$. The Kostlan ensemble is a natural model for random homogeneous polynomials (see \cite{kostlan_93, kostlan_02}), and its level sets have been the focus of recent study \cite{bmw_17}. Its local limit is known as the \textit{Bargmann-Fock} field.

\begin{corollary}[Strong mixing for the Kostlan ensemble]
\label{c:mix_kos}
For each pair of disjoint stratified sets $B_1, B_2 \subset \mathbb{S}^d$ that are contained in an open hemisphere, there exist $c_1, c_2 > 0$ such that, for each $n \ge 1$,
\[    \alpha_{n;\text{top}}(B_1, B_2) \le c_1 e^{- c_2 n }  ,\]
where $\alpha_{n;\text{top}}$ denotes the `topological' mixing coefficient \eqref{e:alpha_gen} for the field $f_n$.
\end{corollary}

\begin{remark}
Since $f_n$ are homogeneous polynomials, they are naturally defined on the real projective space rather than the sphere, which makes it natural to restrict $B_1$ and $B_2$ to be contained in an open hemisphere. Indeed, $f_n$ is degenerate at antipodal points.
\end{remark}

The lower concentration result in Corollary \ref{c:con_euc} can also be generalised to the setting of sequences of Gaussian fields on manifolds; again we focus just on the Kostlan ensemble $(f_n)_{n \in \mathbb{N}}$ on $\mathbb{S}^d$. We define a \textit{topological count} $N = N_n(B)$ for $f_n$ analogously to in Section \ref{s:intro}, after substituting affine stratified sets $B \subset \mathbb{R}^d$ with general stratified sets $B \subset \mathbb{S}^d$; these counts are now indexed by $B \subset \mathbb{S}^d$ and $n \in \mathbb{N}$. A topological count $N$ is called \textit{super-additive} if \eqref{e:super_additive} holds for each~$N_n$. We say that a topological count $N$ satisfies a \textit{law of large numbers} if there exists a $c_N > 0$ such that, for every stratified set $B \subset \mathbb{S}^d$, as $n \to \infty$,
\begin{equation}
\label{e:lln_kos}
\frac{N_n(B)}{n^{d/2} \,  \textup{Vol}(B)} \to c_N \quad \text{in probability} ;
\end{equation}
the scale $n^{d/2}$ can be understood as the natural volume scaling induced by the rate $s_n = 1/\sqrt{n}$ at which the Kostlan ensemble converges to a local limit in \eqref{e:kostlan_scaling}.

\begin{corollary}[Lower concentration for topological counts of the Kostlan ensemble]
\label{c:con_kos}
Let $N$ denote a super-additive topological count that satisfies a law of large numbers \eqref{e:lln_kos} with limiting constant $c_N > 0$. Then for every stratified set $B \subset \mathbb{S}^d$ and every $\eps > 0$, there exist $c_1, c_2 > 0$ such that, for every $n \ge 1$,
\begin{equation}
\label{e:con_kos}
  \mathbb{P}\left[  \frac{N_n(B)}{n^{d/2} \, \textup{Vol}(B)}  \le c_N  - \eps  \right] \le c_1 e^{-c_2 n^{d/(d+2)} } .
  \end{equation}
\end{corollary}

In particular, taking $B=\mathbb{S}^d$ with its trivial stratification $\calF = \{\mathbb{S}^d\}$, the conclusion of Corollary \ref{c:con_kos} is true for $N_n$ the number of connected components of $\{f_n > 0\}$ or $\{f_n = 0\}$ on the sphere $\mathbb{S}^d$ (see \cite{ns_16} for a proof of the law of large numbers for $N_n$).

\subsubsection{Decorrelation for topological counts}

In the classical theory of strong mixing, a major application of mixing bounds is to prove central limit theorems (CLTs) (see, e.g., \cite{doukhan_94, ll_96, rosenblatt_56}). Although establishing CLTs for topological counts is beyond the scope of this work, we illustrate here how mixing bounds can be used to deduce the `decorrelation' of topological counts, a key intermediate step in proving a CLT. 

\smallskip
For simplicity we return to the Euclidean setting of Section \ref{s:intro}. We say that a topological count $N$ has a \textit{finite two-plus-delta moment} on an affine stratified set $B \subset \mathbb{R}^d$ if there exist $\delta,c > 0$ such that
\begin{equation}
\label{e:two_plus}   \E [ N(B)^{2+\delta} ] <   c  < \infty . 
\end{equation}
Although the finiteness of two-plus-delta moments is not known for the topological counts discussed in Section~\ref{s:intro} (except in the one-dimensional case), in principle one can bound \eqref{e:two_plus} by the purely local quantity
\[ \mathbb{E}[ ( \# \text{ of critical points of $f$ in $ B $} )^{2+\delta}   ] ,   \]
which we suspect is finite in great generality.

\begin{corollary}[Decorrelation for topological counts]
\label{c:dec_euc}
Fix affine stratified sets $B_1, B_2 \subset \mathbb{R}^d$ and suppose that $N_1$ and $N_2$ are topological counts that have finite two-plus-delta moments \eqref{e:two_plus} on $B_1$ and $B_2$ with constants $\delta, c > 0$. Then 
\begin{equation}
\label{e:dec_euc}
 \cov \big( N_1(B_1), N_2(B_2)  \big)   \le  8 \, c^{2/(2+\delta)} \, \alpha_{\text{top}}(B_1, B_2)^{\delta/(2+\delta)} .
  \end{equation}
\end{corollary}

We expect that standard methods (i.e.\ \cite{ekstrom_14, rosenblatt_56}) should allow one to deduce, from Corollary~\ref{c:dec_euc}, a CLT for rescaled topological counts that satisfy a law of large numbers whenever strong enough two-plus-delta moment bounds can be established, at least as long as $\kappa(x)$ decays at a high enough polynomial rate (with the polynomial exponent depending on $\delta$). 

\subsubsection{Positive association for increasing topological events}

Recall that a random vector is said to be `positively associated' if increasing events (or equivalently decreasing events) are positively correlated. To state an analogous property for continuous random fields some care must be taken to specify an appropriate class of increasing events, and here we restrict the discussion to topological events. An important example of topological events that are increasing are crossing events for the excursion set $\{f > 0\}$ (but \textit{not} crossing events for the level set $\{f = 0\}$), and the fact that crossing events are positively correlated is crucial in the analysis of level set percolation \cite{bg_16, bmw_17, mv_18, rv_17b}.

\smallskip
In the setting of Gaussian fields, it is known that the class of increasing topological events on a stratified set are positively correlated if and only if the covariance kernel $K$ is positive. The standard approach is to invoke a classical result that (finite-dimensional) Gaussian vectors are positively associated if and only if they are positively correlated \cite{pitt_82}, and then to apply an approximation argument (see \cite{rv_17a}). Here we deduce, directly from our exact formula, a quantitative version of this result, whose proof is immediate from Theorem \ref{t:main} and the observation in Remark \ref{r:piv_mon}.

\begin{corollary}[Positive associations]
\label{c:fkg}
Let $A_1$ and $A_2$  be topological events on stratified sets $B_1$ and $B_2$, and suppose that $A_1$ and $A_2$ are both increasing. Then
\begin{equation}
\label{e:pa}
\prob\brb{A_1 \cap A_2}-\prob\brb{A_1}\prob\brb{A_2}=\int_{B_1\times B_2} K(x,y) \, \d\pi^+(x,y) ,
\end{equation}
where $\d \pi^+$ is the measure defined in Definition \ref{d:pivotal_measure}. In particular, $A_1$ and $A_2$ are positively correlated if $K|_{B_1 \times B_2} \ge 0$.
\end{corollary}

The fact that positive associations fails in general if a Gaussian field is not positively correlated is a serious limitation to many applications; for example, the current theory of level set percolation for Gaussian fields fails more or less completely unless $K \ge 0$ (see however \cite{bg_17} for recent progress in this direction). One advantage of \eqref{e:pa} is that the failure of positive associations can be \textit{quantified}, which gives hope that the errors that arise might be controllable. 

\subsubsection{The Harris criterion}
\label{s:hc}
Lastly, we present an informal discussion of the `Harris criterion'~(HC), demonstrating in particular that Theorem~\ref{t:main} can be used to give an alternative derivation of this criterion. 

\smallskip
In its original formulation (see, e.g., \cite{w_84}), the HC was a heuristic to determine whether long-range correlations influence the large-scale connectivity of discrete critical percolation models. Translated to the setting of Gaussian fields on $\mathbb{R}^d$ (see \cite{bs_07}), the HC claims that the connectivity of the level set of smooth centred Gaussian fields will, at the \textit{critical level} $\ell_c \le 0$ (known to be zero if $d=2$, but believed to be strictly negative if $d \ge 3$), be well-described on large scales by critical (Bernoulli) percolation (the `percolation hypothesis') if and only if
\begin{equation}
\label{e:harris}
s^{2/\nu - 2d}  \int_{B_s \times B_s} \kappa(x-y) \, dx dy \to 0 \quad \text{as } s \to \infty,
\end{equation}
where $B_s$ denotes the ball of radius $s$ centred at the origin, and $\nu$ is the \textit{correlation length exponent} of critical percolation, widely believed to be universal and satisfy
 \[    \nu = \begin{cases}
 4/3, & d = 2, \\
 \in (1/2, 1)  , &  d=3, 4, 5,\\
   1/2, & d \ge 6.\\
 \end{cases}  \]
In the positively-correlated case $\kappa \ge 0$, \eqref{e:harris} is roughly equivalent to demanding that $\kappa$ has polynomial decay with exponent at least $2 / \nu$. The original argument of Harris (as translated to our setting in \cite{bs_07}) goes as follows. Define
\[   m_s = \frac{1}{|B_s|} \int_{ B_s} f(x) \, dx  \]
to be the average value of $f$ on the ball $B_s$. The fluctuations of $m_s$ are of order
\begin{equation}
\label{e:mr}
 \sqrt{ \mathbb{E}[m_s^2] } =  \frac{1}{|B_s|} \left(   \int_{B_s \times B_s }  \kappa(x-y) \, dx dy \right)^{1/2}  .  
 \end{equation}
Recall now that the behaviour of critical (and near-critical) percolation follows a set of power-laws with certain \textit{universal exponents}, one of which is the \textit{correlation length exponent} $\nu$. Roughly speaking, this claims that the connectivity of percolation with probability $p \in [0, 1]$ closely approximates the connectivity of critical percolation on the ball $B_s$ as long as $|p - p_c| \ll s^{-1/\nu}$, where $p_c$ is the critical probability. Under the assumption that $f$ can be replaced by $m_s + f$ on $B_s$, the `percolation hypothesis' therefore generates a contradiction unless $m_s \ll s^{-1/\nu}$, and combining with \eqref{e:mr} gives \eqref{e:harris}. Note that the HC should really be understood as a \textit{necessary condition} for the `percolation hypothesis', since the argument assumes the `percolation hypothesis' and derives a contradiction.
 
 \smallskip
We now demonstrate that Theorem~\ref{t:main} yields an alternative criterion, more or less equivalent to \eqref{e:harris}, that we claim is also a necessary condition for the `percolation hypothesis'. Fix a pair of disjoint boxes $B_1, B_2 \subset \mathbb{R}^d$ and, for each $s \ge 1$ and $i \in \{1,2\}$, let $A_i^s$ denote the crossing events for the critical level set in $sB_i$. Note that pivotal points for crossing events roughly correspond to \textit{four-arm saddles at distance} $s$, i.e.\ saddle points $x$ such that all four arms of the level set $\{ f = f(x) \}$ hit the ball of radius $s$ around $x$. Putting this approximation into Theorem~\ref{t:main}, we deduce that
\begin{align*}
 \prob\brb{A_1^s \cap A_2^s}-\prob\brb{A_1^s}\prob\brb{A_2^s} &\approx  c_\kappa \int_{ sB_1 \times sB_2} \kappa(x-y)  I_{s}(x, y)  \, dx dy \\
& \approx c_\kappa   \, I_{s}(0)^2 \, \int_{sB_1 \times sB_2} \kappa(x-y) \, dx dy   . 
\end{align*}
where $I_s$ denotes the intensity of four-arm saddles at distance $s$, and where in the last step we used stationarity and an (unjustified) factorisation of this intensity. Consider now the universal exponent $\zeta_4$ that is believed to describe the decay of the probability of critical `four-arm' events for all percolation models. If the `percolation hypothesis' is true, then $I_{s}(0) \approx s^{-\zeta_4}$, and since under the `percolation hypothesis' the events $A_1^s $ and $A_2^s $ decorrelate, we end up with the following criterion for this hypothesis:
  \begin{equation}
  \label{e:harris2}
    s^{- 2 \zeta_4}  \int_{ sB_1 \times sB_2} \kappa(x-y)    \, dx dy   \to 0 \quad \text{as } s \to \infty .
    \end{equation}
  To compare to \eqref{e:harris}, recall that by the `Kesten scaling relations' \cite{kesten_87} $\zeta_4  = d - 1/\nu$, and so the exponents $2/\nu - 2d$ and $ - 2 \zeta_4$ in \eqref{e:harris} and \eqref{e:harris2} match. The only difference is the domain of integration, but as $s \to \infty$ this difference is negligible under mild assumptions on the decay of covariance.

\subsection{Proof sketch}
\label{ss:proof_sketch}

Theorem \ref{t:main} can be considered as a generalisation to topological events of a simple formula, essentially due to Piterbarg \cite{piterbarg}, that gives a covariance formula for finite-dimensional Gaussian vectors. This lemma is both the inspiration for Theorem \ref{t:main}, and also one of the key ingredients in the proof. We state Piterbarg's formula in the simplest case of standard Gaussian vectors, since this is all that we need, but a similar statement exists for general non-degenerate Gaussian vectors; for completeness, we give the proof in Appendix \ref{s:Piterbarg proof}.

\begin{lemma}[Piterbarg's formula; see {\cite[Theorem 1.4]{piterbarg}}]\label{l:piterbarg} 
For each $t \in [0, 1]$, let $X_t$ and $Y_t$ be jointly Gaussian vectors in $\R^m$, not necessarily centred, whose covariance matrix is
\[
\begin{pmatrix}
I & t I \\
t I & I
\end{pmatrix};
\]
that is, $\cov(X_{t,i},X_{t,j})=\cov(Y_{t,i},Y_{t,j})=\delta_{i,j}$ and $\cov(X_{t,i},Y_{t,j})=t\delta_{i,j}$. Let $\gamma_t(x,y)$ denote the density of $Z_t=(X_t,Y_t)\in \R^{2m}$. Let $A$ and $B$ be domains in $\R^m$ whose boundaries are piecewise smooth, and which have surface areas, inside the ball of radius $R$, that grow at most polynomially in $R$. Denote by $\nu_A$ and $\nu_B$ the outward unit normal vectors on the boundaries of $A$ and $B$ respectively. Then $\prob\left[Z_t \in A\times B\right]$ is differentiable in $t \in (0, 1)$, and
\[
\frac{d}{dt}\prob\left[Z_t\in A\times B\right]=\int_{\partial A\times \partial B}\langle\nu_A(x),\nu_B(y)\rangle \gamma_t(x,y) \, \d x\, \d y ,
\]
where by $\int \d x\, \d y$ we understand integration with respect to the natural $m-1$ dimensional measures on $\partial A$ and $\partial B$ respectively.

In particular, if $X$ denotes an arbitrary translation of a standard Gaussian vector in $\R^m$, then
\[
\prob[X  \in A\cap B]-\prob[X  \in A]\prob[X  \in B]=\int_0^1\int_{\partial A\times \partial B}\langle\nu_A(x),\nu_B(y)\rangle \gamma_t(x,y) \, \d x \, \d y  \, \d t; \]
the integral converges since the integral $\int_0^s dt$ on the right-hand side exists for all $s<1$, and converges as $s\to 1$ to the left-hand side. 
\end{lemma}

Let us now give a brief sketch of the proof of Theorem \ref{t:main}, showing how Piterbarg's formula plays an essential role. We begin by considering the case of finite-dimensional Gaussian fields, i.e.\ the case in which~$f$ is a Gaussian vector in a finite-dimensional space of continuous functions~$V$ (see Proposition~\ref{p:main}). More precisely, we fix $\langle\cdot,\cdot\rangle$ a scalar product on $V$ and take $f$ to be a translation of the standard Gaussian vector in~$V$. The scalar product also induces a volume measure $du$ on $V$, and allows us to identify~$V$ with $\R^{\text{dim}(V)}$ up to isometries. Hence, we can apply Piterbarg's formula in $V$ and deduce that
\begin{equation}
\label{e:sketch1}
\frac{d}{dt}\prob\left[f_t\in \hat{A}_1\times \hat{A}_2\right]
=
\int_{\partial \hat{A}_1\times \partial \hat{A}_2}\langle\nu_{A_1}(u_1),\nu_{A_2}(u_2)\rangle \gamma_t(u_1,u_2) \, \d u_1  \d u_2\, ,
\end{equation}
where $f_t = (f_t^1, f_t^2)$ has covariance $\left(\begin{matrix} I & t I\\ tI & I\end{matrix}\right)$ in orthonormal coordinates of $V\times V$ equipped with the product scalar product (this coincides with the definition of $f_t$ in \eqref{e:inter}).

\smallskip
The next step is to analyse the boundaries of $\hat{A}_i$. The path $(f^i_t)_{t \in [0,1]}$ is a generic deformation of~$f$. By standard arguments in Morse theory, along this deformation the topology of the set $\{f^i_t\geq 0\}$ changes only when $f^i_t$ passes through a non-degenerate critical point at level $0$ (which can cause $f_t^i$ to either enter or exit $\hat{A}_i$); if such a change in topology occurs we say that this critical point is \textit{pivotal} for the event $\hat{A}_i$ and the function~$f_t^i$. We will see (in Lemma \ref{l:boundary}) that, if we exclude a subset $E \subset \partial\hat{A}_i$ of positive codimension containing the functions with multiple stratified critical points at level $0$, we can define a surjection
\[
\Xi:\partial\hat{A}_i\setminus E\twoheadrightarrow B_i ,
\]
that induces submersions on each stratum of $B_i$, by associating to each $u_i\in\partial\hat{A}_i\setminus E$ its unique critical point at level $0$. The fibre $\Xi^{-1}(x_i)$ is an open subset of the subspace of functions for which $x_i$ is a stratified critical point at level $0$. We will see that it is equal to $\pivx[i]$ up to a negligible set.

\smallskip
Using the map $\Xi$, the coarea formula allows us to switch from an integral over $\partial \hat{A}_1 \times \partial \hat{A}_2  \subset V$ to a sum of integrals over pairs of faces of $B_1$ and $B_2$. We obtain that \eqref{e:sketch1} is equal to 
\[
\sum_{F_1 \in \calF_1, F_2 \in \calF_2} \int_{F_1\times F_2} 
\bigg( \int_{\Xi^{-1}(x_1)\times\Xi^{-1}(x_2)} \frac{\langle\nu_{\partial\hat{A}_1}(u_1),\nu_{\partial\hat{A}_2}(u_2)\rangle}{\textup{Jac}^\perp_{F_1}(u_1)\textup{Jac}^\perp_{F_2}(u_2)}\gamma_t(u_1,u_2) \, du_1 du_2 \bigg)
 \dv_{F_1}(x_1) \dv_{F_2}(x_2)
\] 
where, in the inner integral, the measures $du_i$ are the natural volume measures on the fibres of $\Xi^{-1}(x_i)$ and the terms $\textup{Jac}^\perp_{F_i}(u_i)$ are the normal Jacobians of $\Xi$ at $u_i$.

\smallskip
We then turn our attention to the unit normal vectors in the integrand (see Lemma \ref{l:unit_normal}). Consider $u_i \in  \partial\hat{A}_i\setminus E$ such that $\Xi(u_i)=x_i$. Since $x_i$ is the only place at which the topology of $\{u_i\geq 0\}$ can change by infinitesimal perturbations, $T_{u_i}\partial\hat{A}_i$ is the subspace of functions $v\in V$ such that $v(x_i)=0$. Since $K$ is the reproducing kernel of $V$, $K(x_i,\cdot)$ is orthogonal to $T_{u_i}\partial\hat{A}_i$. Thus,
\[
\langle\nu_{\partial\hat{A}_1}(u_1),\nu_{\partial\hat{A}_2}(u_2)\rangle=\pm\frac{K(x_1,x_2)}{\sqrt{K(x_1,x_1)K(x_2,x_2)}},
\]
where the sign depends on whether a small positive perturbation of $u_i$ at $x_i$ makes $u_i$ enter or exit $\hat{A}_i$.

\smallskip
Finally, in Lemmas \ref{l:fibration} and \ref{l:simplifying_the_density} we (i) compute the Jacobian of $\Xi$ at $u_i\in\Xi^{-1}(x_i)$ and (ii) reinterpret the integral over $\Xi(x_1)^{-1}\times\Xi^{-1}(x_2)$ as an expectation in $f_t=(f_t^1,f_t^2)$ conditioned on the fact that for $i=1,2$, $x_i$ is a critical point of $f_t^i$ at level $0$, containing the indicators that the $x_i$ are pivotal for $f_t^i$. This process involves some standard computations of Jacobians of evaluation maps and a careful study of the relations between the different metrics on the spaces $V\times V$ and $F_1\times F_2$. As a result, we get exactly the term which appears in the definition of the pivotal intensity functions (see~\eqref{e:intensity}), namely
\[
\frac{\langle\nu_{\partial\hat{A}_1}(u_1),\nu_{\partial\hat{A}_2}(u_2)\rangle}{\textup{Jac}^\perp_{F_1}(u_1)\textup{Jac}^\perp_{F_2}(u_2)}=\pm K(x_1,x_2)\prod_{i=1,2}\left|\det\br{H_{x_i}^{F_i}u_i}\right|,
\]
which completes the proof in the finite-dimensional case.

\smallskip 
To extend Theorem \ref{t:main} to the general case, it remains only to argue that $f$ can always be approximated by finite-dimensional fields and that we can successfully pass to the limit in the covariance formula. This latter step is mainly technical, and requires us to show, among other things, that the boundary of $\Xi^{-1}(x_i)=\piv_{x_i}(\hat{A}_i)$ is a null set for the field $f$ conditioned on the existence of critical points at $x_1$ and $x_2$.

%%%%

\section{Heart of the proof: the finite-dimensional case}\label{s:heart_of_the_proof}

In this section we state and prove a reinterpretation of our covariance formula in the case where the space $V$ is finite-dimensional (see Proposition \ref{p:main}). As discussed in the proof sketch above, we prove this proposition by applying Piterbarg's formula (Lemma \ref{l:piterbarg}) and then obtaining a rather explicit description of boundaries of topological events (see Lemma \ref{l:boundary}).

\smallskip
Throughout this section, and indeed for the remainder of the paper, $(B, \calF)$ denotes an arbitrary stratified set of $M$.

\subsection{Restating the formula in terms of the discriminant}

In this subsection we state the finite-dimensional version of the formula (Proposition \ref{p:main}). For this we introduce an alternative notion of `pivotal sets' defined in terms of the `discriminant'.\footnote{In fact, in the cases that matter to us, this alternative notion of `pivotal sets' coincides with that of Definition \ref{d:pivotal_sets} up to null sets. See Remark \ref{rem:piv_two}.}

\begin{definition}[Critical points]
\label{d:critical_points}
Let $u\in C^1(M)$. A \textit{stratified critical point} of $u$ in $B$ is a point $x\in B$ such that $d_xu|_{F}=0$, where $F\in\calF$ is the (unique) stratum containing $x$. When there is no ambiguity we will refer to stratified critical points as critical points for brevity. The \textit{level} of a critical point refers to its critical value.

Assume now that $u\in C^2(M)$. A stratified critical point $x$ of $u$ is said to be a \textit{non-degenerate} if (i) $H_x^Fu$ is non-degenerate, and (ii) for each $F'\in\calF$ such that $F'>F$, $d_xu$ does not vanish on $T_xF'|_F$ (see Definition~\ref{d:stratification}). Roughly speaking (ii) means that $d_x u$ vanishes on $T_xF$ but not on tangent spaces to higher dimensional strata. Note that we define non-degeneracy in terms of the generalised tangent bundle $T_xF'|_F$; this is since all strata are open and disjoint, so $T_xF'$ is not defined. 
\end{definition}

In the following definitions $V\subset C^2(M)$ denotes an arbitrary linear subspace (i.e.\ not necessarily finite-dimensional). To define the discriminant, it will be convenient to introduce the following subsets of $V$:

\begin{notation} 
For each $x \in M$, $\Vx \subset V$ denotes the linear subspace of $u\in V$ such that $u(x)=0$. Moreover, $\Vxd$ denotes the linear subspace of $\Vx$ such that also $d_xu|_F=0$, where $F$ is the (unique) stratum containing~$x$; in other words, $\Vxd$ contains the functions that possess a stratified critical point at $x \in B$ at level~$0$. Similarly, for each $F \in \calF$, $\VFd = \cup_{x \in F} \Vxd$ denotes the functions that possess a stratified critical point on $F$ at level~$0$.
\end{notation}

\begin{definition}[Discriminant]
\label{d:discriminant}
The \textit{discriminant} associated to $B$ in $V$ is the set $\mathfrak{D}_B(V) = \cup_{F \in \calF} \VFd$, that is, the set of functions that possess a stratified critical point in $B$ at level~$0$. For each $u\in V\setminus \mathfrak{D}_B(V)$, the \textit{$B$-discriminant class} of $u$ (in $V$), written as $[u]_{(B,V)}$ is the connected component of $V\setminus\mathfrak{D}_B(V)$ containing $u$. By Lemma \ref{l:discriminant_is_closed} the discriminant is closed; since $C^1(M)$ is separable, the number of classes is therefore at most countable.  We will denote by $\widetilde{\sigma}_{\textup{discr}}\br{B,V}$ the complete $\sigma$-algebra of all collections  of $B$-discriminant classes.
\end{definition}

Before defining the alternate notion of `pivotal sets' in terms of the discriminant, we introduce further subsets of $\Vxd$ and $\VFd$ defined above; as we verify later (see Proposition~\ref{p:transversality}), these subsets are of full measure:

\begin{notation}\label{not:function_spaces}
For each $x \in M$, $\tVxd \subset \Vxd$ denotes the set of $u\in V$ such that $x$ is a non-degenerate stratified critical point at level $0$ and there are no other stratified critical points in $B$ at this level. Similarly, for each $F \in \calF$, $\tVFd = \cup_{x \in F} \tVxd$ denotes the subset of $\VFd$ consisting of functions that have a non-degenerate stratified critical point on stratum $F$ at level $0$, and no other stratified critical points in $B$ at this level.
\end{notation}

\begin{definition}[`Pivotal sets' in terms of the discriminant]\label{d:discriminant_pivotal_sets}
Let $\tilde{A}$ be an element of $\widetilde{\sigma}_{\textup{discr}}\br{B,V}$. We denote by $\hat{A}\subset V$ the set of functions whose $B$-discriminant class is in $\tilde{A}$, and by $\hat{\sigma}_{\textup{discr}}(B,V)$ the $\sigma$-algebra of all possible sets $\hat{A}$ of this type. For $\hat{A}\in \hat{\sigma}_{\textup{discr}}(B,V)$, we define the `pivotal sets' $\tpiv_x(\hat{A})=\partial\hat{A}\cap \tVxd$ and $\tpiv_F(\hat{A})=\partial\hat{A}\cap \tVFd$; note that these are subsets of the discriminant~$\mathfrak{D}_B(V)$. For each $\sigma\in\{+,-\}$, let $\tpivxs$ be the set of $u\in \tpivx$ such that there exists $h\in V$ with $h(x)>0$ such that, for all small enough values of $\eta>0$, $u+\sigma\eta h\in\hat{A}$.
\end{definition}

Finally, we introduce the key conditions on the space $V$:

\begin{condition}\label{cond:amplitude_4}
For each distinct $x,y\in M$, let $V'_{x,y} \subset V$ denotes the set of $u\in V$ such that $(u(x), d_xu, u(y), d_yu)$ vanishes. Then the following map is surjective:
\begin{align*}
V'_{x,y}&\rightarrow\textup{Sym}^2\br{T^*_yM}\\
u&\mapsto H_yu\, .
\end{align*}
\end{condition}

\begin{condition}\label{cond:amplitude_5}
For each distinct $x,y\in M$, the following map is surjective:
\begin{align*}
V&\rightarrow \R\times T_x^*M\times\R\times T_y^*M\\
u&\mapsto (u(x),d_x u,u(y),d_y u)\, .
\end{align*}
\end{condition}

\begin{remark}
\label{rem: finite dim}
For every smooth $M$ there exists a finite-dimensional subspace $V\subset C^\infty(M)$ satisfying Conditions \ref{cond:amplitude_4} and \ref{cond:amplitude_5}. Indeed, given a smooth mapping $G:M\rightarrow\R^N$ for some $N\in\N$, the coordinates of $G$ generate an $N$-dimensional subspace  of $C^\infty(M)$ which we denote by $V^G$. For any distinct $x,y\in M$, the set of $G$ such that $V^G$ does not satisfy Conditions \ref{cond:amplitude_4} and \ref{cond:amplitude_5} at $x$ and $y$ has codimension arbitrarily large as $N\rightarrow \infty$. Therefore, by the multijet transversality theorem (see Theorem 4.13, Chapter II of \cite{golubitsky_guillemin}), applied to the multijet $(x,y)\mapsto(j^1G(x),j^2G(y))$, the set of $G$ such that $V^G$ satisfies Conditions~\ref{cond:amplitude_4} and \ref{cond:amplitude_5} is a residual subset of $C^\infty(M,\R^N)$ for sufficiently large $N$. In particular, such spaces exist.
\end{remark}

We are now ready to present our finite-dimensional restatement of the covariance formula:

\begin{proposition}\label{p:main}
Recall the notation introduced in Section \ref{ss:main_formula}. Let $V$ be a finite-dimensional subspace of $C^2(M)$ that satisfies Conditions \ref{cond:amplitude_4} and \ref{cond:amplitude_5}, and assume that the support of $f$ is exactly~$V$, so that $f$ is a non-degenerate Gaussian vector in $V$. Let $(B_1, \calF_1)$ and $(B_2, \calF_2)$ be stratified sets of $M$. For each $i \in \{1,2\}$, let $\hat{A}_i\in\hat{\sigma}_{\textup{discr}}\br{B_i,V}$ and let $A_i$ be the event $\{f\in\hat{A}_i\}$. Then, for each $t\in[0,1)$,
\begin{align*}
& \frac{d}{dt}\prob\brb{f_t\in\hat{A}_1\times\hat{A}_2}= \sum_{\substack{\sigma_1,\sigma_2\in\{-,+\}}}  \sum_{F_1\in\calF_1,\ F_2\in\calF_2}\int_{F_1\times F_2}K(x_1,x_2) \times \gamma_{t;x_1,x_2}(0) \\
& \ \times \sigma_1\sigma_2 \, \E_{t;x_1,x_2} \! \! \brb{\un_{\tpivxs[1]\times\tpivxs[2]}(f_t^1,f_t^2)
\left|\det\br{H_{x_1}^{F_1}f_t^1}\right| \!
\left|\det\br{H_{x_2}^{F_2}f_t^2}\right|}  \! \dv_{F_1}(x_1)\dv_{F_2}(x_2) .
\end{align*}
\end{proposition}

\subsection{Proof of Proposition \ref{p:main}}\label{ss:proof_of_main_proposition}
Throughout this section we assume that $V$, $f$ and $\hat{A}_i$ are as in the statement of Proposition \ref{p:main}, in particular $V$ is finite-dimensional and satisfies Conditions~\ref{cond:amplitude_4} and \ref{cond:amplitude_5} (although all the notation that is introduced applies equally to arbitrary linear subspaces $V$ of $C^2(M)$). We continue to use $(B, \calF)$ to denote an arbitrary stratified set of $M$, and we also define an arbitrary $\hat{A}\in\hat{\sigma}_{\textup{discr}}\br{B,V}$. We rely on four technical lemmas (namely Lemmas \ref{l:boundary}--\ref{l:unit_normal} and \ref{l:simplifying_the_density}), whose proofs are deferred to Section \ref{s:technical lemmas}. 

\smallskip
The starting point of the proof is to apply Piterbarg's formula to the events $\hat{A}_1$ and $\hat{A}_2$; for this we need to study the regularity of their boundaries. The structure of $\partial\hat{A}_i$ is described by the following lemma:

\begin{lemma}\label{l:boundary}
For each $F \in \calF$, the set $\tpivF$ (from Definition \ref{d:discriminant_pivotal_sets}) is a smooth (immersed) conical hypersurface of~$V$. If $x$ is the unique level-$0$ stratified critical point of some $u\in\tpivF$, then $T_u\tpivF=V_x$ (see Notation \ref{not:function_spaces}). Moreover, there exists a subset $E\subset\partial\hat{A}$ of zero $N-1$ dimensional Hausdorff measure such that
\[\partial\hat{A}=E\cup\bigsqcup_{F\in\calF}\tpivF .\]
\end{lemma}

\begin{figure}[h]
\includegraphics[width=0.95\textwidth]{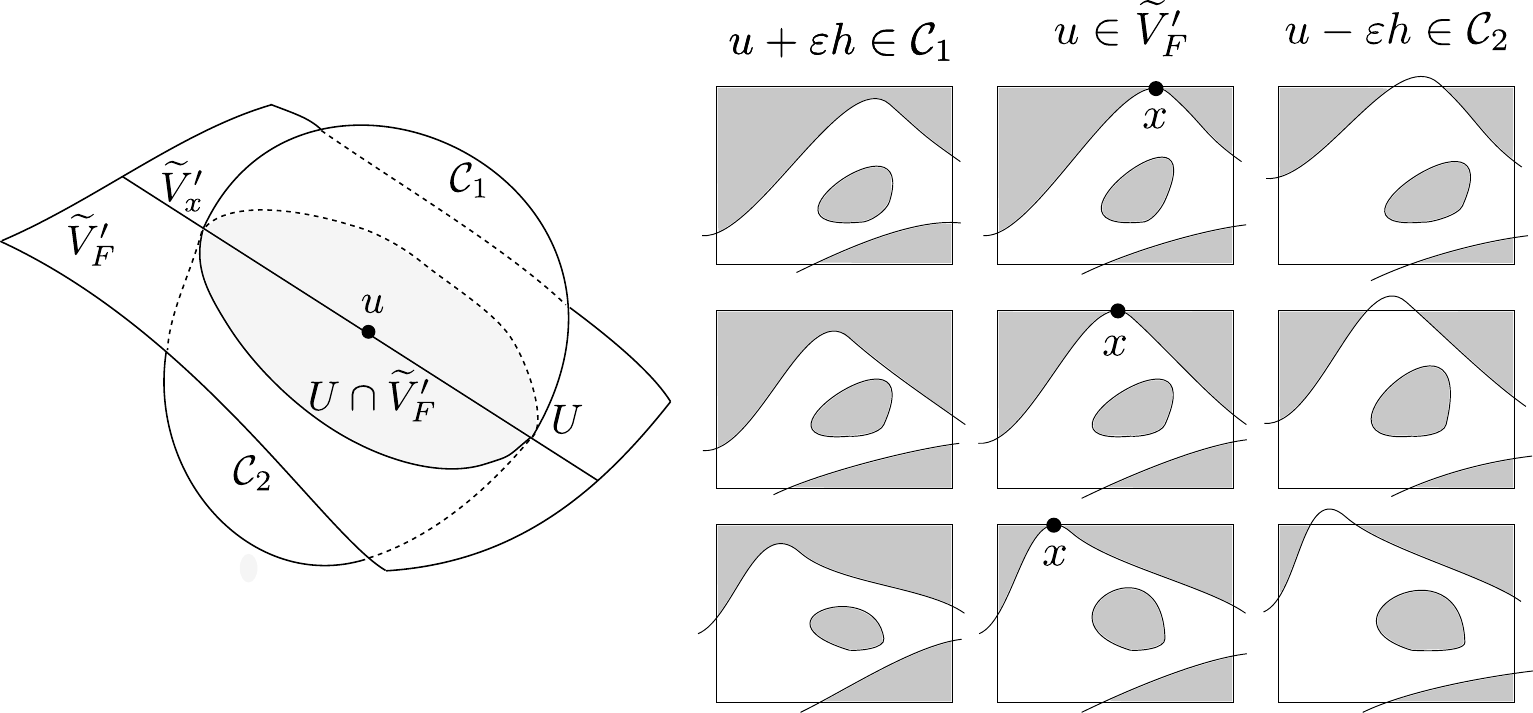}
\caption{Outside of a null set, the boundary of $\hat{A}$ is a hypersurface $\tVFd$ which is covered by the disjoint union over $x\in F$ of the $\tVxd$. Left (functional view): A small neighbourhood $U$ of $u$ in $V$ is split by~$\tVFd$ into two parts, $\mathcal{C}_1$ and $\mathcal{C}_2$, which are inside two different topological classes (one of them belongs to $\hat{A}$ and one does not). Right (spatial view): When $u$ changes continuously within $U\cap \tVFd$, the corresponding level-$0$ stratified critical point $x$ changes continuously within $F$. Central panels shows three functions in $\tVFd$ and their critical points. 
Small perturbations of these functions all belong to the same topological class, for  perturbations positive near the critical point they belong  to $\mathcal{C}_1$ (right panels) and for negative perturbations to $\mathcal{C}_2$ (left panels).  }
\label{fig:two top classes}
\end{figure}

By Lemma \ref{l:boundary}, the boundaries of the sets $\hat{A}_1$ and $\hat{A}_2$ are smooth up to null sets, which implies that their $N-1$ dimensional volume inside any finite ball is finite. Since they are conical, the volume of the boundary inside a ball of radius $R$ is of order $R^{N-1}$, ensuring that Piterbarg's formula applies to these sets.

Now, consider a coordinate system orthonormal with respect to the scalar product $\langle\cdot,\cdot\rangle$ induced by $f$. We typically denote $u=(u^1,\dots,u^N)$ to be the set of coordinates of an element of $V$. For each $t\in[0,1)$, let $\gamma_t:\R^N\times\R^N\rightarrow\R$ be the density of the Gaussian vector with covariance 
\[
\br{\begin{matrix} I_N & tI_N\\ tI_N & I_N\end{matrix}}
\]
and mean $(\mu,\mu)$, where $\mu\in\R^N$ is such that $\E[f]=\sum\mu_i u^i$. This density gives the distribution of $f_t$ as defined in \eqref{e:inter}. Piterbarg's formula (Lemma \ref{l:piterbarg}) implies that
\begin{equation}
\label{eq:piterbarg for f}
\frac{d}{dt}\prob\left[f_t\in \hat{A}_1\times \hat{A}_2\right]=\int_{\partial \hat{A}_1\times \partial \hat{A}_2}
\langle\nu_{\hat{A}_1}(u_1),\nu_{\hat{A}_2}(u_2)\rangle 
\gamma_t(v_1,v_2) \, d\calH^{N-1}(u_1)d\calH^{N-1}(u_2),
\end{equation}
where the integral is taken on the product of the smooth part of the boundaries of $\hat{A}_1$ and $\hat{A}_2$, which are seen as subsets of $\R^N$ through the coordinate system fixed above, and where $\nu_{\hat{A}_1}$ (resp.\ $\nu_{\hat{A}_2}$) is the outward unit normal vector to $\hat{A}_1$ (resp.\ $\hat{A}_2$) defined on the smooth part of its boundary. Applying the expression for the smooth part of the boundary of $\hat{A}_1$ and $\hat{A}_2$ in Lemma \ref{l:boundary}, we have
\begin{equation}
\label{eq:piterbarg with piv}
\begin{aligned}
&\frac{d}{dt}\prob\left[f_t\in \hat{A}_1\times \hat{A}_2\right]
\\
& \quad =
\sum_{F_1\in \calF_1, \, F_2\in \calF_2}
\int_{\tpivF[1]\times \tpivF[2]}
\langle\nu_{\hat{A}_1}(u_1),\nu_{\hat{A}_2}(u_2)\rangle 
\gamma_t(u_1,u_2) \, d\calH^{N-1}(u_1)\calH^{N-1}(u_2)\, .
\end{aligned}
\end{equation}

\smallskip
The next step is to to apply the coarea formula to the integrals in \eqref{eq:piterbarg with piv}. For each $F \in \calF$, let $\Xi_F$ denote the function $\tVFd\rightarrow F$ which maps $u\in \tVFd$ to the unique $x\in F$ such that $u$ has a stratified critical point on $F$ at level $0$. We note that
\[\tpivx=(\Xi_F)^{-1}(x)\cap\tpivF,\]
which means that we can parametrise $\tVFd$ by pairs $(x,u)$ where $x\in F$ and $u\in\tpivx$. The next lemma shows that $\Xi_F$ is a submersion and gives an expression for its normal Jacobian:

\begin{lemma}\label{l:fibration}
For each $F \in \calF$, the map $\Xi_F$ is a submersion. Moreover, for each $u\in\tpivF$, if $x:=\Xi_F(u)$ then the normal Jacobian of $\Xi_F$ at $u$ is
\[J_{F}(u):= \textup{Jac}^\perp\left[\Xi_F\right](u)=\frac{\textup{Jac}^\perp(L_x)}{\left|\det\br{H_x^F u}\right|},\]
where $L_x:\Vx\rightarrow T^*_x F$ denotes the linear operator $u\mapsto d_x u|_F$, and where the determinant is taken in orthonormal coordinates of $T_xF$.
\end{lemma}

Using Lemma \ref{l:fibration} we can apply the coarea formula to the integrals in \eqref{eq:piterbarg with piv}, converting them from integrals over part of the boundary of the events to integrals over the faces of the stratified sets. As a result, each integral in \eqref{eq:piterbarg with piv} can be written as
\begin{equation}
\label{eq:coarea}
\int\limits_{F_1\times F_2}
\Gamma(t;x_1,x_2) \,
\dv_{F_1}(x_1)  \dv_{F_2}(x_2), 
\end{equation}
where, for each $x_1\in F_1$, $x_2\in F_2$ and $t\in[0,1)$,
\begin{equation}
\label{e:G(x,y)}
\Gamma(t;x_1,x_2)=
\int\limits_{\tpivx[1]\times\tpivx[2]}
\frac{\langle\nu_{\hat{A}_1}(u_1),\nu_{\hat{A}_2}(u_2)\rangle\gamma_t(u_1,u_2)}{J_{F_1}(u_1)J_{F_2}(u_2)} \,
\dv_{\Vxd[1]}(u_1)
\dv_{\Vxd[2]}(u_2) , 
\end{equation}
and where $\tpivx[1]\times\tpivx[2]$ is viewed as an open subset of ${\Vxd[1]\times \Vxd[2]}$. Here we have identified the spaces $\Vxd[i]$ with their images in $\R^N$ in the coordinate system fixed previously. The measures  $\dv_{\Vxd[i]}$ are defined as the canonical $N-\dim(F_i)-1$ dimensional volume measures on the affine spaces $\Vxd[i]$ of $\R^N$.

\smallskip
We next interpret the normal vectors in \eqref{e:G(x,y)} in more tractable terms (using the sets from Definition \ref{d:discriminant_pivotal_sets}):

\begin{lemma}\label{l:unit_normal} 
The fibre $\tpivx$ is the disjoint union of the two  subsets $\tpivxp$ and $\tpivxm$. Moreover, for each $\sigma\in\{+,-\}$ and each $u\in\tpivxs$, the outward unit normal vector of $\hat{A}$ at $u$ is
\[
\nu_{\hat{A}}(u)=-\sigma\frac{K(x,\cdot)}{\|K(x,\cdot)\|} = -\sigma\frac{K(x,\cdot)}{\sqrt{K(x,x)}}\,  .
\]
\end{lemma}
\begin{remark}
\label{rem:evaluation}
Since $K$ is the reproducing kernel in $V$, the evaluation map $\mathrm{Ev}_x$ defined by $v\mapsto v(x)$ is equal to the map $v \mapsto \langle v,K(x,\cdot)\rangle$. Hence $K(x,\cdot)$ is orthogonal to $V_x$, and so $\|K(x,\cdot)\|$ can also be interpreted as $\textup{Jac}^\perp(\mathrm{Ev}_x)$, the normal Jacobian of the evaluation operator.
\end{remark}

Since $K$ is the reproducing kernel in $V$, it satisfies $\langle K(x_1,\cdot),K(x_2,\cdot)\rangle=K(x_1,x_2)$. Hence
\begin{equation}
\label{e:scalar product of normals}
\langle\nu_{\hat{A}_1}(u_1),\nu_{\hat{A}_2}(u_2)\rangle
=
\sigma(u_1,u_2)\frac{K(x_1,x_2)}{\|K(x_1,\cdot)\|\,\|K(x_2,\cdot)\|}
=
\sigma(u_1,u_2)\frac{K(x_1,x_2)}{\sqrt{K(x_1,x_1)K(x_2,x_2)}},
\end{equation}
where $\sigma(u_1,u_2)=+$ if either $(u_1,u_2)\in \tpivxp[1]\times\tpivxp[2]$ or $(u_1,u_2)\in \tpivxm[1]\times\tpivxm[2]$, and $\sigma(u_1,u_2)=-$ otherwise. Thus, by Lemma \ref{l:fibration} and \eqref{e:scalar product of normals},
\begin{equation}
\label{eq:gamma2}
\Gamma(t;x_1,x_2)=\int_{\tpivx[1]\times\tpivx[2]}\Upsilon_{x_1,x_2}(u_1,u_2)\gamma_t(u_1,u_2) \,\dv_{\Vxd[1]}(u_1)\dv_{\Vxd[2]}(u_1),
\end{equation}
where
\begin{equation}
\label{eq:Upsilon def}
\Upsilon_{x_1,x_2}(u_1,u_2)=\frac{\sigma(u_1,u_2)K(x_1,x_2)}{\sqrt{K(x_1,x_1)K(x_2,x_2)}}\times \frac{
\left|\det\br{H_{x_1}^{F_1}u_1}\right|
\left|\det\br{H_{x_2}^{F_2}u_2}\right|}
{\textup{Jac}^\perp(L_{x_1})\textup{Jac}^\perp(L_{x_2})},
\end{equation}
and where $L_{x_i}:\Vx\rightarrow T^*_{x_i} F_i$ denotes the linear operator $u\mapsto d_{x_i} u|_{F_i}$. The integral in the definition of $\Gamma$ can be interpreted as a conditional expectation:

\begin{lemma}\label{l:simplifying_the_density}
For each $t\in [0,1)$ and each distinct $x_1\in F_1$ and $x_2\in F_2$, 
\begin{multline*}
\Gamma(t;x_1,x_2)=K(x_1,x_2) \, \gamma_{t;x_1,x_2}(0)\times
\\
\E_{t;x_1,x_2}\brb{\sigma(f_t^1,f_t^2)\un_{\tpivx[1]\times\tpivx[2]}(f_t^1,f_t^2)
\left|\det\br{H_{x_1}^{F_1}f_t^1}\right|
\left|\det\br{H_{x_2}^{F_2}f_t^2}\right|}.
\end{multline*}
\end{lemma}
\noindent Combining \eqref{eq:piterbarg with piv}, \eqref{eq:coarea} and Lemma \ref{l:simplifying_the_density} yields the formula in Proposition \ref{p:main}.

\section{Proof of the auxiliary lemmas}\label{s:technical lemmas}

In this subsection we prove the auxiliary lemmas from Section \ref{s:heart_of_the_proof}, namely Lemmas~\ref{l:boundary}--\ref{l:unit_normal}, and Lemma \ref{l:simplifying_the_density}. While we make use of the notation from Section \ref{s:heart_of_the_proof}, we do not rely on results from that section.

\subsection{Differential topology in the space of functions: Proof of Lemmas \ref{l:boundary}--\ref{l:unit_normal}}
\label{ss:differential_topology_1}

Throughout this section $V$ denotes a linear subspace of $C^2(M)$; moreover, with the exception of the statement of Proposition \ref{p:transversality}, we will assume that $V$ is finite-dimensional and satisfies Conditions~\ref{cond:amplitude_4} and \ref{cond:amplitude_5}. Again we fix an arbitrary stratified set $(B,\calF)$ in $M$ and $\hat{A}\in\hat{\sigma}_{\textup{discr}}(B,V)$.

\smallskip
We begin with a couple of definitions; for the time being we work independently of the choice of $\hat{A}$. Let $F\in\calF$, and recall from Section \ref{s:heart_of_the_proof} the subsets $\tVFd \subset  \VFd \subset V$ and the map $\Xi_F(u)$ which sends $u \in \tVFd$ to its unique non-degenerate stratified critical point at level $0$. Let $\calI_F$ be the set of pairs $(u,x)\in V\times F$ such that $x$ is a stratified critical point of $u$ at level $0$ (so that in fact $u\in\VFd$), and let $\widetilde{\calI}_F$ be the set of pairs $(u,x)\in \calI_F$ such that $x$ is the unique non-degenerate stratified critical point of $u$ at level $0$ (so that $u\in\tVFd$). By Condition \ref{cond:amplitude_5}, the map $(u,x)\mapsto (u(x),d_xu)$ is a submersion on $V\times F$, and so $\calI_F$ is a smooth submanifold of $V\times F$ whose codimension is one plus the dimension of $F$. Moreover, for each $(u,x)\in\calI_F$,
\begin{equation}\label{e:tangent_space_to_I_F}
T_{(u,x)}\calI_F=\left\{(v,\tau)\in V\times T_xF \st v(x)=0,\ d_xv|_F+H_x^Fu(\tau,\cdot)=0\right\} .
\end{equation}
Let $\textup{pr}^1_F:\calI_F\rightarrow V$ and $\textup{pr}^2_F:\calI_F\rightarrow F$ be the projections onto the first and second coordinates. Note that $\VFd = \textup{pr}^1_F\br{{\calI}_F}$ and $\tVFd = \textup{pr}^1_F(\widetilde{\calI}_F)$, and observe also that the map $\Xi_F(u)$ completes the following commutative diagram:
\begin{equation}\label{e:the_diagram}
\begin{tikzcd}[column sep=tiny]
& \widetilde{\calI}_F \ar[dl, "\textup{pr}^1_F" left, near start] \ar[dr, "\textup{pr}^2_F"]
& \\
\tVFd \ar[rr, "\Xi_F" below]
&
& F
\end{tikzcd}
\end{equation}
Lemmas \ref{l:boundary} and \ref{l:fibration} both pertain to elements of this diagram: for Lemma \ref{l:fibration} this is explicitly so, whereas for Lemma \ref{l:boundary} it is since, as we shall see, $\tpivF$ is an open subset of $\tVFd$. In the proof of Lemmas \ref{l:boundary} and \ref{l:fibration}, we use the following proposition (whose proof is postponed until the very end of the subsection):

\begin{proposition}\label{p:transversality}
Let $F\in\calF$. Then the set $\widetilde{\calI}_F$ is open in $\calI_F$ and the set $\tVFd$ is open in~$\mathfrak{D}_B$. Moreover, if $V$ has finite dimension $N\in\N$ and satisfies Conditions \ref{cond:amplitude_4} and \ref{cond:amplitude_5}, then 
\[ \calH^{N-1} \big(\VFd\setminus\tVFd)=0. \]
\end{proposition}

\begin{remark}
Although we only apply Proposition \ref{p:transversality} to finite-dimensional $V$, we state it in full generality so as to clarify which tools are used to prove each point.
\end{remark}
\begin{remark}
Roughly speaking, Proposition \ref{p:transversality} ensures that if the field $f$ is conditioned to have a stratified critical point at level $0$, then a.s.\ this critical point is non-degenerate, and there are no other stratified critical points at level $0$.
\end{remark}

\begin{proof}[Proof of Lemma \ref{l:boundary}]
To show that $\tpivF$ is a smooth (immersed) conical hypersurface of~$V$, we first show that $\tVFd$ is a smooth immersed (although maybe not embedded) hypersurface of~$V$. By Proposition \ref{p:transversality}, $\widetilde{\calI}_F$ is a smooth submanifold of $V\times F$ with the same tangent space as $\calI_F$ at each point. The mapping $\textup{pr}^1_F:\widetilde{\calI}_F\rightarrow V$ is one-to-one, and we claim that it has constant rank. To see this, let us take $(u,x)\in\widetilde{\calI}_F$ and check that 
\[d_{(u,x)}\textup{pr}^1_F\br{T_{(u,x)}\calI_F}=\Vx.\]
 The inclusion $\subset$ is clear by \eqref{e:tangent_space_to_I_F}. For the reverse inclusion, let $v\in V_x$ and define $\lambda=-d_xv|_F$. Since $(u,x)\in\widetilde{I}_F$, $H_x^Fu$ is non-degenerate, and so there exists $\tau\in T_xF$ such that $H_x^Fu(\tau,\cdot)=\lambda$. Therefore, $(v,\tau)\in T_{(u,x)}\calI_F$ and $d_{(u,x)}\textup{pr}^1_F(v,\tau)=v$, which proves the reverse inclusion. To sum up, $\textup{pr}^1_F$ is a mapping of corank one on $\widetilde{\calI}_F$, and so its image $\tVFd$ is a smooth immersed (although maybe not embedded) hypersurface of $V$ with the tangent space
\begin{equation}\label{e:boundary_of_A_4}
T_u\tVFd=d_{(u,x)}\textup{pr}^1_F\br{T_{(u,x)}\calI_F}=V_x .
\end{equation}

Next, we show that $\tpivF$ is open in $\tVFd$. Indeed, by Proposition \ref{p:transversality}, $\tVFd$ is open in $\mathfrak{D}_{B}$. Moreover, $\tVFd$ is a smooth submanifold of $V$, which implies that, for each $u\in\tVFd$, there exists $U\subset V$ containing $u$ such that $(u,U\cap\tVFd,U )\simeq (0,\R^{N-1}\times\{0\}, \R^N)$ and such that $U\cap\mathfrak{D}_{
B}=U\cap\tVFd$. Hence there exist exactly two $B$-discriminant classes $\calC_1$, $\calC_2$ that intersect $U$ and
\begin{equation}\label{e:boundary_of_A_3}
\overline{\calC_1}\cap\tVFd\cap U=\overline{\calC_2}\cap\tVFd\cap U=\tVFd\cap U
\end{equation}
as illustrated in Figure \ref{fig:two top classes}. In particular, if $u\in\tpivF$ then $\tVFd\cap U\subset\tpivF$, and so $\tpivF$ is an open subset of $\tVFd$. 

\smallskip
To sum up, since $\tpivF$ is open in $\tVFd$ and since $\tVFd$ is a smooth (immersed) hypersurface of~$V$, $\tpivF$ is also a smooth (immersed) hypersurface of $V$. Noting also that $\hat{A}$ is conical hence so is $\partial\hat{A}$, and observing moreover that, by \eqref{e:boundary_of_A_4}, $T_u\tpivF=V_x$ for every $(u,x)\in\widetilde{\calI}_F$, we complete the proof of the first two statements of the lemma.

\smallskip
For the third statement of the lemma, we define 
\begin{equation*}
E=\partial\hat{A}\setminus \Big( \bigsqcup_{F\in\calF}\tVFd \Big) .
\end{equation*}
By the definition of $\tpivF = \partial\hat{A}\cap\tVFd$, we have
\begin{equation}\label{e:boundary_of_A_2}
\partial\hat{A}=E\cup\bigsqcup_{F\in\calF}\tpivF .
\end{equation}
Moreover, we claim that $\calH^{N-1}(E)=0$. To see this, observe that $\partial\hat{A}\subset\mathfrak{D}_{B} :=\bigcup_{F\in\calF} \VFd$. Indeed, since the discriminant $\mathfrak{D}_{B}$ is closed (see Lemma \ref{l:discriminant_is_closed}), the $B$-discriminant class of any $u\in V\setminus\mathfrak{D}_{B}$ forms a neighbourhood of $u$; in particular, $u\notin\partial\hat{A}$. Hence we have an alternate expression for $E$:
\[E=\bigsqcup_{F\in\calF}\partial\hat{A}\cap\br{\VFd\setminus\tVFd} .\]
Since by Proposition \ref{p:transversality} the $N-1$ dimensional Hausdorff measure of each term of the union on the right-hand side vanishes, it follows that $\calH^{N-1}(E)=0$.
\end{proof}

\begin{proof}[Proof of Lemma \ref{l:fibration}]
We first show that $\Xi_F$ is a submersion. Let $(u,v)\in T\tVFd$, so that there exist $x\in F$ and $\tau\in T_xF$ such that $((u,x),(v,\tau))\in T\widetilde{\calI}_F$. In particular, by \eqref{e:tangent_space_to_I_F} we have $d_xv|_F+H_x^Fu(\tau,\cdot)=0$. Since $H_x^Fu$ is non-degenerate, $\tau$ is uniquely determined by $v$. More precisely, let $\check{H}_x^Fu$ be the image of $H_x^Fu$ by the canonical isomorphism $\br{T^*F}^{\otimes 2}\simeq\textup{Hom}(T^*F,TF)$. Then $\tau=-\br{\check{H}_x^Fu}^{-1}(d_xv|_F)$. Since the diagram \eqref{e:the_diagram} commutes, we have proven that
\[
d_u\Xi_F(v)=-\br{\check{H}_x^Fu}^{-1}(d_xv) .
\]
By Condition \ref{cond:amplitude_5}, the map $v\mapsto d_xv$ is surjective when restricted to $V_x$. Hence $\Xi_F$ is a submersion, which proves the first statement of the lemma.

\smallskip
Let us now show that the Jacobian of $\Xi_F$ is as claimed in the lemma. Let $g_F^{-1}$ be the metric induced on $T^*F$ by the metric $g_F$ on $TF$. Since $(\check{H}_x^Fu)^{-1}$ is an isomorphism $(T^*_xF,g_{F,x}^{-1})\rightarrow (T_xF,g_{F,x})$, the normal Jacobian of $\Xi_F$ is the product of the Jacobian of  $(\check{H}_x^Fu)^{-1}$ and of the normal Jacobian of the map $L_x:(V_x,\langle\cdot,\cdot\rangle)\rightarrow (T^*_xF,g_{F,x}^{-1})$, defined in the statement of the lemma to be $L_x(v)=d_xv$. Since the first Jacobian is the absolute value of the inverse of $\det(H_x^Fu)$, i.e.\ the determinant of the matrix of the bilinear form $H_x^Fu$ in a $g_{F,x}^{-1}$-orthonormal basis of $T_xF$, the proof is complete.
\end{proof}

\begin{remark}
Although for our purposes we do not need to compute $\textup{Jac}^\perp(L_x)$ explicitly (since it eventually cancels out in the main formula), for completeness we have
\[\textup{Jac}^\perp(L_x)=\sqrt{\det\br{L_xL_x^*}} = \sqrt{\det\br{d_x\otimes d_x K_x|_{F,F}}},\] 
where $K_x(y_1,y_2)=K(y_1,y_2)-K(x,y_2)K(y_1,x)/K(x,x)$ is the covariance kernel of $f$ conditioned on $f(x)=0$ or, equivalently, of the orthogonal projection of $f$ onto $\Vx$; this follows from the same routine computation as in Remark \ref{rem:evaluation}. More generally, if $L:V\rightarrow\R^k$ is a linear operator, the orthogonal Jacobian of $Lf$ is the square root of the determinant of the covariance of $Lf$.
\end{remark}

Let us now complete the proof of Lemma \ref{l:unit_normal}; for this we rely on elements from the proof of Lemma \ref{l:boundary}:

\begin{proof}[Proof of Lemma \ref{l:unit_normal}]
Let $u\in\tpiv_F(\hat{A})$, $x=\Xi_F(u)$, and take $U$, $\calC_1$ and $\calC_2$ as in \eqref{e:boundary_of_A_3}. By Lemma \ref{l:boundary}, we have $T_u\tpiv_F(\hat{A})=V_x$. In particular, for any such $v$, $\langle K(x,\cdot),v\rangle=v(x)=0$, so $K(x,\cdot)$ is orthogonal to $T_x\tpiv_F(\hat{A})$. Moreover, $\langle K(x,\cdot),K(x,\cdot)\rangle=K(x,x)$, which must be positive (otherwise all functions in $V$ vanish at $x$ which contradicts Condition \ref{cond:amplitude_5}). Therefore, the outward unit normal vector $\nu_{\hat{A}}(u)$ to $A$ at $u$ is plus or minus
\begin{equation}
\label{e:boundary normal vector}
v_x:=\frac{K(x,\cdot)}{\sqrt{K(x,x)}} .
\end{equation}
The sign of this vector depends on which of the $\calC_i$ belongs to $\hat{A}$. More precisely, a perturbation $u+\eta h$ (with $\eta\ll 1$) enters $\hat{A}$ whenever $\langle v_x,h\rangle=h(x)$ has the right sign. In particular, this shows that the sets $\tpivxp$ and $\tpivxm$ form a partition of $\tpivx$ and that, for each $\sigma\in\{+,-\}$ and each $u\in\tpivxs$, $\nu_{\hat{A}}(u)=-\sigma \frac{K(x,\cdot)}{\sqrt{K(x,x)}}$.
\end{proof}

Finally, we prove Proposition \ref{p:transversality}. For this we use the following standard fact which we state without proof:
\begin{lemma}\label{l:hausdorff}
Let $h:M\rightarrow M'$ be a Lipschitz map and let $S\subset M$ be a $k$-dimensional submanifold of $M$. Then the Hausdorff dimension of $h(S)$ is at most $k$. In particular, $\calH^d(h(S))=0$ for every $d>k$. 
\end{lemma}

\begin{proof}[Proof of Proposition \ref{p:transversality}]
Let us first give some intuition. The set $\mathfrak{D}_B\setminus\tVFd$ consists of functions which, in addition to having a level-$0$ critical point on $F$, are degenerate in some way. We express the five different cases of degeneracy as the vanishing of five explicit smooth functionals of pairs $(u,x)\in V\times F$ or triplets $(u,x,y)\in V\times F\times F_2$ for some $F_2\in\calF$. From this we deduce both that $\widetilde{\calI}_F$ is open in $\calI_F$ and that its complement has positive codimension. We then conclude by projecting the vanishing loci onto $V$.

\smallskip
Recall that $V\subset C^2(M)$ is a linear space. Let $d_1$ denote the dimension of $F$, and let $F_2,F_3\in\calF$ be strata of dimensions $d_2$ and $d_3$ respectively. We consider the following five subsets:
\begin{enumerate}
\item If $F<F_2$, let $\calI_{F,F_2}^1$ be the set of pairs $(u,x)\in \calI_{F}$ such that $d_xu\in T^*_{F_2}M|_x$.
\item Let $\calI_{F}^2$ be the set of pairs $(u,x)\in\calI_{F}$ such that $H_xu$ is singular.
\item If $F_2<F_3$, let $\calI_{F,F_2,F_3}^3$ be the set of triplets $(u,x,y)\in \calI_{F}\times F_2$ such that $x$ and $y$ are distinct, $y$ is also a stratified critical point of $u$ and $d_yu\in T^*_{F_3}M|_y$.
\item Let $\calI_{F,F_2}^4$ be the set of triplets $(u,x,y)\in \calI_{F}\times F_2$ such that $x$ and $y$ are distinct, $y$ is also a stratified critical point of $u$ and $H_y^{F_2}u$ is singular.
\item Let $\calI_{F,F_2}^5$ be the set of triplets $(u,x,y)\in \calI_{F}\times F_2$ such that $x$ and $y$ are distinct and $y$ is also a stratified critical point of $u$ with critical value $0$.
\end{enumerate}

\begin{claim}\label{cl:incidence_manifolds}
Each of the five subsets defined above is a closed subset of $V\times F$ (resp.\ $V\times F\times F_2$, as appropriate). Moreover, if we assume in addition that $V$ has finite dimension $N\in\N$ and satisfies Conditions \ref{cond:amplitude_4} and~\ref{cond:amplitude_5}, then each of these subsets is a finite union of submanifolds of codimension at least $N+d_1+1$ (resp.\ $N+d_1+d_2+1$).
\end{claim}
\begin{remark}
The proof of Claim \ref{cl:incidence_manifolds} is the only place in the paper where we use the fact that $\mathcal{F}$ is a tame stratification of $B$, rather than merely a Whitney stratification.
\end{remark}

\begin{proof}
We begin with a couple of definitions. For each $F\in\calF$, let $T_F^*M$ be the \textit{conormal bundle} to $F$, that is, for each $x\in F$, $T_F^*M|_x$ is the set of $\xi\in T^*_xM$ such that $\xi|_{T_xF}=0$. This is a smooth vector bundle whose rank is exactly the codimension of $F$ in $M$. In particular, $T_F^*M$ has codimension $d$ in $T^*M$. Recall from the definition of a tame stratification that, given $F_1,F_2\in\calF$ such that $F_1<F_2$, the set of limit points of $TF_2$ with basepoints on $F_1$ defines a vector bundle over $F_1$, denoted by $TF_2|_{F_1}$, which we call the generalised tangent bundle of $F_2$ over $F_1$. This allows us to extend the definition of conormal bundle as follows: the \textit{conormal bundle to $F_2$ over $F_1$}, denoted $T^*_{F_2}M|_{F_1}$, is the set of $(x,\xi)\in T^*M|_{F_1}$ such that $\xi$ vanishes on $T^*F_2|_x$. This defines a smooth vector bundle over $F_1$ whose rank is the codimension of $F_2$ in $M$. Thus, a point $x\in F_1$ is a non-degenerate stratified  critical point of some $u\in C^2(M)$ if and only if it is a non-degenerate critical point of $u|_{F_1}$ and for each $F_2\in\calF$ such that $F_1<F_2$, $(x,d_xu)\notin T^*_{F_2} M|_{F_1}$.

\smallskip
Now, assume first that $V$ has finite dimension $N\in\N$ and satisfies Conditions \ref{cond:amplitude_4} and \ref{cond:amplitude_5}. Since the proofs all follow the same structure, we cover in detail only the case of $\calI^1_{F,F_2}$, and then indicate what changes need to be made in the other cases.

Consider the map $\Phi_1:\calI_{F}\rightarrow T^*M|_{F}$ defined by $(u,x) \mapsto d_xu$, and recall the expression of the tangent spaces of $\calI_{F}$ given in \eqref{e:tangent_space_to_I_F}. By Condition \ref{cond:amplitude_5}, the map $\Phi_1$ is a submersion. Moreover, the set $T^*_{F_2}M|_{F}$ is a smooth submanifold of $T^*M|_{F}$ of codimension $1+d_2$ that is also a closed subset, and therefore $\calI^1_{F,F_2}=\Phi_1^{-1}\br{T^*_{F_2}M|_{F}}$ is a smooth submanifold of $\calI_{F}$ of codimension $d_1$ as well as a closed subset of this space. We have thus covered the case of $\calI^1_{F,F_2}$.

For $\calI^2_{F}$ we consider the map $\Phi_2:\calI_{F}\rightarrow\textup{Sym}^2\br{T^*F}$ defined by $(u,x)\mapsto H^{F}_xu$, which is a submersion by Condition \ref{cond:amplitude_4}. Instead of $T^*_{F_2}M|_{F}$, we consider the zero set of the determinant map $\det:\textup{Sym}^2\br{T^*F}\rightarrow\R$ induced by some auxiliary metric. Its zero set $\calW_{F}$ is closed and can be partitioned into the spaces of matrices of fixed rank in $\{0,\dots,d_1-1\}$ so it is a finite union of smooth submanifolds of positive codimension. Since $\calI^2_{F}=\Phi_2^{-1}\br{\calW_{F}}$, we are done.

The cases $\calI^3_{F,F_2,F_3}$ and $\calI^4_{F,F_2}$ are analogous to the first two cases. The maps $\Phi_1$ and $\Phi_2$ should be replaced by maps $\Phi_3$ and $\Phi_4$ defined on $\calI_{F,F_2}=\{(u,x,y)\in V\times F\times F_2 \st u(x)=0,\ d_xu|_{F}=0,\ d_yu|_{F_2}=0\}$ which is a smooth submanifold of $\calI_{F}\times F_2$ of codimension $d_2$ and whose tangent space at $(u,x,y)$ is
\[
\left\{(v,\tau_1,\tau_2)\in V\times T_xF\times T_yF_2 \st v(x)=0,\ d_xv|_{F}+H_xu(\tau_1,\cdot)=0,\ d_yv|_{F_2}+H_yu(\tau_2,\cdot)=0\right\}\, .
\]
They should be defined as follows: $\Phi_3:(u,x,y)\mapsto(u(x),d_xu,d_yu)$ and $\Phi_4(u,x,y)\mapsto H_yu$. As for $\calI^2_{F}$, Condition \ref{cond:amplitude_5} should be replaced by Condition \ref{cond:amplitude_4} in the case of $\calI^4_{F,F_2}$.

Finally, for $\calI^5_{F,F_2}$ we can consider the map $\Phi_5:\calI_{F,F_2}\rightarrow\R$ that maps each triple $(u,x,y)$ to $u(y)$. This map is a submersion by Condition \ref{cond:amplitude_5}. The conclusion follows accordingly.

This ends the proof of the finite-dimensional part of the claim. Consider now the general case. Observe that we still have $\calI^1_{F,F_2}=\Phi_1^{-1}\br{T^*_{F_2}M|_{F}}$, which is the preimage of a closed subset by a continuous map; in particular it is also closed. Since the same argument works with the four other cases, we also deduce the infinite-dimensional case of the claim.
\end{proof}

Let us now use Claim \ref{cl:incidence_manifolds} to prove that $\widetilde{\calI}_F$ is open in $\calI_F$. Consider $(u_k,x_k)\in (\calI_F\setminus\widetilde{\calI}_F)^\N$ that converges in $\calI_F$; we claim its limit $(u,x)$ belongs to $\calI_F\setminus\widetilde{\calI}_F$. Observe that $\calI_F\setminus\widetilde{\calI}_F$ is the union of the following sets:
\begin{enumerate}
\item The union over the $\{F_2\in\calF : F<F_2\}$ of the sets $\calI^1_{F,F_2}$.
\item The set $\calI^2_F$.
\item The union over $\{F_2,F_3\in\calF: F_2<F_3\}$ of the images of the projections $\calI^3_{F,F_2,F_3}\rightarrow\calI_F$.
\item The union over $F_2\in\calF$ of the images of the projections $\calI^4_{F,F_2}\rightarrow\calI_F$.
\item The union over $F_2\in\calF$ of the images of the projections $\calI^5_{F,F_2}\rightarrow\calI_F$.
\end{enumerate}
Since the above union is over a finite set, one of them contains an infinite number of terms of the sequence $(u_k,x_k)_{k\in\N}$. We can and will thus assume, up to extraction, that the sequence $(u_k,x_k)$ belongs to one of the sets just described. We now describe what happens in each case:
\begin{enumerate}
\item By Claim \ref{cl:incidence_manifolds}, $\calI^1_{F,F_2}$ is closed in $\calI_F$, so $(u,x)\in\calI^1_{F,F_2}\subset\calI_F\setminus \widetilde{\calI}_F$.
\item We reason likewise.
\item By construction, for each $k\in\N$, $(u_k,x_k)$ is the projection of a triplet in $\calI^3_{F,F_2,F_3}$. By compactness of $B$, we can extract a subsequence for which the third coordinate of the triplet converges in $\overline{F_2}$. Since the subsequence must have the same limit in the projection as the full sequence, we just denote it by $(u_k,x_k,y_k)_{k\in\N}\in (\calI^3_{F,F_2,F_3})^\N$ so that the third coordinate converges to some $y\in\overline{F_2}$. If $y\in F_2$ then $(u,x,y)\in\calI_F\times F_2$. Then by Claim~\ref{cl:incidence_manifolds}, $\calI^3_{F,F_2,F_3}$ is closed in $\calI_F\times F_2$ so $(u,x,y)\in\calI^3_{F_2,F_3}$, which implies that $(u,x)$ belongs to its projection onto $\calI_F$. If, on the other hand, $y\notin F_2$, (by Definition \ref{d:stratification}), $y$ must belong to some $F_4$ such that $F_4\in\calF$ such that $F_4<F_2$. Then $d_yu\in T^*_{F_2}M|_y$ (actually we even have $d_yu\in T^*_{F_3}M|_y$). If $y\neq x$, we must then have $(u,x,y)\in\calI^3_{F,F_4,F_2}$ so $(u,x)$ belongs to its projection onto $\calI_F$. Otherwise, if $y=x$, then, $(u,x)=(u,y)\in\calI^1_{F,F_4}$.
\item We reason as in the third case. As before, up to extraction, we can find $(y_k)_{k\in\N}\in F_2$ converging to some $y\in\overline{F_2}$ such that for each $k\in\N$, $(u_k,x_k,y_k)\in\calI^4_{F,F_2}$. Again, as before, if $y$ belongs to some face $F_3<F_2$, we have $d_yu\in T^*_{F_2}M|_y$ so $(u,x,y)\in \calI^4_{F,F_3}$. Otherwise, if $y\in F_2$, using Claim \ref{cl:incidence_manifolds} we deduce that $(u,x,y)\in\calI^4_{F,F_2}$.
\item We reason as in the fourth case.
\end{enumerate}
We have therefore proven that $\calI_F\setminus\widetilde{\calI}_F$ is closed in $\calI_F$. 

\smallskip
Next, we show that $\tVFd$ is open in $\mathfrak{D}_B$. By construction, $\mathfrak{D}_B$ is the union of the projections onto the first coordinates of the sets $\calI_{F_1}$ for $F_1\neq F$ and of the sets $\calI^1_F$, $\calI^2_F$, $\calI^3_{F,F_2,F_3}$, $\calI^4_{F,F_2}$ and $\calI^5_{F,F_2}$ defined above, taken over all the adequate $F_2$ and $F_3$. As before we take $(u_k)_{k\in\N}\in (\mathfrak{D}_B\setminus\tVFd)^\N$ converging to some $u\in \VFd$ and, up to extraction, there exist two strata $F_1\leq F_2$ and a sequence $(x_k)_{k\in\N}\in F_2^\N$ and $x\in F_1$ such that for each $k\in\N$, $(u_k,x_k)\in\calI_F\setminus\widetilde{\calI}_{F_2}$ and $\lim_{k\rightarrow\infty} (u_k,x_k)=(u,x)$. By Lemma~\ref{l:discriminant_is_closed}, $x$ is a stratified critical point of $u$. Let us prove that $u\in \mathfrak{D}_B\setminus\tVFd$. From now on, the reasoning is analogous to that used for $\calI_F\setminus\widetilde{\calI}_F$.
\begin{enumerate}
\item If $F_1\neq F$, then, $(u,x)\in\calI_{F_1}$ so $u\notin\tVFd$.
\item If $F_2>F_1=F$, then, as before $d_xu\in T^*|_{F_2}M|_x$ and so $(u,x)\in\calI^1_{F,F_2}$ and $u\notin\tVFd$.
\item If $F_2=F_1=F$ then for each $k\in\N$, $(u_k,x_k)$ belongs to $\calI_F\setminus\widetilde{\calI}_F$ which is closed in $\calI_F$ so that $(u,x)\notin\widetilde{\calI}_F$ and so $u\notin\tVFd$.
\end{enumerate}
This proves that $\mathfrak{D}_B\setminus\tVFd$ is closed in $\mathfrak{D}_B$ as announced. 

\smallskip
To finish, assume that $V$ has finite dimension $N\in\N$ and satisfies Conditions \ref{cond:amplitude_4} and \ref{cond:amplitude_5}. By (the finite-dimensional case of) Claim \ref{cl:incidence_manifolds}, $\VFd\setminus\tVFd$ is a finite union of projections of submanifolds of $V\times F$ and $V\times F\times F_2$ for $F_2\in\calF$ of codimensions at least $\dim(F)+2$ and $\dim(F)+\dim(F_2)+2$ respectively. By Lemma \ref{l:hausdorff}, we must therefore have $\calH^{N-1}(\VFd\setminus\tVFd)=0$.
\end{proof}

\subsection{Conditional expectation computation: Proof of Lemma \ref{l:simplifying_the_density}}\label{ss:conditional expectation}

In this section we prove Lemma \ref{l:simplifying_the_density}, that is, we rewrite the function $\Gamma$ defined by \eqref{e:G(x,y)} (see also \eqref{eq:gamma2}) in terms of a conditional expectation. 

\smallskip
Fix $t\in[0,1)$ and distinct $x_1\in F_1$ and $x_2\in F_2$. In the first part of the proof the exact expression of $\Upsilon_{x_1,x_2}$, defined by \eqref{eq:Upsilon def}, will not play any role except through the fact that it is bounded by a polynomial in $u_1,u_2$. Let $P_{x_1,x_2}$ be the orthogonal projector in $V\times V$ (equipped with the product metric) onto the subspace $\Vxd[1]\times \Vxd[2]$, and let $P_{x_1,x_2}^\perp=I-P_{x_1,x_2}$ be the complementary orthogonal operator onto the orthogonal complement, which we denote by  $(\Vxd[1]\times \Vxd[2])^\perp$. We write $(u_1,u_2)=w+w^\perp$ where $w=P_{x_1,x_2}(u_1,u_2)$ and $w^\perp=P_{x_1,x_2}^\perp(u_1,u_2)$. 
Let us define 
\[
j_{x_1,x_2}:V\times V \rightarrow \R\times T^*_{x_1}F_1\times\R\times T^*_{x_2}F_2
\]
by 
\[
j_{x_1,x_2}(u_1,u_2)=(u_1(x_1),d_{x_1}u_1|_{F_1},u_2(x_2),d_{x_2}u_2|_{F_2}).
\]
Note that the space $\Vxd[1]\times \Vxd[2]$ is exactly the kernel of $j_{x_1,x_2}$, hence  $j_{x_1,x_2}$ is a linear isomorphism from $(\Vxd[1]\times \Vxd[2])^\perp$ onto $\R\times T^*_{x_1}F_1\times\R\times T^*_{x_2}F_2$. With this notation we can rewrite the integral in \eqref{eq:gamma2} as 
\begin{equation}
\label{eq:G in orthogonal coordinates}
\Gamma(t;x_1,x_2)=\int_{\Vxd[1]\times \Vxd[2]} \, \un_{\tpivx[1]\times\tpivx[2]}(w) \Upsilon_{x_1,x_2}(w) \, \gamma_t(w) \,dw ,
\end{equation}
where $dw=\dv_{\Vxd[1]}\dv_{\Vxd[2]}$. In the same spirit we write $g_t=P_{x_1,x_2}f_t$ and $g_t^\perp=P_{x_1,x_2}^\perp f_t$ so that $f_t=g_t+g_t^\perp$. The density of $f_t$, conditioned on $g_t^\perp=0$, at $w=(u_1,u_2)\in \Vxd[1]\times \Vxd[2]$ is given by
\[
\gamma_{f_t|g_t^\perp=0}(w)=\frac{\gamma_t(u_1,u_2)}{\gamma_{g_t^\perp}(0)},
\]
where $\gamma_{g_t^\perp}(0)$ is the density of $g_t^\perp$ evaluated at $0$. Notice that for $f_t$, conditioning on $g_t^\perp=0$ is the same as conditioning on $(f_t^1(x_1),d_{x_1}f_t^1|_{F_1},f_t^2(x_2),d_{x_2}f_t^2|_{F_2})=0$. Since by definition $(u_1,u_2)=w$ on $\Vxd[1]\times \Vxd[2]$, \eqref{eq:G in orthogonal coordinates} becomes
\begin{align}
\label{eq:G with gamma perp}
\Gamma(t;x_1,x_2)&=\gamma_{g_t^\perp}(0)\int_{\Vxd[1]\times \Vxd[2]}
\un_{\tpivx[1]\times\tpivx[2]}(w)\Upsilon_{x_1,x_2}(w)\gamma_{f_t|g_t^\perp=0}(w) \,dw\\
\nonumber &=\gamma_{g_t^\perp}(0) \, \E_{t;x_1,x_2}\brb{\un_{\tpivx[1]\times\tpivx[2]}(f_t)\Upsilon_{x_1,x_2}(f_t^1,f_t^2)}.
\end{align}
In the above expression, the density $\gamma_{g_t^\perp}(0)$ is with respect to the orthogonal coordinates in $(\Vxd[1]\times \Vxd[2])^\perp$, and we need to express it in terms of $K$. Let $\widetilde{Q}_{t;x_1,x_2}$ be the covariance matrix of $g_t^\perp$ in some orthonormal system of coordinates in  $(\Vxd[1]\times \Vxd[2])^\perp$. Let $Q_{t;x_1,x_2}$ be the covariance of 
\[
(f_t^1(x_1),d_{x_1}f_t^1|_{F_1}, f_t^2(x_2), d_{x_2}f_t^2 |_{F_2} )
=
j_{x_1,x_2}(f_t)
=
j_{x_1,x_2}(g_t)
\] 
in any orthonormal coordinate system of $\R\times T^*_{x_1}F_1\times\R\times T^*_{x_2}F_2$ equipped with the product metric. Treating $j_{x_1,x_2}$ as an isomorphism from $(\Vxd[1]\times \Vxd[2])^\perp$ onto $\R\times T^*_{x_1}F_1\times\R\times T^*_{x_2}F_2$ we see that  the covariances $\widetilde{Q}_{t;x_1,x_2}$ and $Q_{t;x_1,x_2}$ are linked by the following relation
\[
\widetilde{Q}_{t;x_1,x_2}=\br{j_{x_1,x_2}^*}^{-1}Q_{t;x_1,x_2}j_{x_1,x_2}^{-1} .
\]
In particular, $\det(\widetilde{Q}_{t;x_1,x_2}) =\det(Q_{t;x_1,x_2})/\det(j_{x_1,x_2}j_{x_1,x_2}^*)^{-1}$. Recalling that $\gamma_{t;x_1,x_2}(0)$ is the density of $j_{x_1,x_2}(f_t)$ at $0$, we have
\[
\gamma_{g_t^\perp}(0)
=
\gamma_{t;x_1,x_2}(0)\sqrt{\det\br{j_{x_1,x_2}j_{x_1,x_2}^*}}  .
\]
It remains to compute $\sqrt{\det\br{j_{x_1,x_2}j_{x_1,x_2}^*}}$. Notice first that $j_{x_1,x_2}$ factors as the direct product of the two linear maps $j_{x_i}:\Vxd[i]^\perp\rightarrow\R\times T^*_{x_i}F_i$ for $i\in\{1,2\}$ defined as $j_{x_i}(u)=(u(x_i),d_{x_i}u|_{F_i})$,
\begin{equation}\label{e:change_of_coordinates_3}
\det\br{j_{x_1,x_2}j_{x_1,x_2}^*} =\det\br{j_{x_1}j_{x_1}^*}\det\br{j_{x_2}j_{x_2}^*} .
\end{equation}
To compute $\det\br{j_{x_i}j_{x_i}^*}$ note that, since $j_{x_i}$ is $0$ on $\Vxd[i]$, this determinant does not depend on whether $j_{x_i}$ acts on $\Vxd[i]^\perp$ or the entire $V$; we treat it as an operator on $V$.  Next, we write $V$ as orthogonal sum of $V_{x_i}$ which is the space of functions such that $v(x_i)=0$ and its orthogonal complement which is spanned by $K(x_i,\cdot)$ (see the discussion preceding \eqref{e:boundary normal vector}). Let us choose orthonormal coordinates in $V$ that are adopted to this decomposition, that is $K(x_i,\cdot)/\|K(x_i,\cdot)\|$ must be one of the basis vectors. In this coordinates $j_{x_i}$ factors as $u\mapsto u(x_i)$ acting on the span of $K(x_i,\cdot)$ (this is the operator $\mathrm{Ev}_{x_i}$ from Remark \ref{rem:evaluation}) and $u\mapsto d_{x_i}u\mid_{F_i}$ on $V_x$ (which is the operator $L_{x_i}$). The factorisation implies that
\[
\sqrt{\det\br{j_{x_i}j_{x_i}^*}}
=
\sqrt{\det\br{\mathrm{Ev}_{x_i}\mathrm{Ev}_{x_i}^*}
\det\br{L_{x_i}L_{x_i}^*}}
=
\textup{Jac}^\perp(\mathrm{Ev}_{x_i})\textup{Jac}^\perp(L_{x_i})
.
\]

Plugging this computation into \eqref{eq:G with gamma perp} we see that $\Gamma$ is equal to
\begin{align*}
& \gamma_t(x_1,x_2)\sqrt{K(x_1,x_1)K(x_2,x_2)} \textup{Jac}^\perp(L_{x_1})\textup{Jac}^\perp(L_{x_2})
\\
& \qquad \qquad \times \E_{t;x_1,x_2}\brb{\un_{\tpivx[1]\times\tpivx[2]}(w)\Upsilon_{x_1,x_2}(f_t^1,f_t^2)}.
\end{align*}
Recalling the definition of $\Upsilon_{x_1,x_2}(u_1,u_2)$, and in particular pulling the terms $\sqrt{K(x_i,x_i)}$ and $\textup{Jac}^\perp(L_{x_i})$ from this definition out of the expectation (since they do not depend on $u_i$) so that they cancel with those already present, we deduce the result.

\begin{remark}
The cancellations in the above derivation are not so mysterious, since the relevant terms are Jacobians of evaluations of $f$ and its differential and they appear, first, when we switch from space coordinates to functional coordinates, and then once again when we move back. 
\end{remark}

\section{Proof of the main theorem: from the finite to the infinite-dimensional case}

In this section we complete the proof of the covariance formula in Theorem \ref{t:main}. The basic idea is to (i) reinterpret topological events in terms of the discriminant, (ii) approximate the field $f$ by a sequence of fields $f_k$ taking values in a finite-dimensional spaces $V_k$, and then (iii) pass to the limit in the formula of Proposition \ref{p:main}.

\smallskip
In Section \ref{ss:boundary_of_pivotal} we show that the boundary of pivotal events is well behaved, which will allow us to take limits of the expectations in the right-hand side of Proposition \ref{p:main}. In Section~\ref{ss:topology_is_discriminant} we verify that topological events are encoded by the discriminant. Next, in Section~\ref{ss:approximation} we construct the finite-dimensional approximation and state an abstract continuity lemma for expectations that we use in the proof. Finally in Section~\ref{ss:main_proof} we assemble these elements into a proof of Theorem \ref{t:main}.

\smallskip At the end of the section we also verify that Corollary \ref{c:crossings} is indeed a special case of Theorem~\ref{t:main}, as claimed in Section \ref{s:intro}.

\subsection{On the boundary of pivotal events}\label{ss:boundary_of_pivotal}
In this section we compare pivotal events in different subspaces of $C^2(M)$, link the two distinct notions of pivotal events we have introduced, and study the boundary of pivotal events.

\smallskip \sloppy Recall that $(B, \calF)$ denotes an arbitrary stratified set of $M$. Fix a linear subspace $V\subset C^2(M)$, not necessarily finite-dimensional. Also fix $\widetilde{A}\in\widetilde{\sigma}_{\textup{discr}}(B,C^2(M))$, and let $\hat{A}\in \hat{\sigma}_{\textup{discr}}(B,C^2(M))$ be the set of $u\in C^2(M)$ whose discriminant class (in $C^2(M)$) belongs to~$\widetilde{A}$. Observe that the set $\hat{A}_V = \hat{A}\cap V$ belongs to $\hat{\sigma}_{\textup{discr}}(B,V)$, i.e.\ it is encoded by the $V$-discriminant. Indeed, it is the set of functions $u\in V$ whose discriminant class in $C^2(M)$ belongs to $\widetilde{A}$. Recall also the definition, for $x\in B$ and $\hat{A}\in \hat{\sigma}_{\textup{discr}}(B,V)$, of the sets $\tpivx$ and $\tpivxs$ from Definition~\ref{d:critical_points}. \fussy

\smallskip
The main result of this section is the following:
\begin{lemma}[On pivotal events]\label{l:pivotal_events}
Suppose that $V$ contains the constant functions on $M$. Then
\begin{enumerate}
\item $\tpiv_x(\hat{A}_V)=\tpivx\cap V$ and, for each $\sigma\in\{+,-\}$, $\widetilde{\piv}_x^\sigma(\hat{A}_V)=\tpivxs\cap V$.
\end{enumerate}
Moreover, let $f$ be a Gaussian field on $M$ satisfying Condition \ref{cond:non_degeneracy}. Then, conditionally on $x$ being a stratified critical point of $f$ with $f(x)=0$, a.s.\ 
\begin{enumerate}
\item[(2)] $f\in\tpivx$ if and only if (i) $f\in\pivx$ and (ii) $H_x^Ff$ is a non-degenerate bilinear form. Moreover, for each $\sigma\in\{+,-\}$, the same is true if we replace $\tpivx$ by $\tpivxs$ and $\pivx$ by $\pivxs$.
\item[(3)] If $x$ is a non-degenerate critical point then $f\notin\partial\tpivx$, where $\tpivx$ is seen as a subset of the space $\Vxd$.
\end{enumerate}
\end{lemma}

\begin{remark}
Note that we only apply Lemma \ref{l:pivotal_events} to approximations of the field $f$ (as opposed to $f$ itself), so it is irrelevant that the constant functions will not belong to the Cameron-Martin space of $f$ in general.
\end{remark}

\begin{remark}\label{rem:piv_two}
If we had been willing to impose a non-degeneracy condition on the Hessian of~$f$, we could have concluded from Lemma \ref{l:pivotal_events} that, conditionally on $x$ being a level-$0$ stratified critical point of $f$, a.s.\ $f\in\tpivx$ if and only if $f\in\pivx$, and this is the sense in which we think of $\tpivx$ and $\pivx$ as equal up to null sets. Since non-degeneracy of the Hessian is unnecessary for the result to hold, we do not do this.
\end{remark}

In order to prove Lemma \ref{l:pivotal_events}, we use the following result:

\begin{lemma}\label{l:pivotal_events_1}
Let $u\in C^2(M)$ be such that $u$ has a unique non-degenerate stratified critical point $x$ at level $0$ (c.f.\ the set $\cup_{F \in \calF}\tVFd$). Then
\begin{enumerate}
\item For $\eps>0$ small enough, neither $u+\eps$ nor $u-\eps$ have a stratified critical point at level $0$. Moreover, let $\calC_1=\calC_1(u)$ and $\calC_2=\calC_2(u)$ be the connected components in $C^2(M)\setminus\mathfrak{D}_B(C^2(M))$ of $u+\eps$ and $u-\eps$ respectively. Then, $\overline{\calC_1}\cup\overline{\calC_2}$ (resp.\ $\overline{\calC_1}$, $\overline{\calC_2}$) is a neighbourhood of $u$ in $C^2(M)$ (resp.\ in the set of functions $u'\in C^2(M)$ such that $u(x)\geq 0$, in the set of functions $u'\in C^2(M)$ such that $u(x)\leq 0$).
\item Let $U\subset M$ be a neighbourhood of $x$. Then there is a neighbourhood $\calU$ of $u$ in the discriminant $\mathfrak{D}_B(C^2(M))$ such that, for each $u'\in \calU$, $u'$ has exactly one stratified critical point at level $0$, which is non-degenerate and belongs to $U$. Moreover, we have $\calC_1(u')=\calC_1(u)$ and $\calC_2(u')=\calC_2(u)$ (defined as in $(1)$).
\end{enumerate}
\end{lemma}

\begin{remark}
If we were working in a finite-dimensional space, we could think of $u$ as belonging to the smooth part of the discriminant. Since this discriminant is a hypersurface, this would mean that in a small neighbourhood of $u$ the discriminant would be diffeomorphic to a hyperplane, separating the ambient space into two connected components $\calC_1$ and $\calC_2$, and moreover small perturbations of $u$ would yield the same $\calC_1$ and $\calC_2$. Lemma \ref{l:pivotal_events_1} encodes (part of) this intuition. 
\end{remark}

\begin{remark}\label{rem:definition_of_tpivxs}
The first point of Lemma \ref{l:pivotal_events_1} implies that the definition of $\tpivxs$ does not change if one takes, in the definition of this set, $h\in C^2(M)$ instead of merely $h \in V$. Similarly, the definition does not change if one requires $h$ to be a positive constant (which assists in showing a function does not belong to $\tpivxs$).
\end{remark}

\begin{proof}[Proof of Lemma \ref{l:pivotal_events_1}]
We start by showing that the property that $u$ has a stratified critical point near $x$ is stable under $C^2$ perturbations; this involves isolating $x$ as a critical point in a uniform way. Fix a neighbourhood $U$ of $x$. Since $x$ is a non-degenerate critical point, it is isolated in the set of critical points of $u$ (see Lemma \ref{l:sndcps_are_isolated}), which is compact by Lemma~\ref{l:discriminant_is_closed}. Thus, the critical value $u(x)$ is isolated in the set of critical values of $u$. In particular, for each $\eps>0$, both $u+\eps$ and $u-\eps$ belong to $ C^2(M)\setminus\mathfrak{D}_B(C^2(M))$, which justifies the existence of $\calC_1$ and $\calC_2$. Let us show that $\overline{\calC_1}\cup\overline{\calC_2}$ is a neighbourhood of $u$ in $C^2(M)$. Let $F\in\calF$ be the stratum containing $x$. For each $r>0$, let $B_r$ be the Riemmanian ball of radius $r>0$ in $F$ centred at $x$. Since $x$ is a non-degenerate critical point at $x$, the section $du$, which is $C^1$, vanishes transversally at $x$ on the stratum $F$ and stays bounded from below on the higher strata near $x$. Therefore, there exist $r=r(u)>0$ and $\eta=\eta(u)>0$ such that for each $w\in C^2(M)$ such that $\|w\|_{C^2(B)}\leq \eta$, the following holds:
\begin{itemize}
\item The ball $B_r$ is included in $U$;
\item The section $d(u+w)|_F$ vanishes exactly once on $B_r$;
\item For any $F'\neq F$, $d(u+w)|_{F'}$ does not vanish on $B_r$;
\item $u+w$ has no stratified critical points with critical value in $[-10\eta,10\eta]$ outside of $B_r$;
\item If moreover $\|w\|_{C^2(B)}\leq \eta/8$ then $|u+w|\leq \eta/4$ on $B_r$.
\end{itemize}
In particular, for each $s\in(0,\eta]$, $u\pm s$ does not belong to the discriminant. Let $w\in C^2(M)$ be such that $\|w\|_{C^2(B)}\leq \eta/8$. Let us show that $u+w\in\overline{\calC_1}\cup\overline{\calC_2}$, and that if $u+w\in\mathfrak{D}_B(C^2(M))$ then $u+w$ has a unique stratified critical point at level $0$ which belongs to~$B_r$. To this end, we will first consider a path $(v_t)_t$ from $u$ to $u+w$ where $w$ is a small perturbation. Along this path we will find further perturbations $v_{t,s}=v_t+s$ of $v_t$, for suitable choices of $s$, that do not belong to the discriminant and that belong to the two connected components $\calC_1$ and $\calC_2$.

\smallskip
 More precisely, for each $t\in[0,1]$ and each $s\in[-\eta/2,\eta/2]$, let $v_{t,s}=u+tw+s$. Then, for each $t\in[0,1]$ and each $s\in[-\eta/2,\eta/2]$, $\|v_{t,s}-u\|_{C^2(B)}\leq \eta$ and $\|v_{t,0}-u\|_{C^2(M)}\leq \eta/8$. In particular, $v_{t,s}$ has a unique stratified critical point in $B_r$, which we call $y_t$ (since it does not depend on $s$) and no other stratified critical points with critical value in $[-9\eta,9\eta]$. Moreover, $\sup_{B_r}|v_{t,0}|\leq\eta/4$ so that $\min_{B_r}v_{t,\eta/2}\geq \eta/4$ and $\max_{B_r}v_{t,-\eta/2}\leq-\eta/4$. In particular, for each $t\in [0,1]$, $v_{t,\pm\eta/2}\notin\mathfrak{D}_B(C^2(M))$. Thus, $v_{1,\pm\eta/2}\in\calC_1\cup\calC_2$. Now, for each $s\in[-\eta/2,\eta/2]$, $v_{1,s}=u+w+s$. In particular, if $u(y_1)+w(y_1)\geq 0$, $v_{1,s}$ does not belong to the discriminant for $s\in(0,\eta/2]$ and converges to $u+w$ as $s\rightarrow 0$. If on the other hand, $u(y_1)+w(y_1)\leq 0$, the same approximation holds by taking $s\in[-\eta/2,0)$ and $s\rightarrow 0$. In any case, by construction of this approximation $v_{1,s}\in\calC_1\cup\calC_2$ as long as $s\neq 0$ so that $u+w\in \overline{\calC_1\cup\calC_2}=\overline{\calC_1}\cup\overline{\calC_2}$. This shows that $\overline{\calC_1}\cup\overline{\calC_2}$ is a neighbourhood of $u$ in $C^2(M)$ and that for each $v\in(\overline{\calC_1}\cup\overline{\calC_2})\cap\mathfrak{D}_B(C^2(M))$, $v$ has a unique stratified critical point at level $0$, which is in $B_r\subset U$.

\smallskip
Next notice that $u+w$ is in the same connected component of the complement of the discriminant as $u+s$ for $s\ll1$ if $w(x)<0$, and is in the same connected component as $u-s$ for $s\ll1$ if $w(x)>0$, which proves that $\overline{\calC_1}$ (resp.\ $\overline{\calC_2}$) is a neighbourhood of $u$ in the set of functions taking non-negative (resp.\ non-positive) values at $x$. Finally, if $u(y_t)+tw(y_t)=0$, by construction $u+w+\eps$ (resp.\ $u+w-\eps$) is in the same connected component as $u+\eps$ (resp.\ $u-w$) for $\eps>0$ small enough. In other words, $\calC_i(u)=\calC_i(u+w)$ for $i\in\{1,2\}$. This ends the proof of the lemma.
\end{proof}

\begin{proof}[Proof of Lemma \ref{l:pivotal_events}]
We prove the three statements in the lemma sequentially:

\smallskip
\noindent(1). In order to prove that $\widetilde{\piv}_x(\hat{A}_V)=\tpivx\cap V$ and $\widetilde{\piv}_x^\sigma(\hat{A}_V)=\tpivxs\cap V$, it is enough that $\partial\hat{A}_V=\partial\hat{A}\cap V$, where $\partial\hat{A}_V$ is the boundary of $\hat{A}_V$ in $V$. Clearly, $\partial\hat{A}_V\subset\partial\hat{A}\cap V$. On the other hand, let $u\in\partial{\hat{A}}\cap V$. Then, by Lemma \ref{l:pivotal_events_1}, there exist $\calC_1$ and $\calC_2$ two connected components of $C^2(M)\setminus\mathfrak{D}_B(C^2(M))$ such that for $\eps>0$ small enough, $u+\eps\in\calC_1$, $u-\eps\in\calC_2$ and $\overline{\calC_1}\cup\overline{\calC_2}$ is a neighbourhood of $u$ in $C^2(M)$. Let us assume that $\calC_1\subset\hat{A}$ and $\calC_2\subset C^2(M)\setminus\hat{A}$ since $u\in\partial\hat{A}$, and the case $\calC_2\subset\hat{A}$ and $\calC_1\subset C^2(M)\setminus\hat{A}$ follows by exchanging $\hat{A}$ and its complement. Since $u\in V$ and the constant functions belong to $V$, $u\pm\eps\in V$. In particular, letting $\eps\rightarrow 0$, we deduce that $u\in\overline{\hat{A}\cap V}\cap\overline{V\setminus\hat{A}}=\partial\hat{A}_V$, from which it follows that $\partial\hat{A}_V=\partial\hat{A}\cap V$ as announced.

\smallskip
\noindent(2). Let $f_x$ denote the field $f$ conditioned on $f(x)=0$ and on $x$ being a stratified critical point of $f$. Then, $f_x$ is a.s.\ $C^2$. Assume now that $f_x\in\pivxp$. Then, there exists a (random) $h\in C^2(M)$ satisfying $h\geq 0$ such that, for small enough values of $\delta>0$, $f_x+\delta h\in\hat{A}$ and $f_x-\delta h\notin\hat{A}$. In particular, $f_x\in\partial\hat{A}$. Moreover, since $f$ satisfies Condition \ref{cond:non_degeneracy}, by the regression formula  $(f_x(y),d_yf_x)$ is non-degenerate for $y \neq x$, and so by Bulinskaya's lemma (\cite[Proposition 1.20]{azais_wschebor}) a.s.\ $f_x$ has no other stratified critical points at level $0$. If we also assume that $H_x^Ff_x$ is non-degenerate, then by Remark \ref{rem:definition_of_tpivxs} $f_x\in\tpivxp$. Thus, we have shown that if $f_x\in\pivxp$ and $H_x^Ff_x$ is non-degenerate then a.s.\ $f_x\in\tpivxp$.

Conversely, assume that $f_x\in\tpivxp$ and let $U\subset M$ be a neighbourhood of $x$ in $M$. Then, $H_x^Ff_x$ is non-degenerate. Since $f$ satisfies Condition \ref{cond:non_degeneracy}, as before by Bulinskaya's lemma a.s.\ $f_x$ has no other critical points at level $0$. We may thus apply Lemma~\ref{l:pivotal_events_1} to $f_x$, which implies that there exists a neighbourhood $\calU$ of $f_x$ in $C^2(M)$, two connected components $\calC_1$ and $\calC_2$ of $C^2(M)\setminus\mathfrak{D}_B(C^2(M))$, and a geodesic ball $B_r \subset U$ of radius $r > 0$ centred at $x$, such that the following holds:
\begin{itemize}
\item For all small enough $\delta>0$, $f_x+\delta\in\calC_1$ and $f_x-\delta\in\calC_2$.
\item The union $\overline{\calC_1}\cup\overline{\calC_2}$ covers $\calU$.
\item Each $v\in\calU\cap\mathfrak{D}_B(C^2(M))$ has a unique stratified critical point in $B_r$ and no stratified critical points at level $0$ outside of~$B_r$.
\end{itemize}
Since $f_x\in\tpivxp$, we have $\calC_1\subset\hat{A}$ and $\calC_2\subset \hat{A}^c$. Let $h\in C^2_c(W)$ be equal to $1$ on $B_r$. Then, for all small enough $\delta>0$, $f_x\pm h\in\calU$, so $f_x$ has no stratified critical points at level $0$ outside of $B_r$. Inside $B_r$ it coincides with $f_x$ up to a constant $\pm\delta$. In particular, if $\delta\neq 0$, $f_x\pm h\in\calC_1\cup\calC_2$. Moreover, by considering the path $(f_x\pm(\delta h +s(1-h)))_{s\in[0,\delta]}$, we conclude that $f_x+h\in\calC_1\subset\hat{A}$ and $f_x-h\in\calC_2\subset\hat{A}^c$. But $h$ is supported arbitrarily close to $x$. Thus, $f_x\in\pivxp$. Reasoning symmetrically, we get the same statement with the $+$ exponent replaced by $-$, and combining the two results we get the same property for $\pivxs$ replaced by $\piv_x(\hat{A})$.

\smallskip
\noindent(3). Assume that $x$ is a non-degenerate critical point of $f_x$. Then, Lemma \ref{l:pivotal_events_1} applies so that there are two discriminant classes $\calC_1$ and $\calC_2$ such that $\overline{\calC_1}\cup\overline{\calC_2}$ is a neighbourhood of $f_x$ in $C^2(M)$ and there is a neighbourhood $W$ of $f_x$ in the discriminant such that for each $v\in W$, and each small enough $\eps>0$, $v+\eps\in\calC_1$ and $v-\eps\in\calC_2$. If $f_x\in\tpivx$ then exactly one of the two classes belongs to $\hat{A}$, and hence the elements of $W$ will all belong to the boundary of $\hat{A}$. Therefore, $f_x$ belongs to the interior of $\tpivx$ in the space $V_x'$ of functions in $C^2(M)$ with a stratified critical point at $x$ at level $0$. Similarly, if $f_x\notin\tpivx$ then either both $\calC_1$ and $\calC_2$ are subsets of $\hat{A}$ or neither of them are. So then, as before, the elements of $W$ cannot belong to the boundary of $\hat{A}$ so that $f_x$ is in the interior of $\Vxd\setminus\tpivx$. In both cases, $f_x\notin\partial\tpivx$, which proves the last part of the proposition.\qedhere
\end{proof}

\subsection{The topological class is encoded by the discriminant}\label{ss:topology_is_discriminant}
In this subsection we verify that topological events are encoded by the discriminant, making explicit the link between the events that appear in Theorem \ref{t:main} and the events that appear in Proposition \ref{p:main}; in passing, we also prove the measurability of the stratified isotopy classes.

\smallskip
In this section $(B, \calF)$ again denotes an arbitrary stratified set of $M$; nevertheless, here we prefer to view~$\calF$ as a general Whitney stratification (see Remark~\ref{rk:whitney_stratifications}) since we make use of the standard theory of Whitney stratifications. Recall the definition of $B$-discriminant classes from Definition \ref{d:discriminant}, as well as the definition of the stratified isotopy class from Definition \ref{d:topological_events}.

\begin{lemma}[Topological class is encoded by the discriminant]\label{l:topology_to_discriminant}
Suppose that $u,v\in C^2(M)$ have the same $B$-discriminant class in $C^2(M)$. Then their excursion sets $\{u>0\}$ and $\{v>0\}$ have the same stratified isotopy class, i.e., $[\{u>0\}]_B=[\{v>0\}]_B$.
\end{lemma}

Since the discriminant classes are $C^1$-open (see Lemma \ref{l:discriminant_is_closed}), there are at most countably many of them. This immediately implies the following:

\begin{corollary}
\label{c:measurability}
There are at most countably many stratified isotopy classes of subsets of $B$. Moreover, the map  $[\calD]_B$ from the probability space $\Omega$ into the set of stratified isotopy classes is measurable. 
\end{corollary}

Before proving Lemma \ref{l:topology_to_discriminant}, let us recall some standard facts about Whitney stratifications; they can all be easily checked from the definitions of the objects they involve:

\begin{itemize}
\item 
If $I\subset \R$ is an open interval, then the collection $\calF_I=(F\times I)_{F\in\calF}$ is a Whitney stratification of $B\times I$.
\item 
For each open subset $W\subset M$, $\calF_W=(F\cap W)_{F\in\calF}$ is a Whitney stratification of $B\cap W$.
\item
 Consider $f:M\rightarrow N$ a smooth map between two Riemannian manifolds. Assume that $f|_B$ is proper and that for each $F\in\calF$, $f|_F:F\rightarrow N$ is a submersion. Then, for each $y\in N$, the preimage $f^{-1}(y)\cap B$ is naturally equipped with a Whitney stratification $\calF_y$ whose strata are the intersections $F\cap f^{-1}(y)$ where $F\in\calF$ (see Definition 1.3.1 of Part I of \cite{goresky_macpherson}).
\end{itemize}

The proof of Lemma \ref{l:topology_to_discriminant} is a standard application of Thom's first isotopy lemma (see (8.1) of \cite{mather_1973}) and the isotopy extension theorem (see \cite{edwards_kirby_1971}). In fact, the only place we use $C^2$ regularity in this proof is when we apply Thom's first isotopy lemma.

\begin{proof}[Proof of Lemma \ref{l:topology_to_discriminant}]
Let $u\in C^2(M)\setminus\mathfrak{D}_B(C^2(M))$, i.e.\ $u$ has no stratified critical points in $B$ at level~$0$. Recall that, by Lemma \ref{l:discriminant_is_closed} and since $B$ is compact, the set of critical points of $u$ is compact. In particular, this set is at positive distance from the zero set of $u$ and there exists a bounded open neighbourhood $W\subset M$ of $u^{-1}(0)$ in $M$ and a convex neighbourhood $\calU$ of $u$ in $C^1(M)$ such that for each $v\in\calU$ and each face $F\in\calF$, $dv|_F\neq 0$ in $\overline{W}$ and $v\neq 0$ on $B\setminus W$. We will prove that for each $v\in\calU\cap C^2(M)$, $[\{v>0\}]_B=[\{u>0\}]_B$. To do so, notice that since $\calU$ is open and convex, there exists $I$ an open interval containing $[0,1]$ such that for each $t\in I$, $u_t=tv+(1-t)u\in\calU\cap C^2(M)$. The family $\calF_{W,I}=((F\cap W)\times I)_{F\in\calF}$ defines a Whitney stratification of $(B\cap W)\times I$ in $W\times I$. Moreover, since for each $t\in I$, $u_t$ has no critical points on any face of $\calF$ inside $W$, the map
\begin{align*}
(W\cap B)\times I \rightarrow\R\times I \ , \quad U:(x,t) \mapsto (u_t(x),t)
\end{align*}
is a submersion when restricted to any face of $\calF_{W,I}$. It is proper since $B$ is compact and $id_I:I\rightarrow I$ is proper. In particular, by Thom's first isotopy lemma, since $U$ is $C^2$ there exists a stratified homeomorphism $h: W\times I\rightarrow (W\times I)\cap U^{-1}(0) \times\R\times I$ (where $U^{-1}(0)\cap B$ is equipped with the preimage Whitney stratification that exists since $F$ is transverse to $\{0\}$ in $\R\times I$) such that $U\circ h^{-1}$ is the projection on the last two factors. Note that $U^{-1}(0)=\{(x,t)\in M\times I \st (u_t(x),t)=(0,0)\}=u_0^{-1}(0)\times\{0\}$. In particular, the map
\begin{align*}
(B\cap u_0^{-1}(0))\times I \rightarrow B\times I \ , \quad 
(z,t) \mapsto (f_t(x),t):=h^{-1}((x,0),0,t)
\end{align*}
defines an isotopy of $u_0^{-1}(0)$ in $B$ such that for each $t\in I$, $f_t(u_0^{-1}(0)\cap B)=u_t^{-1}(0)\cap B$. Since it is constructed from $h$ it extends to an isotopy of a tubular neighbourhood of $u_0^{-1}(0)$ in $B$ that preserves strata of $\calF$. By Corollary 1.4 of \cite{edwards_kirby_1971} (and its extension provided in Section 7 of the same article), there exists a continuous isotopy $B\times I\rightarrow B\times I$ $(x,t)\rightarrow (\Phi_t(x),t)$ such that for each $t\in I$, $\Phi_t$ is a stratified homeomorphism of $B$ and $\Phi_t\circ f_0=f_t$. In particular, $\Phi_t(u_0^{-1} (0)\cap B)=u_t^{-1}(0)\cap B$ for each $t\in I$. Since $\Phi_0=id$, and $\Phi_t$ is continuous in $t$, we also have $\Phi_1(\{u_0>0\}\cap B)=\{u_1>0\}\cap B$ and so $[\{u>0\}]_B=[\{v>0\}]_B$. Given that this is true for all $v\in \calU\cap C^2(M)$, we have shown that equivalence classes for the equivalence relation generated by the map $u\mapsto [\{u>0\}]_B$ are $C^2(M)$-open.
In particular, since $\mathfrak{D}_B(C^1(M))$ is $C^1$-closed (by Lemma \ref{l:discriminant_is_closed} or just $C^2$-closed by the present argument) each topological class in $C^2(M)\setminus\mathfrak{D}_B(C^2(M))$ must be a union of connected components of $C^2(M)\setminus\mathfrak{D}_B(C^2(M))$ and the proof is over.
\end{proof}

\subsection{Approximation results}\label{ss:approximation}
To deduce Theorem \ref{t:main} from Proposition \ref{p:main}, we approximate the field $f$ by a sequence of fields $(f_k)_{k\in\N}$ taking values in finite-dimensional subspaces $(V_k)_{k\in\N}$ of $C^2(M)$. Then, we integrate the result of Proposition \ref{p:main} and pass to the limit. 

\smallskip
In this subsection, we first show the existence of an approximating sequence in a general setting (see Lemma~\ref{l:approximation}), and then state the abstract continuity lemma for expectations (see Lemma \ref{l:portm}) which we use to show the convergence of the terms in Proposition \ref{p:main}. 

\begin{lemma}[Existence of finite-dimensional approximations]\label{l:approximation}
Fix $l\in\N$ and let $f$ be an a.s.\ $C^l$ Gaussian field on a smooth manifold $M$ of dimension $d$. Let $V\subset C^l(M)$ be a linear subspace of $C^l(M)$ such that $f$ belongs a.s.\ to $V$. Then the following holds:
\begin{enumerate}
\item There exists a sequence $(V_k)_{k\in\N}$ of finite-dimensional linear subspaces of $V$ and a sequence of Gaussian fields $(f_k)_{k\in\N}$, all defined in the same probability space as $f$, that converges in probability to $f$ in the topology of uniform $C^l$ convergence on compact subsets of $M$, and such that for each $k\in\N$, $f_k\in V_k$ a.s.\ and $f_k$ defines a non-degenerate Gaussian vector in $V_k$. If $f$ is centred then the $f_k$ can also be chosen to be centred.
\item Moreover, let $W\subset V$ be a finite-dimensional subspace. Then, we may find sequences $(V_k)_{k\in\N}$ and $(f_k)_{k\in\N}$ as in $(1)$ such that $W\subset V_k$.
\end{enumerate}
\end{lemma}

\begin{proof}
Consider a countable atlas $(U_j,\phi_j)_{j\in\N}$ of $M$. Let $J\subset\N$, let $\eta>0$ be a parameter to be fixed later, and let $I\subset M$ be a locally finite set such that, for each $j\in J$ and $z\in \phi_j(U_j)$, there exists $x\in\phi_j(I\cap U_j)$ for which $|z-x|\leq \eta$. Let $\eps>0$, fix $j\in J$ and let $B\subset U_j$ be a compact subset. Let us prove that there exists $\eta_0=\eta_0(j,B,\eps)>0$ such that for all  $\eta\leq\eta_0$, the field $f_I:=\E\brb{f\ |\ f_I}$ satisfies
\begin{equation}\label{e:approximation_1}
\prob\brb{\|f\circ\phi_j^{-1}-f_I\circ\phi_j^{-1}\|_{C^l(B)}>\eps}<\eps\, .
\end{equation}
Since the seminorms $\|\cdot\|_{C^l(B)}$, where $j\in\N$ and $B\subset U_j$ ranges over the compact subsets of $U_j$, generate the topology of $C^l(M)$, repeating the above construction for a  sequence $(\eps_k)_{k\in\N}\rightarrow 0$ yields a sequence $(f_k)_{k\in\N}$ of fields satisfying \eqref{e:approximation_1} for $\eps=\eps_k$  which proves the first point of the lemma.

\smallskip
To prove \eqref{e:approximation_1}, fix $\alpha\in\N^d$ and let $g=\partial^\alpha(f\circ\phi_j^{-1})$ and $g_I=\partial^\alpha(f_I\circ\phi^{-1}_j)$. Also, let $N=N(d,l)\in\N$ be the number of multi-indices $\alpha\in\N^d$ such that $|\alpha|\leq l$. Observe that, for each $z\in \phi_j(B)$, there exists $z\in \phi_j(U_j\cap I)$ such that  $|x-z|<\eta$ so that
\[
|g(z)-g_I(z)|\leq|g(z)-g(x)|+\E\brb{|g(x)-\E\brb{g(z)\ |\ f|_I}|}\leq |g(z)-g(x)|+\E\brb{|g(z)-g(x)|} .
\]
In particular, 
\[\E \big[ \sup_{z\in B}|g(z)-g_I(z)| \big] \leq 2\E \Big[ \sup_{x,y\in B,\ |x-y|\leq\eta}|g(x)-g(y)| \Big] =:2\E\brb{X_\eta}. \]
Now, since $g$ is a continuous Gaussian field on $B$ (which is compact), it is a.s.\ bounded and the family $(X_\eta)_{\eta>0}$ is uniformly $L^1$ (see for instance  \cite[Theorem 2.9]{azais_wschebor}). Moreover, since $g$ is continuous, $X_\eta$ converges a.s.\ to $0$ as $\eta\rightarrow 0$. In particular, $\lim_{\eta\rightarrow 0}\E\brb{X_\eta}=0$. Thus, there exists $\eta_\alpha=\eta_\alpha(\eps,d,l)>0$ such that for each $\eta\leq\eta_\alpha$, $\prob\brb{\sup_{z\in B}|g(z)-g_I(z)|>\eps}\leq \eps/C$. The estimate \eqref{e:approximation_1} follows by taking a union bound of the probability of the events $\sup_{z\in B}|g(z)-g_I(z)|$ where $\alpha\in\N^d$ ranges over all the multiindices such that $|\alpha|\leq l$, and will be valid for $\eta\leq\eta_0=\min_{|\alpha|\leq l}\eta_\alpha$. 

Consider now a sequence $(\eta_k)_{k \in \N}$ of positive real numbers converging to $0$ and $(I_k)_{k\in\N}$ an increasing sequence of finite subsets of $M$, such that for each $k\in\N$ and $j\leq k$, $\phi_j(I_k\cap U_j)$ is an $\eta_k$-net of $\phi_j(U_j)$. For each $k\in\N$, let $f_k=f_{I_k}$ defined as above (with $J=[0,k]$). Then, for each $j\in\N$ and each compact subset $B\subset$, by \eqref{e:approximation_1}, $\lim_{k\rightarrow \infty}\|f\circ\phi_j^{-1}-f_k\circ\phi_j^{-1}\|_{C^l(B)}=0$ in probability so $(f_k)_{k\in\N}$ converges to $f$ in probability. 

We claim that each $f_k$ belongs to a finite-dimensional subspace $V_k$. Indeed, let $K$ be the covariance of $f$. By the regression formula (Proposition 1.2 of \cite{azais_wschebor}), for each $k\in\N$, $f_k$ is a random linear combination of the functions $K(\cdot,x)$ for $x\in I_k$ and of $\E[f]$ the mean of $f$. Hence it belongs to the finite-dimensional subspace $V_k$ generated by these functions. Moreover, since $f_k$ is the mean of a random variable with values in $V$, we have $V_k\subset V$ and if $f$ is centred, by construction, $f_k$ is centred. This concludes the proof of the first statement. 

For the second statement, take $(f_k)_{k\in\N}$ as above, let $(h_1,\dots,h_m)$ be a basis of $W$, and let $\xi_1,\dots,\xi_m$ be independent standard normals. Then, clearly, $\frac{1}{k}\left(\xi_1h_1+\dots+\xi_m h_m\right)$ converges to~$0$ in probability in $C^l(M)$, so that replacing $f_k$ by 
\[ f_k + \frac{1}{k}\left(\xi_1h_1+\dots+\xi_m h_m\right) \]
yields the required result.
\end{proof}

Next we state without proof an abstract continuity lemma for expectations; this can be considered a simple variant of the standard Portemanteau lemma.

\begin{lemma}\label{l:portm}
Let $(X,Y)$ and $(X_k,Y_k)_{k\in\N}$ be random variables with values in $\R\times E$, where $E$ is a Polish space. Assume that the sequence $(X_k,Y_k)_{k\in\N}$ converges in law towards $(X,Y)$, and that the sequence $(X_k)_{k\in\N}$ is uniformly integrable. Let $A\subset E$ and assume $\prob\brb{X\neq 0,\ Y\in\partial A}=0$. Then
\[
\lim_{k\rightarrow \infty}\E\brb{X_k\one_{[Y_k\in A]}}=\E\brb{X\one_{[Y\in A]}}.
\]
\end{lemma}
%\begin{proof}
%By the Skorokhod representation theorem, we may assume that $(X_k,Y_k)_{k\in\N}$ and $(X,Y)$ are defined on the same probability space and that $(X_k,Y_k)_{k\in\N}$ converges pointwise to $(X,Y)$. Since the sequence $(X_k)_{k\in\N}$ is uniformly integrable, the sequence with $k$-th term $a_k=\E\brb{X_k\one_{[Y_k\in A]}}$ is bounded. Let $a$ be one its subsequential limits. Since the sequence $(X_k,Y_k)_{k\in\N}$ converges pointwise to $(X,Y)$, it has a subsequence $(X_{k_l},Y_{k_l})_{l\in\N}$ converging a.s. to $(X,Y)$ such that $\lim_{l\rightarrow +\infty} a_{k_l}=a$. We have, a.s.
%\[
%X\one_{[Y\in\mathring{A}]}\leq\liminf_l X_{k_l}\one_{[Y_{k_l}\in A]}\leq\limsup_l X_{k_l}\one_{[Y_{k_l}\in A]}\leq X\one_{[Y\in\overline{A}]}\, .
%\]
%But by assumption, a.s. $X\un_{[Y\in\mathring{A}]}=X\one_{[Y\in\overline{A}]}=X\one_{[Y\in A]}$. Thus, a.s.
%\[
%\lim_{l\rightarrow+\infty}X_{k_l}\one_{[Y_{k_l}\in A]}=X\one_{[Y\in A]}\, .
%\]
%Taking expectations, since $(X_k)_{k\in\N}$ is uniformly integrable, we deduce that $a=\E\brb{X\one_{[Y\in A]}}$ and the result follows.
%\end{proof}

\subsection{Completing the proof of Theorem \ref{t:main}}\label{ss:main_proof}
To complete the proof of Theorem \ref{t:main} we assemble the previous elements together, namely we:
\begin{itemize}
\item Approximate $f$ by a sequence of finite-dimensional fields $(f_k)_{k\in\N}$ with nice regularity and non-degeneracy properties constructed using Lemma \ref{l:approximation}.
\item Use Lemma \ref{l:topology_to_discriminant} to encode the topological events $A_1$ and $A_2$ via the discriminant.
\item Apply Proposition \ref{p:main} to the fields $f_k$ and the events encoded by the discriminant.
\item Pass to the limit in each term of the formula given by Proposition \ref{p:main}, using Lemmas \ref{l:pivotal_events} and \ref{l:portm}.
\item Show that the two definitions of pivotal events coincide using Lemma \ref{l:pivotal_events}.
\end{itemize}

\begin{proof}[Proof of Theorem \ref{t:main}]
Recall that $(B_1,\calF_1)$ and $(B_2,\calF_2)$ are stratified sets of $M$, and $A_1$ and $A_2$ are topological events on $B_1$ and $B_2$ respectively. By Lemma \ref{l:topology_to_discriminant}, for each $i\in\{1,2\}$ there exists $\hat{A}_i\in\hat{\sigma}_{\textup{discr}}(B_i,C^2(M))$ such that $\prob[f\in \hat{A_i}\triangle A_i]=0$, and so it will be sufficient to work with the events $\hat{A}_i\in\hat{\sigma}_{\textup{discr}}(B_i,C^2(M))$.

Let us first define the approximating sequence of fields. By Remark \ref{rem: finite dim}, there exists a finite-dimensional subspace $W\subset C^2(M)$ satisfying Conditions \ref{cond:amplitude_4} and \ref{cond:amplitude_5} that contains the constant functions. Hence we may define a sequence of Gaussian fields $(f_k)_{k\in\N}$, taking values in a sequence of finite-dimensional linear subspaces $W \subset V_k\subset C^2(M)$, that satisfy all the properties guaranteed by Lemma \ref{l:approximation} (setting $\ell = 2$, and so in particular the $f_k$ converge in probability in the topology of uniform $C^2$ on compact sets). Since $W$ satisfies Conditions \ref{cond:amplitude_4} and \ref{cond:amplitude_5} so does each $V_k$, and so Proposition \ref{p:main} applies to the sets $\hat{A}_{i,k}=\hat{A}_i\cap V$ for $i \in \{1,2\}$ and the field $f_k$.

Next, recall that $f^1$ and $f^2$ denote independent copies of $f$, and let $(f_k^1)_{k\in\N}$ and $(f_k^2)_{k\in\N}$ be independent copies of $(f_k)_{k\in\N}$, with $f_k^1$ converging to $f_1$ and $f_k^2$ converging to $f_2$ (i.e.\ in $C^2$). Similarly, recall that $f_t$ denotes the interpolation $(f_t^1,f_t^2)=(f^1,t(f^1-\mu)+\sqrt{1-t^2}(f^2-\mu)+\mu)$, and define for each $k\in\N$ the interpolation $f_{t,k}=(f_{t,k}^1,f_{t,k}^2)$ analogously. Applying Proposition~\ref{p:main} we have, for each $k\in\N$ and $t\in[0,1)$,
\begin{align}
\label{e:main_proof_1}
& \prob\brb{f_{t,k}\in \hat{A}_{1,k}\times \hat{A}_{2,k}}-\prob\brb{f_{0,k}\in \hat{A}_{1,k}\times \hat{A}_{2,k}} \\
\nonumber & \quad =\sum_{F_1\in\calF_1,\ F_2\in\calF_2}\int_0^t\int_{F_1\times F_2}K_k(x_1,x_2)
\times \Lambda_k(s;x_1,x_2)\gamma_{s,k;x_1,x_2}(0) \,\dv_{F_1}(x_1)\dv_{F_2}(x_2) \d s,
\end{align}
where $\Lambda_k(s;x_1,x_2)$ equals
\begin{equation}\label{e:main_proof_2}
\E_{s;x_1,x_2}\brb{\sigma (f^1_{s,k},f^2_{s,k})\one_{\widetilde{\piv}_{x_1}(\hat{A}_{1,k})\times\widetilde{\piv}_{x_2}(\hat{A}_{2,k})}(f^1_{s,k},f^2_{s,k})|\det\br{H_{x_1}f_{s,k}^1|_{F_1}}||\det\br{H_{x_2}f_{s,k}^2|_{F_2}}|},
\end{equation}
and where $K_k$ is the covariance of $f_k$, and $\gamma_{s,k;x_1,x_2}$ is the density of 
\[ (f^1_{s,k}(x_1), d_{x_1}f^1_{s,k}|_{F_1} , f^2_{s,k}(x_2) , d_{x_2}f^2_{s,k}|_{F_2}) \]
in orthonormal coordinates (recall that the subscript $s;x_1,x_2$ in the expectation denotes conditioning on this vector vanishing).
 
\smallskip
Let us compute the limits of both sides of \eqref{e:main_proof_1} as $k\rightarrow \infty$, beginning with the left-hand side. Notice that for each $t\in[0,1)$ and $k\in\N$, $f_k^t\in \hat{A}_{1,k}\times\hat{A}_{2,k}$ if and only if $f_k^t\in \hat{A}_1\times\hat{A}_2$. Since $\partial (\hat{A}_1\times\hat{A}_2)\subset\mathfrak{D}_B(C^2(M))\times C^2(M)\cup C^2(M)\times \mathfrak{D}_B(C^2(M))$, and since $f$ satisfies Condition~\ref{cond:non_degeneracy}, by Bulinskaya's lemma (Proposition 1.20 of \cite{azais_wschebor})
\[ \prob[ (f^1_t,f^2_t)\in \partial(\hat{A}_1\times\hat{A}_2) ]=0. \]
 Thus, by Lemma \ref{l:portm} (setting $X=1$, $Y=(f^1_t,f^2_t)$, $E=C^2(M)\times C^2(M)$ and $A=\hat{A}_1\times\hat{A}_2$), we have
\begin{equation}\label{e:main_proof_3}
\lim_{k\rightarrow\infty}\prob [f_{t,k}\in \hat{A}_{1,k}\times\hat{A}_{2,k}]=\prob [f_t\in \hat{A}_2\times \hat{A}_2] .
\end{equation}

\smallskip
We turn now to the right-hand side of \eqref{e:main_proof_1}; we begin by computing the pointwise limit of the integrand, and then apply the dominated convergence theorem. Fix $F_1\in\calF_1$, $F_2\in\calF_2$, $x_1\in F_1$, $x_2\in F_2$ and $s\in[0,t]$ (so that $s<1$). Since the Hessians restricted to zero-dimensional faces have vanishing determinants, we may assume that $\dim(F_1),\dim(F_2)>0$. This allows us to assume that $x_1\neq x_2$ by removing a set of measure zero from the integral in $(x_1,x_2)$. Now, since $f_k^s$ converges in probability to $f^s$ as $k\rightarrow\infty$, it also converges in law. In particular $K_k(x_1,x_2)$ converges to $K(x_1,x_2)$ and $\gamma_{s,k;x_1,x_2}(0)$ converges to $\gamma_{s;x_1,x_2}(0)$. To deal with $\Lambda_k(s;x_1,x_2)$, we note that $f_{s,k}$ converges in law to $f_s$ in $C^2(M)$ as $k\rightarrow\infty$ so that $K_k$ converges to $K$ in $C^2(M\times M)$. Now, since the vector $(f_s^1(x_1),d_{x_1}f_s^1|_{F_1},f_s^2(x_2),d_{x_2}f_s^2|_{F_2})$ is non-degenerate, by the regression formula (Proposition 1.2 of \cite{azais_wschebor}), the law of $f$ conditioned on this vector vanishing is well-defined and depends continuously in $K$. Since the covariance of a field determines its law, we deduce that the sequence of fields $(f_{s,k})_{k\in\N}$ conditioned on $(f_{s,k}^1(x_1),d_{x_1}f_{s,k}^1|_{F_1},f_{s,k}^2(x_2),d_{x_2}f_{s,k}^2|_{F_2})=0$ converges in law to $f$ with the above conditioning. We denote the conditional law of these fields by $\prob_{s;x_1,x_2}[\dots]$. By the first statement of Lemma~\ref{l:pivotal_events} 
\[
\prob_{s;x_1,x_2}\brb{f_{t,k}\in \widetilde{\piv}_{x_1}(\hat{A}_{1,k})\times\widetilde{\piv}_{x_2}(\hat{A}_{2,k})\triangle\widetilde{\piv}_{x_1}(\hat{A}_1)\times\widetilde{\piv}_{x_2}(\hat{A}_2)}=0
\] 
so, if we temporarily set 
\[
A=\tpivx[1]\times\tpivx[2],\quad 
X_s=X_s(x_1,x_2)=\sigma(f_s)|\det\br{H_{x_1}^{F_1}f_s^1}||\det\br{H_{x_2}^{F_2}f_s^2}|
\]
 and 
\[
X_{s,k}=X_{s,k}(x_1,x_2)=\sigma(f_{s,k})|\det\br{H_{x_1}^{F_1}f_{s,k}^1}||\det\br{H_{x_2}^{F_2}f_{s,k}^2}|,
\] 
we have
\[
\Lambda_k(s;x_1,x_2)=\E_{s;x_1,x_2}\brb{\one_A(f_{s,k})X_{s,k}(x_1,x_2)} .
\]
Since under the conditioning $f_k$ converges in law to $f$, and since these are Gaussian fields, the sequence $(X_{s,k}(x_1,x_2))_{k\in\N}$ is uniformly integrable (though the bound may depend on $s$, $x_1$ and $x_2$). On the other hand, the random variables $f_{s,k}$ and $f_s$ take values in the Polish space $C^2(M)$. By the third point of Lemma \ref{l:pivotal_events}, a.s.\ either the Hessian of one of the $f_s^i$'s is degenerate, which implies that either $X_s=0$ or $f_t\notin\partial A$. Moreover, the pair $(X_{s,k},f_{s,k})$ converges in probability to the pair $(X_s,f_s)$. By Lemma \ref{l:portm}, we have $\lim_{k\rightarrow \infty}\Lambda_k(s;x_1,x_2)=\Lambda(s;x_1,x_2)$ which is equal to
\begin{equation}\label{e:main_proof_4}
\E_{s;x_1,x_2}\brb{\sigma (f^1_s,f^2_s)\one_{\widetilde{\piv}_{x_1}(\hat{A}_1)\times\widetilde{\piv}_{x_2}(\hat{A}_2)}(f^1_s,f^2_s)|\det\br{H_{x_1}^{F_1}f_s^1}||\det\br{H_{x_2}^{F_2}f_s^2}|} .
\end{equation}
In summary, the integrand of the right hand side of \eqref{e:main_proof_1} converges pointwise to the same quantity with $f_{s,k}$ replaced by $f_s$ everywhere.

To apply the dominated convergence theorem to the right hand side of \eqref{e:main_proof_1}, we must find a uniform $L^1$ bound on the integrand. To bound $\Lambda_k(s;x_1,x_2)\times\gamma_{s,k;x_1,x_2}(0)$ we use Lemma~\ref{l:det} with $X^i_{s,k}=H_{x_i}f^i_{s,k}$, $Y_{s,k}=Y_{s,k}(x_1,x_2)=(f_{s,k}^1(x_1),d_{x_1}f_{s,k}^1|_{T_{x_1}F_1},f_{s,k}^2(x_2),d_{x_2}f_{s,k}^2|_{T_{x_2}F_2})$ for each $i\in\{1,2\}$ and $k\in\N$. For any finite-dimensional Gaussian vector $X$ in a space equipped with a scalar product, let $\textup{DC}(X)$ be the determinant of the covariance of $X$ in orthonormal coordinates. The covariances of the coordinates of the $X_{s,k}^i$ are bounded in terms of derivatives up to order two in each variable of the covariances $K_k$; since these are uniformly bounded, there exists a constant $C<\infty$ for which, for each $k\in\N$,
\[
|\Lambda_k(s;x_1,x_2)\gamma_{s,k;x_1,x_2}(0)|\leq\frac{C}{\sqrt{\textup{DC}(Y_{s,k})}}.
\]
Next, for $i,j\in\{1,2\}$ let $Y^{ij}_k(x_j)=(f_k^i(x_j),d_{x_j}f_k^i|_{T_{x_j}F_j})$ so that, for any $j_1,j_2\in\{1,2\}$, $Y^{1j_1}_k$ is independent from $Y^{2j_2}_k$ and
\[
Y_{s,k}=(Y^{11}_k,s(Y^{12}_k-\E\brb{Y_k^{12}})+\sqrt{1-s^2}(Y^{22}_k-\E\brb{Y_k^{22}}+\E\brb{Y_k^{12}})) .
\]
Then, by Lemma \ref{l:linalg_3}, for each $s\in[0,t]$,
\begin{align*}
\textup{DC}(Y_{s,k})&=\textup{DC}(Y^{11}_k,sY^{12}_k+\sqrt{1-s^2}Y^{22}_k)\\
&\geq \textup{DC}(Y^{11}_k,sY^{12}_k)+\textup{DC}(Y^{11}_k)\textup{DC}(\sqrt{1-s^2}Y^{22}_k)\\
&\geq (1-s^2)^{\textup{dim}(F_2)+1}\textup{DC}(Y^{11}_k)\textup{DC}(Y^{22}_k)\\
&=(1-t^2)^{n+1}\textup{DC}(f_k(x_1),d_{x_1}f_k|_{T_{x_1}F_1})\textup{DC}(f_k(x_2),d_{x_2}f_k|_{T_{x_2}F_2}) .
\end{align*}
Since $(f_k)_{k\in\N}$ converges in law to $f$, and since $(f(x),d_xf)$ is non-degenerate for each $x\in M$ (and since $B$ is compact), there exist $k_0\in\N$ and a constant $c>0$ such that for each $t\in[0,1)$, $s\in[0,t]$, $x_1\in F_1$ and $x_2\in F_2$, as long as $k\geq k_0$,
\[
\textup{DC}(Y_{s,k})\geq (1-t^2)^{n+1} c .
\]
In particular $|\Lambda_k(s;x_1,x_2)\gamma_{s,k;x_1,x_2}(0)|\leq C c^{-1/2}(1-t^2)^{-(n+1)/2}$. Hence the integrand in the right hand side of \eqref{e:main_proof_1} is uniformly integrable, so the dominated convergence theorem applies.

\smallskip
All in all, letting $k\rightarrow\infty$ in both sides of \eqref{e:main_proof_1} yields
\begin{align}
\label{e:main_proof_5}
& \prob \big[f_t\in\hat{A}_1\times\hat{A}_2] -\prob \big[f_0\in\hat{A}_1\times\hat{A}_2] \\
\nonumber & \qquad =\sum_{F_1\in\calF_1,\ F_2\in\calF_2}\int_0^t\int_{F_1\times F_2}K(x_1,x_2) 
\times\Lambda(s;x_1,x_2)\gamma_{s;x_1,x_2}(0)\dv_{F_1}(x_1) \, \dv_{F_2}(x_2)\d s,
\end{align}
where $\Lambda$ is defined in \eqref{e:main_proof_4}.

\smallskip Let us now finish off the proof. Fix $F_1\in\calF_1$ and $F_2\in\calF_2$. By the second point of Lemma \ref{l:pivotal_events}, for each $i\in\{1,2\}$, each $s\in[0,t]$ and each $x_i\in F_i$, a.s.\ (under the conditioning present in $\Lambda$)
\[
\det(H_{x_i}^{F_i}f^i_s)\one_{\tpivx[i]}(f^i_s)=\det(H_{x_i}^{F_i}f^i_s)\one_{\pivx[i]}(f^i_s)
\]
since the only place at which the two pivotal events do not coincide is where the determinant of the Hessian vanishes. Moreover, $\pivx[i]$ splits as the disjoint union of $\pivxp[i]$ and $\pivxm[i]$. All in all, a.s.\ under the conditioning used in $\Lambda$,
\[
\begin{aligned}
&\sigma (f^1_s,f^2_s)\one_{\pivx[1]\times\pivx[2]}(f^1_s,f^2_s)
\\
& \quad =\br{\one_{\pivxp[1]\times\pivxp[2]}(f^1_s,f^2_s)+\one_{\pivxm[1]\times\pivxm[2]}(f^1_s,f^2_s)}
\\
& \quad \quad -\br{\one_{\pivxp[1]\times\pivxm[2]}(f^1_s,f^2_s)+\one_{\pivxm[1]\times\pivxp[2]}(f^1_s,f^2_s)} .
\end{aligned}
\]
In particular,
\[
\Lambda(s;x_1,x_2)\gamma_{s;x_1,x_2}(0)=I_s^+(x_1,x_2)-I_s^-(x_1,x_2),
\]
where $I^\pm_s$ are the signed pivotal intensity functions from Definition \ref{d:pivotal_measure}. Thus, \eqref{e:main_proof_5} yields
\begin{align*}
& \prob [f_t\in\hat{A}_1\times\hat{A}_2] - \prob [f_0\in\hat{A}_1\times\hat{A}_2] \\
& \quad =\sum_{F_1\in\calF_1,\ F_2\in\calF_2}\int_{F_1\times F_2}K(x_1,x_2)
\times \int_0^t\br{I^+_s(x_1,x_2)-I^-_s(x_1,x_2)}\d s\, \dv_{F_1}(x_1)\dv_{F_2}(x_2) .
\end{align*}
By the definition of $\hat{A}_1$ and $\hat{A}_2$ as well as $(f_t)_{t\in[0,1]}$, 
\[
\prob [f_1\in\hat{A}_1\times\hat{A}_2] - \prob [f_0\in\hat{A}_1\times\hat{A}_2] =\prob[A_1\cap A_2]-\prob[A_1]\prob[A_2],
\]
so letting $t\rightarrow 1$ yields
\[
\prob[A_1\cap A_2]-\prob[A_1]\prob[A_2]=\int_{B_1\times B_2}K(x_1,x_2)\br{\d\pi^+(x_1,x_2)-\d\pi^-(x_1,x_2)},
\]
where $d\pi^\sigma(x_1,x_2)$, for $\sigma\in\{+,-\}$, are the signed pivotal measures from Definition \ref{d:pivotal_measure}.
\end{proof}

To complete the section, we verify that Corollary \ref{c:crossings} is indeed a special case of Theorem \ref{t:main}:
\begin{proof}[Proof of Corollary \ref{c:crossings}]
Recall that $B_1$ and $B_2$ are closed boxes, $F^{1}_0$ and $F^{2}_0$ are their interiors, and $F^{i}_j$, for $j \in \{1,2,3,4\}$ and $i \in \{1,2\}$, are their four sides. Together with the corners of the boxes, which we denote $F^{i}_j$ for $j \in \{5,6,7,8\}$ and $i \in \{1,2\}$, the set of $F^i_j$ form a tame (and affine) stratification of $B_1$. Moreover, the events $A_1$ and $A_2$ are indeed topological events since stratified isotopies preserve crossings, and so Theorem~\ref{t:main} applies to these events, yielding an exact formula for $\text{Cov}[A_1, A_2]$.

Let us next analyse the terms in this formula. By Remark \ref{r:zero}, the corners do not contribute to the sum over strata. Further, by Remark \ref{r:piv_mon}, since $A_i$ are both increasing events, the sets $\piv^-_{x_i}(A_i)$ are empty and so the pivotal measure only contains positively pivotal events. Finally notice that, for each $i=1,2$, $t\in(0,1)$ and $(x_1,x_2)\in (B_1,B_2)$, under the conditioning that $x_1$ and $x_2$ are stratified critical points at level $0$ of $f_t^1$ and $f_t^2$ respectively, the fields $f_t^1$ and $f_t^2$ have a.s.\ no other critical points at level $0$ (by the non-degeneracy assumption). Moreover, if the Hessians of $f_t^i$ at $x_i$ do not degenerate, the $x_i$ are non-degenerate stratified critical points of $f_t^i$. The pivotal event for $A_i$ is then equivalent to the existence of a path in $\{f_t^i\geq 0\}$ joining `left' to `right' and a path in $\{f_t^i\leq 0\}$ joining `top' to `bottom', both passing through $x_i$.
\end{proof}

%%%%%%%%%%%%

\section{Proofs of the applications}
\label{s:application_proofs}

In this section we give proofs for the applications that are discussed in Sections~\ref{s:intro} and~\ref{s:main_formula}, in particular Theorem \ref{t:mix_gen} and Corollaries \ref{c:mix_euc}, \ref{c:con_euc}, \ref{c:mix_kos}, and~\ref{c:con_kos}.

\subsection{Strong mixing for topological events}

\begin{proof}[Proof of Theorem \ref{t:mix_gen}]
Let $c$ denote a constant, that can change line-to-line, that depends only on $d$. By Theorem \ref{t:main} and the definition of $\alpha_\text{top}$, after replacing $K$ with its absolute value, and dropping the condition that $f_t^1$ and $f_t^2$ lie in the pivotal sets, it suffices to show that, for all $t \in [0, 1]$,
\begin{equation}
\label{e:ptmixgen}
  \gamma_t(x_1,x_2) \,\E_{t; x_1,x_2}\brb{|\det(H_{x_1}^{F_1}f_t^1)\det(H_{x_2}^{F_2}f_t^2)| }  
  \end{equation}
is bounded above by the maximum, over $i, j, k \in \{1, 2\}$, of
\[   c \,  \frac{   \E \big[ \| H_{x_i}^{F_i} f \|_\text{op}^2 \, | \, d_{x_i} f|_{F_i} = 0 \big]^{d_i} }{\sqrt{ \textup{det}(\Delta(x_1, x_2) ) }} \max \Big\{ 1,    \Big(  \frac{  K(x_j, x_j) \, \textup{det}( d_{x_k}\otimes d_{x_k} K|_{F_i \times F_i})  }{ \sqrt{\textup{det}(\Delta(x_1, x_2))}  } \Big)^{2d_i} \Big\}     . \]
Let $\Delta_t(x_1, x_2)$ denote the covariance matrix of \eqref{e:vec} and let $\Omega_t(x_1, x_2)$ denote the covariance matrix of $(d_{x_1} f_t^1|_{F_1} , d_{x_2} f_t^2|_{F_2}  )$. Applying Lemma \ref{l:det} to the matrices $X_1 = H_{x_1}^{F_1}f_t^1$ and $X_2 = H_{x_2}^{F_2}f_t^2$ and the vectors $Y = (f_t^1(x_1), f_t^2(x_2))$ and $Z = (d_{x_1} f_t^1|_{F_1},  d_{x_2} f_t^2|_{F_2})$, \eqref{e:ptmixgen} is bounded above by the maximum, over $i, j, k \in \{1, 2\}$, of
\[  c \,  \frac{   \E \big[ \| H_{x_i}^{F_i} f_t^i \|_\text{op}^2 \, | \, d_{x_i} f_t^i|_{F_i} = 0 \big]^{d_i} }{\sqrt{ \textup{det}(\Delta_t(x_1, x_2) ) }} \max \Big\{ 1,    \Big(  \frac{ \text{Var}[ f_t^j(x_j) ] \, \sqrt{ \textup{det}( \Omega_t (x_1, x_2)) }}{ \sqrt{\textup{det}(\Delta_t(x_1, x_2))}  } \Big)^{2d_i}  \Big\}   . \]
Recall that $f_t^i$ is equal in law to $f$, and that, by Lemma \ref{l:int},
\[ \textup{det}(\Delta_t(x_1, x_2)) \ge \textup{det}(\Delta_1(x_1, x_2)) = \textup{det}(\Delta(x_1, x_2)) \]
and
\[  \textup{det}(\Omega_t(x_1, x_2)) \le \textup{det}(\Omega_0(x_1, x_2)) =  \Pi_{k={1,2}} \, \textup{DC}[d_{x_k} f|_{F_k}  ] \le \max_{k=1,2}   \textup{DC}[d_{x_k} f|_{F_k}  ]^2  ,\]
where $\textup{DC}(X)$ is the determinant of the covariance of $X$ in orthonormal coordinates. Since
\[ \text{Var}[ f(x_j) ] = K(x_j, x_j)  \quad \text{and}  \quad   \textup{DC}[d_{x_k} f|_{F_k}  ] =  \textup{det}( d_{x_k} \otimes d_{x_k} K|_{F_i \times F_i} )^2 ,  \]
we have the desired result.
\end{proof}

\begin{proof}[Proof of Corollary \ref{c:mix_euc}]
The non-degeneracy condition in the statement of Corollary \ref{c:mix_euc} is equivalent to Condition \ref{cond:non_degeneracy}, and so we are in the setting of Theorem \ref{t:mix_gen}. First we argue that there exists constants $c_1, c_2 > 0$, depending only on $d$, $\kappa(0)$, and the Hessian of $\kappa$ at $0$, such that, if 
\[\max_{\alpha\in\N^d:|\alpha|\leq 2} \, \, \sup_{x_1 \in B_1, x_2 \in B_2} |\partial^\alpha \kappa(x_1-x_2)|< c_1,\]
then for any affine sets $F_1$ and $F_2$ the covariance matrix of the Gaussian vector
\begin{equation}
\label{e:ptmixeuc1}
  (f(x_1), \nabla f|_{F_1}(x_1), f(x_2), \nabla f|_{F_2}(x_2)) 
  \end{equation}
has a determinant bounded below by $c_2$. Let $\Sigma_1(x_1,x_2)$ denote this matrix, and observe that
\[   \Sigma_1(x_1,x_2) = \left[ {\begin{array}{cc}
   M_{11} & M_{12}  \\
 M_{12}^T & M_{22} \\
  \end{array} } \right]   \]
where, by stationarity,
  \[ M_{ii}   =  \left[ {\begin{array}{cc}
   \kappa(0) & 0  \\
 0 & H_0^{F_i} \kappa   \end{array} } \right]  , \]
 and $M_{12}$ depends only on the value and second derivatives of $\kappa$ at $x_1$ and $x_2$ (here $H_0^F$ denotes the Hessian at the point $x$ in an orthonormal basis of the linear span of $F$). The result then follows by the continuity, on the set of strictly positive-definite matrices, of the determinant with respect to the entry-wise sup-norm.

Combined with the stationarity of $f$, under the assumption that
\[\max_{\alpha\in\N^d:|\alpha|\leq 2} \, \, \sup_{x_1 \in B_1, x_2 \in B_2} |\partial^\alpha \kappa(x_1-x_2)|< c_1,\]
the quantity $c_{F_1, F_2}(x_1, x_2)$ in Theorem \ref{t:mix_gen} can be bound above by
\[ c  \max_{i} \bigg\{    \E \big[ \| H_0^{F_i} f \|_\text{op}^2 \, | \, \nabla f|_{F_i}(0)  = 0 \big]^{\textup{dim}(F_i)} \bigg\} ,   \]
for some $c > 0$. Since this is a finite quantity, we have proved the result. 
 
To verify the observation in Remark \ref{r:mix_euc} note that we have already established that $c_1$ depends only on $d$, $\kappa(0)$, and the Hessian of $\kappa$ at $0$. Next, since all norms on $\mathbb{R}^d$ are equivalent,
\[ \E \big[  \| H_0^{F_i} f \|_\text{op}^2 \, | \, \nabla f|_{F_i}(0) = 0  \big] \le c_d  \max_{j_1, j_2} \E \big[ ( H_0^{F_1,F_2} f )_{j_1,j_2}^2 \, | \,  \nabla  f|_{F_i}(0) = 0  \big] , \]
where $c_d$ is a constant depending only on the dimension $d$. By stationarity, and since conditioning on part of a Gaussian vector reduces the variance of all coordinates, this is at most
\[  c_d \max_{j_1, j_2} \E \big[ ( H_0^{F_1,F_2} f )_{j_1,j_2}^2  \big] \le c_d \max_{j_1, j_2}  \frac{\partial^4  \kappa(0) }{\partial x_{j_1}^2 \partial x_{j_2}^2 } . \]
Finally, applying the Cauchy--Schwartz inequality in Fourier space,
 \begin{equation*}
\max_{j_1, j_2} \,  \frac{ \partial^4 \kappa(0) }{ \partial x_{j_1}^2  \partial x_{j_2}^2   } \le \max_j \frac{\partial^4  \kappa(0)  }{ \partial x_j^4 }  
 \end{equation*}
and we have the result.
\end{proof}

\begin{proof}[Proof of Corollary \ref{c:mix_kos}]
Observe that each of $f_n|_{B_1 \cup B_2}$ satisfies Condition \ref{cond:non_degeneracy}, since $B_1 \cup B_2$ does not include antipodal points. Note also that a condition analogous to \eqref{e:decay} holds; more precisely, as $n \to \infty$,
\[   \sup_{x_1 \in B_1, x_2 \in B_2}  \| (K_n(x_1, x_2) , d_{x_1} K_n (x_1, x_2), d_{x_2} K_n (x_1, x_2),   d_{x_1} \otimes d_{x_2} K_n(x_1, x_2) )\|_\infty  \to 0   , \]
which, as in the proof of Theorem \ref{t:mix_gen}, implies that, as $n \to \infty$,
\[   \sup_{x_1 \in F_1, x_2 \in F_2}   \Big| \frac{    \textup{det} ( \Delta_1(x_1, x_2)  ) }  {   \textup{det} ( \Delta_0(x_1, x_2)  ) }     - 1 \Big|   \to 0, \]
where $\Delta_t(x_1, x_2)$ is the covariance matrix, for the field $f_n$, that is defined in the proof of Theorem~\ref{t:mix_gen} (note that we have omitted the dependence on $n$ in the notation). Observe also that the scale $s_n = 1/\sqrt{n}$ at which the Kostlan ensemble converges to a local limit is a polynomial, and so all derivatives of $K_n$ on the diagonal $(x, x)$ grow at most polynomially, uniformly over $x$. Along with the discussion in Remark \ref{r:riemann}, we deduce that 
\[ \sup_{F_1 \in \calF_1, F_2 \in \calF_2 }  \sup_{x_1 \in F_1, x_2 \in F_2} c_{F_1, F_2}(x_1, x_2)   \]
grows at most polynomially as a function of $n$, where $c_{F_1, F_2}(x_1, x_2) $ is the constant appearing in Theorem \ref{t:mix_gen} applied to $f_n$ (again we omit the dependence on $n$ in the notation). Since on the other hand
\[ \sup_{x_1 \in B_1, x_2 \in B_2}  |K_n(x_1, x_2)|     \]
 decays exponentially in $n$ (recall that $B_1$ and $B_2$ are contained within an open hemisphere), we deduce the result from Theorem \ref{t:mix_gen}.
\end{proof}

\subsection{Lower concentration for topological counts}
Our proof of the lower concentration results in Corollaries \ref{c:con_euc} and \ref{c:con_kos} essentially follows the proof of \cite[Theorem 1.4]{rv_17a}. We make use of the following simple lemma:

\begin{lemma}
\label{l:chain}
Let $B_i \subset M$ be a sequence of disjoint stratified sets, and let $A_i \in \sigma_\text{top}(B_i)$ be topological events such that $\sup_i \mathbb{P}[A_i] < h < 1$. Then
\[   
| \mathbb{P}[ \cap_i A_i ] - \Pi_i \mathbb{P}[A_i] | 
\le 
\frac{1}{1-h} \, \sup_{n \in \mathbb{N} } \, \alpha_\text{top} \big( B_n, \cup_{j > n} B_j \big)  . 
 \]
\end{lemma}
\begin{proof}
By the definition of  $\alpha_\text{top}$
\[
\left|\mathbb{P}\left(A_n \cap (\cap_{j>n} A_j)\right)-\mathbb{P}(A_n)\mathbb{P}(\cap_{j>n} A_j)\right|\le 
\alpha_\text{top} \big( B_n, \cup_{j > n} B_j \big). 
\]
Iterating this inequality for $n=1,2,\dots$ and using that $\mathbb{P}(A_n)<h$ we get the upper bound
\[ 
(1 + h + h^2 + \ldots ) \times \sup_{n \in \mathbb{N} } \, \alpha_\text{top} \big( B_n, \cup_{j > n} B_j \big)  ,
\]
which is equal to the desired upped bound.
\end{proof}

\begin{proof}[Proof of Corollary \ref{c:con_euc}]
First observe that we may assume the infimum in \eqref{e:con_euc} is eventually attained in the set $r \in [g_s, s/g_s]$ for some function $g_s \to \infty$ as $s \to \infty$, since otherwise the right-hand side of \eqref{e:con_euc} is bounded from below, and we may then choose $c_1 > 0$ large enough so that \eqref{e:con_euc} holds trivially. 

Fix $\varepsilon > 0$, a function $g_s \in (0, \sqrt{s})$ such that $g_s \to \infty$, and a mesoscopic parameter $g_s < r < s / g_s$. Let $B_0$ denote the unit cube, considered as a stratified set via its collection of generalised faces of all dimensions. Consider placing $(\textup{Vol}(B) + o(1))(s/r)^d$ disjoint translations of $rB_0$ inside $sB$ 
%in such a way that each is contained within a single piece $F_i \in \calF_{sB}$
, and let $M_i$ denote these mesoscopic cubes. By the super-additivity of $N$, if $N(sB)/ s^d \le (c_N - \varepsilon) \textup{Vol}(B) $, then there exist at least $\varepsilon \textup{Vol}(B) /(2c_N - \varepsilon) \times (s/r)^d$ mesoscopic cubes $M_i$ such that $N(M_i) / r^d \le c_N - \varepsilon/2$. If this holds, then we can also find at least $\nu (s/r)^d$ mesoscopic cubes $M_i$ with this property that are separated by a distance $r$, where
\[  \nu = c_d  \times \frac{\varepsilon \textup{Vol}(B)}{2c_N  - \varepsilon}  ,\]
and $c_d > 0$ is a constant that depends only on the dimension. 

By stationarity and the law of large numbers \eqref{e:lln}, for each $h \in (0, 1)$
\[ \mathbb{P}[ N(M_i) / r^d \le c_N - \varepsilon/2 ] < h  \]
eventually as $s \to \infty$. Hence, by Lemma \ref{l:chain}, for every choice of $\nu (s/r)^d$ mesoscopic cubes $M_i$ which are $r$-separated, the probability that $N(M_i) / r^d \le c_N - \varepsilon/2$ for all of them is at most 
\[
 h^{  \nu (s/r)^d}  +   \frac{1}{1-h} \alpha_{r, s}, 
\]
where $\alpha_{r, s}=\max_n \alpha_\text{top}(M_n, \cup_{i>n}M_i)$. There are at most $2^{ ( \textup{Vol}(B) + o(1) )(s/r)^d }$ ways to choose cubes $M_i$, hence by the union bound
\[
\mathbb{P} \big[ N(sB)/ s^d \le ( c_N - \varepsilon) \textup{Vol}(B)  \big]  \le  2^{ ( \textup{Vol}(B) + o(1) )(s/r)^d  } \Big( h^{  \nu (s/r)^d}  +   \frac{1}{1-h} \alpha_{r, s}  \Big).
\]
Let $\mathcal{F}_1$ and $\mathcal{F}_2$ be the standard stratifications of $M_n$ and $\cup_{i>n}M_i$. Clearly, there is a constant $c>0$ which depends on the dimension only, such that $|\mathcal{F}_1|\le c$ and $|\mathcal{F}_2|\le c \nu (s/r)^d$. In both stratifications the strata with the largest volume are interiors of mesoscopic cubes that have volume $r^d$. By Corollary \ref{c:mix_euc} this implies that there is a constant $c_1>1$ such that
\[
\alpha_{r, s} \le c_1 r^d s^d \bar{\kappa}(r).
\]
Combining these estimates, taking $c_B=2 \log 2 \textup{Vol}(B)$ and choosing $h$ small enough, we see that for every $C > 0$ 
\[
\mathbb{P} \big[ N(sB)/ s^d \le ( c_N - \varepsilon) \textup{Vol}(B)  \big]  \le  c_1 \big(  e^{- C(s/r)^d} + e^{c_B (s/r)^d }  r^d s^d \bar{\kappa}(r) \big)  ,
\]
provided $s$ is large enough. This proves the result for $r \in [g_s, s/g_s]$. As mentioned in the very beginning of the proof, by choosing sufficiently large $c_1$ we can extend the estimate to all $r\in [0,s]$.
\end{proof}

\begin{proof}[Proof of Corollary \ref{c:con_kos}]
This follows closely the proof of Corollary \ref{c:con_euc}. We first treat the case that $B$ is contained in an open hemisphere. Equip $\mathbb{S}^d$ with a marked pole $x_0$, and for $r \in (0, 1]$ let $r B_0$ denote the symmetric spherical cap centred at $x_0$ with volume $r^d$, considered as a stratified set via the stratification $\calF_B = \{ \textup{int}(B_0), \partial B_0\}$. Fix $\varepsilon > 0$ and a function $g_n$ such that $g_n \to \infty$ as $n \to \infty$. Define a mesoscopic scale $g_n/\sqrt{n} < r < 1/g_n$ and consider placing $(\textup{Vol}(B) + o(1))/ r^d $ disjoint copies of $rB_0$ inside $B$ 
%in such a way that each is contained within a single piece $F_i \in \calF_B$
. Following exactly the proof of Corollary \ref{c:con_euc}, we deduce the existence of $c_1, c_2 > 0$ such that
\[   \mathbb{P}[ N(B)/ n^{d/2} \le (c_N - \varepsilon) \textup{Vol}(B)  ]   \le  c_1 (  e^{-c_2 r^{-d} } +   \alpha_{r, n} ) ,    \]
where $\alpha_{r, n}$ denotes the supremum of $\alpha$-mixing coefficients $\alpha_\text{top}(B_1, B_2)$ among all pairs of disjoint stratified sets $B_1$ and $B_2$ contained in $B$ and separated by a distance at least $r$. By Theorem~\ref{t:mix_gen} (see also the proof of Corollary \ref{c:mix_kos}), there exist $k, c_3, c_4 > 0$ such that
\[   \alpha_{r, n} \le c_3 n^k e^{-c_4 r^2 n} .  \]
Setting $r = n^{-1/(2+d)}$ yields the desired bound.

In the general case, we simply choose a finite number of disjoint stratified sets $B_i$
 %that are each contained within a single piece $F_i \in \calF_B$, and moreover 
 that are each contained within an open hemisphere. Since by super-additivity $N(B)/ n^{d/2} \le c_N \textup{Vol}(B) - \varepsilon $ implies that $N(B_i)/ n^{d/2} \le c_N \textup{Vol}(B_i) - \varepsilon/k $ for some $B_i$, the argument goes through in this case as well.
\end{proof}

\subsection{Decorrelation for topological counts}

Corollary \ref{c:dec_euc} is a direct consequence of the following general result, applied to the random variables $X = N(B_1)$ and $Y = N(B_2)$:

\begin{proposition}[See {\cite[Theorem 17.2.2]{il_71}}]
Let $X$ and $Y$ be random variables and define the $\alpha$-mixing coefficient associated to their $\sigma$-algebras
\[    \alpha(X, Y)  = \sup_{A \in \sigma(X), B \in \sigma(Y)} | \mathbb{P}[A \cap B] - \mathbb{P}[A] \mathbb{P}[B] | . \] 
Suppose further that
\[    \mathbb{E}[X^{2+\delta}] < c   \quad \text{and} \quad   \mathbb{E}[Y^{2+\delta}]  < c  \]
for positive constants $\delta, c > 0$.
Then
\[  | \textup{Cov}(X, Y) | \le 8 c^{2/(2+\delta)}  \alpha(X, Y)^{2/(2+\delta)} .\]
\end{proposition}

%%%%%%%%%%%%

\appendix
\section{Gaussian computations}

In this section we gather results about finite-dimensional Gaussian vectors. If $X$ is a Gaussian vector in finite-dimensional vector space equipped with a scalar product, let $\textup{DC}(X)$ be the determinant of its covariance.

\begin{lemma}\label{l:linalg_1}
Let $X,Y$ be jointly Gaussian vectors such that $Y$ is non-degenerate. Then $\textup{DC}(X|Y)$ does not depend on $Y$, and
\[\textup{DC}(X,Y)=\textup{DC}(X|Y)\textup{DC}(Y) .\]
\end{lemma}
\begin{proof}
This is an easy consequence of the Gaussian regression formula (\cite[Proposition 1.2]{azais_wschebor}).
%Let $A$ (resp.\ $B$) be the covariance of $X$ (resp.\ $Y$) so that there exists a matrix $C$ for which the covariance of $(X,Y)$ is
%\[\Lambda=\left(
%\begin{matrix}
%A & C\\
%C^T & B
%\end{matrix}
%\right) .\]
%Since $Y$ is non-degenerate, $B$ is invertible. Therefore,
%\[
%\det(\Lambda)=\det(B)\left|\begin{matrix} A & C \\ (B^{-1})^T C & Id\end{matrix}\right|=\det(B)\left[\begin{matrix} A - C(B^{-1})^T C & 0\\ B^{-1}C & Id\end{matrix}\right|=\det(B)\det(A-C(B^{-1})^T C).
%\]
%Since, by the regression formula (\cite[Proposition 1.2]{azais_wschebor}), $A-CB^{-1}\, ^tC$ is the covariance of $X$ conditioned on $Y$, we have the result.
\end{proof}

\begin{lemma}\label{l:linalg_2}
Given two independent Gaussian vectors $X$ and $Y$ of the same dimension,
\[\textup{DC}(X+Y)\geq \textup{DC}(X)+\textup{DC}(Y) .\]
\end{lemma}
\begin{proof}
In terms of covariance matrices, this amounts to saying that given $A,B$ two symmetric non-negative matrices of the same size,
\[\det(A+B)\geq\det(A)+\det(B) ,\]
which follows from the Minkowski inequality.
% but can be derived easily by codiagonalising $A$ and~$B$, expanding the determinant $\det(A+B)$ and forgetting all but the $\det(A)$ and $\det(B)$ terms using the fact that the eigenvalues are non-negative.
\end{proof}

\begin{lemma}\label{l:linalg_3}
Let $X,Y,Z$ be jointly Gaussian vectors such that $Y$ and $Z$ have the same dimension and $Z$ is independent of $(X,Y)$. Then
\[\textup{DC}(X,Y+Z)\geq \textup{DC}(X,Y)+\textup{DC}(X)\textup{DC}(Z) .\]
\end{lemma}
\begin{proof}
Let us assume that $X$ is non-degenerate; the general case follows by continuity. By Lemma \ref{l:linalg_1} 
\[\textup{DC}(X,Y+Z)=\textup{DC}(X)\textup{DC}(Y+Z|X) .\]
Applying Lemma \ref{l:linalg_2},
\[ \textup{DC}(Y+Z|X)\geq \textup{DC}(Y|X)+\textup{DC}(Z|X)=\textup{DC}(Y|X)+\textup{DC}(Z) ,\]
with the final equality since $Z$ is independent of $X$. Hence
\begin{equation*}
\textup{DC}(X,Y+Z)\geq\textup{DC}(X)\textup{DC}(Y|X)+\textup{DC}(X)\textup{DC}(Z)=\textup{DC}(X,Y)+\textup{DC}(X)\textup{DC}(Z) ,
\end{equation*}
where the equality holds by Lemma \ref{l:linalg_1}.
\end{proof}

\begin{lemma}
\label{l:det}
Let $X_1$ and $X_2$ be respectively $d_1 \times d_1$ and $d_2 \times d_2$ random matrices, and let $Y = (Y_1, Y_2) \in \R^2$ and $Z \in \R^{d_1 + d_2}$ be random vectors. Suppose that $(Y, Z)$ is a non-degenerate Gaussian vector and, conditionally on $Z = 0$, $X_1$ and $X_2$ have entries that are jointly Gaussian with $Y$. Let $\varphi_{Y, Z}$ denote the density of $(Y, Z)$. Then there exists a constant $c > 0$, depending only on $d_1$ and $d_2$, such that 
\begin{equation}
\label{e:ldet1}
 \varphi_{Y, Z}(0)  \, \mathbb{E}[  |  \textup{det}(X_1) \textup{det}(X_2)| \, |\, Y = 0 , Z = 0  ] 
 \end{equation}
is bounded above by the maximum, over $i  \in \{1, 2\}$, of
\[  \frac{c \, \E\big[ \|X_i\|_\text{op}^2 \, | \, Z = 0\big]^{d_i} }{\sqrt{ \textup{DC}(Y, Z)} }  \, \max \Big\{   1 , \Big( \frac{ \max_k \textup{Var} \big[ Y_k \big]  \, \sqrt{\textup{DC}(Z)}  }{ \sqrt{ \textup{DC}(Y, Z) } }  \Big)^{2d_i}     \Big\}  , \]
where $\| \cdot \|_{\text{op}}$ denotes the ($L^2$-)operator norm.
\end{lemma}

\begin{proof} 
Let $c$ denote a positive constant, depending only on $d_1$ and $d_2$, that may change from line to line. In the proof we use repeated the fact that conditioning on part of a Gaussian vector reduces the variance of all coordinates. By the Cauchy-Schwarz inequality and an elementary bound on the determinant, \eqref{e:ldet1} is bounded above by
\[  c \,  \varphi_{Y, Z}(0) \,  \max_{i, j_1,j_2} \mathbb{E} \big[ ( X_i)_{j_1, j_2}^{2d_i}  \, |\, Y = 0 ,Z = 0  \big]  .   \]
Since a normally distributed random variable $Z \sim \mathcal{N}(\mu, \sigma^2)$ satisfies
\[  \E[ Z^{2d_i} ] \le c \max\{ \mu^{2d_i} , \sigma^{2d_i}  \} ,\]
and since the variance of a random variable is less than its second moment,
\[ \mathbb{E} [ ( X_i)_{j_1, j_2}^{2d_i}  \, |\, Y = 0, Z = 0 ] \le c   \max \Big\{  \E \big[ ( X_i)_{j_1, j_2} ^2 | Z = 0   \big]^{d_i} \, , \,   \mathbb{E} \big[ ( X_i)_{j_1, j_2}  \, |\, Y = 0, Z=0  \big]^{2d_i}  \Big\}   . \]
Let $\Sigma_{Y|Z}$ and $\Sigma_Z$ denote the covariance matrices of $Y|Z$ and $Z$ respectively. By conditioning on $Z = 0$ and applying Lemma \ref{l:linalg_1}, we have that
\begin{align*}
  \varphi_{Y, Z}(0)   &   =  c \frac{ e^{-\frac{1}{2} \E[Y | Z = 0]^T \Sigma_{Y|Z}^{-1}  \E [Y | Z = 0] }  }{ \sqrt{ \textup{DC}(Y|Z) }} \frac{ e^{- \frac{1}{2} \E[Z]^T \Sigma_Z^{-1} \E[Z] } }{ \sqrt{ \textup{DC}(Z) } }      \\
  &  \le c \frac{ e^{-\frac{1}{2} \E[Y | Z = 0]^T \Sigma_{Y|Z}^{-1}  \E [Y | Z = 0] } }{ \sqrt{ \textup{DC}(Y, Z)  }  } \le  \frac{ c}{ \sqrt{ \textup{DC}(Y, Z)   }  } .
  \end{align*}
Since, moreover, 
\[   \max_{j_1, j_2} \E \big[ ( X_i)_{j_1, j_2} ^2 \, | \, Z = 0 \big]    \le  \E[ \|X_i\|_\text{op}^2 \, | \, Z = 0 ] , \] 
it suffices to show that
\[ \sup_{\ell \in \mathbb{R}^2} \Big\{ \mathbb{E} \big[ ( X_i)_{j_1,j_2}  \, |\, Y - E[Y|Z=0] =  \ell, Z = 0 \big]^{2d_i}  e^{-\frac{1}{2} \ell^T \Sigma_{Y|Z}^{-1} \ell }   \Big\}  \]
is bounded above by
\[ c  \, \E \big[ ( X_i)_{j_1, j_2} ^2 \, | \, Z = 0 \big]^{d_i}     \max \Big\{   1,  \Big(   \frac{    \max_k  \textup{Var} \big[ Y_k \big]  \sqrt{ \textup{DC}(Z) } }{ \sqrt{ \textup{DC}(Y, Z) }  }  \Big)^{2d_i} \Big\}  . \]
To show this, decompose $\Sigma_{Y|Z}^{-1} = U^T \Lambda^{-1} U$, where $U = (u_{k_1, k_2})$ is a $2 \times 2$ orthogonal matrix and $\Lambda = \textup{Diag}(\lambda_k)$ is the $2 \times 2$ diagonal matrix of (positive) eigenvalues of $\Sigma_{Y|Z}$. Abbreviating $W = (w_k) : =U \textup{Cov}[(X_i)_{j_1,j_2} Y_k | Z = 0]$ and replacing $\ell$ by $U \ell$, by the Gaussian regression formula (\cite[Proposition 1.2]{azais_wschebor}) we have that
\begin{align*}
&   \sup_\ell \Big\{  \mathbb{E} \big[ ( X_i)_{j_1,j_2}  \, |\, Y - \E[ Y |Z=0] = \ell, Z = 0   \big]^{2d_i}  e^{-\frac{1}{2} \ell^T \Sigma_{Y|Z}^{-1} \ell }  \Big\}  \\
& \qquad \qquad = \sup_\ell \Big\{ \big(  \E[(X_i)_{j_1,j_2} \, | \, Z = 0] + W^T \Lambda^{-1} \, \ell \big)^{2d_i}  \, e^{ -\frac{1}{2} \ell^T \Lambda^{-1} \ell }  \Big\} \\
   &\qquad \qquad \le  c  \max \Bigg\{ \sup_\ell \Big\{ \E[(X_i)_{j_1,j_2} \, | \, Z = 0]^{2d_i} \, e^{ -\frac{1}{2} \ell^T \Lambda^{-1} \ell }  \Big\} \, , \ \sup_\ell \Big\{ \big(   W^T \Lambda^{-1} \, \ell \big)^{2d_i}  \, e^{ -\frac{1}{2} \ell^T \Lambda^{-1} \ell }  \Big\} \Bigg\} \\
   &\qquad \qquad \le  c  \max \Bigg\{ \E[(X_i)^2_{j_1,j_2} \, | \, Z = 0]^{d_i}  \, , \ \sup_\ell \Big\{ \big(   W^T \Lambda^{-1} \, \ell \big)^{2d_i}  \, e^{ -\frac{1}{2} \ell^T \Lambda^{-1} \ell }  \Big\} \Bigg\}  .
   \end{align*}
Differentiating in $\ell$, the maxima of the expression on the right is attained at
\[ \ell = (\ell_1, \ell_1) =
\begin{cases}
 \frac{\pm \sqrt{2d_i}}{ \sqrt{ w_1^2 \lambda_1^{-1} + w_2^2 \lambda_2^{-1}}    }  \big( w_1, w_2 \big), & (w_1, w_2) \neq (0, 0)  ,  \\
(0, 0) , & (w_1, w_2) = (0,0) , \\
 \end{cases}
\]
and yields a maximum value of
\[   (2d_i/e)^{d_i} \, ( w_1^2  \lambda_1^{-1} + w_2^2  \lambda_2^{-1} )^{d_i}  \le c \, \Big(  \max_k w_k^2  \max_k  \lambda_k^{-1}   \Big)^{d_i}      . \]
Since the eigenvalues of a positive-definite real symmetric matrix are bounded by a constant times the maximum diagonal entry,
\[   \max_k \lambda_k^{-1} =   \frac{ \max_k \lambda_k }{ \textup{det}(\Lambda)} \le c \, \frac{ \max_k \textup{Var} \big[ Y_k \, | \,  Z = 0\big]  }{ \textup{det}(\Sigma_{Y|Z})  } \le c \, \frac{ \max_k \textup{Var} \big[ Y_k \big]  }{ \textup{det}(\Sigma_{Y|Z})  } = c \, \frac{ \max_k \textup{Var} \big[ Y_k \big]  \textup{DC}(Z)   }{ \textup{DC}(Y, Z)  }  ,   \]
where in the last step we used Lemma \ref{l:linalg_1}. Moreover, since $U$ has entries bounded above in absolute value by one (being orthogonal), and by the Cauchy-Schwarz inequality,
\[  \max_k  |w_k| \le   c \,  \max_k  |\textup{Cov}[(X_i)_{j_1,j_2} Y_k \, | \, Z = 0] | \le  c \,  \E[(X_i)^2_{j_1,j_2}  \, | \, Z = 0]^{1/2}  \, \max_k \textup{Var} \big[ Y_k \big]^{1/2}   . \]
Combining we have the result.
\end{proof}

\begin{lemma}
\label{l:int}
Let $(Y_1, Y_2)$ denote a $(d_1 + d_2)$-dimensional non-degenerate Gaussian vector. For each $t \in [0, 1]$, define $Y^t = (Y_1, t Y_2 + \sqrt{1 - t^2} \tilde{Y}_2)$
where $\tilde{Y}_2$ is a copy of $Y_2$ independent of $(Y_1,Y_2)$. Then 
\[ DC(Y^1) \le DC(Y^t) \le DC(Y^0) .\] 
\end{lemma}
\begin{proof}
Observe that $DC(Y^t)$ has the block form
   \[   \left[ \begin{array}{cc}
  A  &   tB   \\
   t B^T   & C  \\
  \end{array}  \right]  ,  \]
 where $A$ and $C$ are (strictly) positive-definite. Since $A$ is positive-definite and $B C^{-1} B^T$ is symmetric and positive-definite, there exists a $P$ such that 
 \[ A =  P^T P   \quad \text{and} \quad  B C^{-1}B^ T = P^T D P , \]
 where $D = \textup{Diag}( (d_i)_i )$ is a positive diagonal matrix. Hence
\[ DC(Y^t)  =  \textup{det}(C) \, \textup{det}(A - t^2 B C^{-1} B^T)  =  \textup{det}(C) \, \textup{det}(P)^{2} \, \Pi_i (1 - t^2 d_i)  \]
which, since $d_i > 0$, is decreasing in $t \in [0, 1]$. 
\end{proof}

\medskip
\section{Proof of Piterbarg's formula}
\label{s:Piterbarg proof}

In the proof of Piterbarg's formula, we will use the classical fact that the density function $\varphi(x; \Sigma)$ of a (possibly non-centred) Gaussian vector with covariance $\Sigma$ satisfies
\begin{equation}
\label{e:gauss}
\frac{1}{2}\frac{\partial^2}{\partial x_i^2} \varphi(x; \Sigma)  =  \frac{\partial}{\partial \Sigma_{ii} } \varphi(x; \Sigma)
\qquad \mathrm{and} \qquad    
\frac{\partial^2}{\partial x_i \partial x_j} \varphi(x; \Sigma)  =  \frac{\partial}{\partial \Sigma_{ij} } \varphi(x; \Sigma) , \ i\ne j .
\end{equation}
\begin{proof}[Proof of Lemma \ref{l:piterbarg}]
Let $(f_i)_{i\geq 1}$ and $(g_i)_{i\geq 1}$ be sequences of smooth compactly supported functions on~$\R^m$ that converge to $\un_A$ and $\un_B$ in the sense of tempered distributions. Following the proof of \cite[Theorem 1.4]{piterbarg}, by writing the derivative with respect to $t$ in terms of derivatives with respect to the elements of the covariant matrix, and then by using the identity~\eqref{e:gauss} and integrating by parts, we obtain
\[\begin{aligned}
\frac{d}{dt}\E\brb{f_i(X_t)g_i(Y_t)}
=&
\sum_{k=1}^m \int_{\R^{2m}}\partial_{x_k}f_i(x) \partial_{y_k}g_i(y) \,\gamma_t(x,y)\d x \, \d y
\\
=&
\sum_{k=1}^m \int_{\R^{2m}}f_i(x) g_i(y) \,\partial_{x_k}\partial_{y_k}\gamma_t(x,y)\d x \, \d y ;
\end{aligned}\]
(all the other terms disappear since the only covariances that depend on $t$ are $\cov(X_{t,k},Y_{t,k})=t$). Passing to the limit as $i\to \infty$ gives that
\[
\frac{d}{dt}\prob\brb{Z_t\in A\times B}
=
\frac{d}{dt}\E\brb{\un_A(X_t)\un_B(Y_t)}
=\sum_{k=1}^m
\int_{\R^{2m}}\un_A(x) \un_B(y)\partial_{x_k}\partial_{y_k}\gamma_t(x,y)\d x \, \d y.
\]
By Gauss's theorem, applied both in the $x_k$ and in the $y_k$ variables we have
\[
\int_{\R^{2m}}\un_A(x) \un_B(y)\partial_{x_k}\partial_{y_k}\gamma_t(x,y)\d x \, \d y=
\int_{\partial A\times \partial B}\nu_A(x)_k\nu_B(y)_k \gamma_t(x,y) \, \d x\,\d y
\]
where $\nu_A(x)_k$ is the $k$-th component of $\nu_A(x)$, and $\d x$ on the right-hand side of the equation is the volume element on $\partial_A$ (and similarly for $B$ and $y$). Since $\sum_{k=1}^m (\nu_A(x))_k(\nu_B(y))_k=\langle \nu_A(x),\nu_B(y)\rangle$,
\[
\frac{d}{dt}\prob\brb{Z_t\in A\times B}=\int_{\partial A\times \partial B}\langle \nu_A(x),\nu_B(y)\rangle \, \gamma_t(x,y) \d x\, \d y ,
\]
which proves the first part of the statement.

To prove the last part of the lemma let us consider $X$ to be a translation of the standard Gaussian vector by $\mu$. Let $Y$ be an independent copy of $X$. We can define $X_t=X$ and $Y_t=t(X-\mu)+\sqrt{1-t^2}(Y-\mu)+\mu$. It is easy to see that these vectors satisfy the  assumptions in the first part of the lemma.  Note that in this case $Z_0=(X,Y)$ and $Z_1=(X,X)$. Integrating with respect to $t$ from $0$ to $1$ we have 
\[
\int_0^1\int_{\partial A\times \partial B}\langle\nu_A(x),\nu_B(y)\rangle \gamma_t(x,y)\d x\, \d y \, \d t =
\prob[(X,X)\in A\times B]-\prob[(X,Y)\in A\times B].
\]
Since $\prob[(X,X)\in A\times B] =  \prob[X\in A\cap B ]$ and $\prob[(X,Y)\in A\times B] = \prob[X \in A]\prob[X\in B]$, this proves the second part of the statement.
\end{proof}

\medskip
\section{On stratified critical points}
Here we prove two elementary lemmas about stratified critical points. Recall that $M$ is a smooth manifold and $(B, \calF)$ is a stratified set of $M$.

\begin{lemma}\label{l:discriminant_is_closed}
Let $(u_k,x_k)_{k\in\N}$ be a sequence in $C^1(M)\times B$ converging to a limit $(u,x)\in C^1(M)\times B$. Assume that, for each $k\in\N$, $x_k$ is a stratified critical point of $u_k$. Then $x$ is a stratified critical point of $u$.
\end{lemma}
Lemma \ref{l:discriminant_is_closed} implies that the discriminant $\mathfrak{D}_B$ is $C^1$-closed. Moreover, taking $u_k=u$ for all~$k$, it implies that the set of stratified critical points of $u$ in $B$ is compact.

\begin{proof}[Proof of Lemma \ref{l:discriminant_is_closed}]
Without loss of generality, we may assume that there exist $F,F'\in\calF$ such that $x_k\in F$ for each $k\in\N$ and $x\in F'$. If $F'=F$ then the sequence $(u_k|_F)_k$ converges to $u|_F$ in $C^1$ so $d_xu|_F=0$ and $x$ is a stratified critical point of $u$. Otherwise, $F'<F$ and $d_xu$ vanishes on $T_xF|_{F'}$ which contains $T_xF'$, so $x$ is a critical point of $u|_{F'}$.
\end{proof}

\begin{lemma}\label{l:sndcps_are_isolated}
Let $u\in C^2(M)$ and let $x\in B$ be a non-degenerate stratified critical point of $u$. Then $x$ is isolated in the set of stratified critical points in~$B$.
\end{lemma}

Lemma \ref{l:sndcps_are_isolated} shows that Definition~\ref{d:critical_points} is the natural definition of non-degenerate critical points in the setting of stratified sets.

\begin{proof}[Proof of Lemma \ref{l:sndcps_are_isolated}]
Let $x$ be a stratified critical point of $u$ belonging to $F'\in\calF$. Assume that there exists a sequence $(x_k)_{k\in\N}$ of stratified critical points of $u$ distinct from $x$ converging to $x$; let us show that $x$ is degenerate. Without loss of generality, we may assume that there exists $F\in\calF$ such that, for each $k\in\N$, $x_k\in F$. If $F'=F$ then $x$ is a degenerate critical point of $u|_{F'}$. Otherwise, $F'<F$ and $d_xu$ vanishes on $T_xF|_{F'}$, in which case $x$ is a degenerate stratified critical point of $u$.
\end{proof}

%%%%%%%%%%%%%%
\bigskip
\bibliographystyle{plain}
\bibliography{ref_mixing}

\end{document}